\documentclass[a4paper,10pt]{amsart}

\usepackage[english]{babel}
\newcommand{\labelto}[1]{\xrightarrow{\makebox[1.5em]{\scriptsize ${#1}$}}}
\newcommand{\im}{{\mathrm{im}}}
\usepackage{amsfonts}
\usepackage{latexsym}
\usepackage{xcolor}
\usepackage[margin=1in]{geometry}
\usepackage{amsthm}
\usepackage{amssymb}
\usepackage{amsmath}
\usepackage{MnSymbol}

\usepackage{multirow}

\usepackage{makecell}

\numberwithin{equation}{section}
\usepackage{hhline}
\usepackage{longtable}
\allowdisplaybreaks[4]
\newcommand\restr[2]{{
  \left.\kern-\nulldelimiterspace 
  #1 
  \vphantom{\big|} 
  \right|_{#2} 
  }}

\usepackage{tikz}
\usetikzlibrary{arrows,decorations.markings}
\usetikzlibrary{matrix}
\usetikzlibrary{graphs}
\usetikzlibrary{backgrounds}

\usepackage{color}

\newcommand{\Z}{\mathbb{Z}}
\newcommand{\C}{\mathbb{C}}
\newcommand{\R}{\mathbb{R}}

\newcommand{\GL}{\mathrm{GL}}
\newcommand{\SL}{\mathrm{SL}}

\newcommand{\Lie}{\mathrm{Lie}}

\newcommand{\mf}{\mathfrak}
\newcommand{\g}{\mf{g}}
\newcommand{\h}{\mf{h}}

\newcommand{\so}{\mf{so}}

\newcommand{\ssl}{\mf{sl}}

\newcommand{\s}{\mf{s}}

\renewcommand{\c}{\mf{c}}

\newcommand{\z}{\mf{z}}

\newcommand{\diag}{\mathrm{diag}}

\newcommand{\ad}{\mathrm{ad}}
\newcommand{\G}{\widehat{G}}
\newcommand{\cG}{\mathcal{G}}
\newcommand{\cN}{\mathcal{N}}
\newcommand{\cC}{\mathcal{C}}
\newcommand{\A}{\mathcal{A}}

\renewcommand{\O}{\mathcal{O}}

\newcommand{\Aut}{\mathrm{Aut}}

\newcommand{\wG}{\widehat{G}}
\newcommand{\TT}{\mathcal{T}}

\numberwithin{equation}{section}
\newtheorem{theorem}{Theorem}[section]
\newtheorem{proposition}[theorem]{Proposition}

\newtheorem{lemma}[theorem]{Lemma}

\theoremstyle{remark}
 
\theoremstyle{remark}
\newtheorem{rmk}[theorem]{Remark}
\newtheorem{example}[theorem]{Example}   

\usepackage[all]{xy}

\title[]{Classification of nilpotent and semisimple fourvectors of a real eight-dimensional space}
\author{Emanuele Di Bella}
\address{
Dipartimento di Matematica\\
Universit\`{a} di Trento\\
Italy}
\email{emanuele.dibella@unitn.it}
\author{Willem A. de Graaf}
\address{
Dipartimento di Matematica\\
Universit\`{a} di Trento\\
Italy}
\email{willem.degraaf@unitn.it}
\author{Andrea Santi}
\address{Dipartimento di Matematica\\
Universit\`a di Roma Tor Vergata\\ 
Italy}
\email{santi@mat.uniroma2.it}

\date{}

\begin{document}


\def\DynkinNodeSize{1.5mm}
\def\DynkinArrowLength{1.5mm}
\tikzset{
dnode/.style={
circle,
inner sep=0pt,
minimum size=\DynkinNodeSize,
fill=white,
draw},
middlearrow/.style={
decoration={markings,
mark=at position 0.6 with
{\draw (0:0mm) -- +(+135:\DynkinArrowLength); \draw (0:0mm) -- +(-135:\DynkinArrowLength);},
},
postaction={decorate}
},
leftrightarrow/.style={
decoration={markings,
mark=at position 0.999 with
{
\draw (0:0mm) -- +(+135:\DynkinArrowLength); \draw (0:0mm) -- +(-135:\DynkinArrowLength);
},
mark=at position 0.001 with
{
\draw (0:0mm) -- +(+45:\DynkinArrowLength); \draw (0:0mm) -- +(-45:\DynkinArrowLength);
},
},
postaction={decorate}
},
sedge/.style={
},
dedge/.style={
middlearrow,
double distance=0.5mm,
},
tedge/.style={
middlearrow,
double distance=1.0mm+\pgflinewidth,
postaction={draw}, 
},
infedge/.style={
leftrightarrow,
double distance=0.5mm,
},
}


\begin{abstract}
In 1981 Antonyan classified the orbits of $\SL(8,\C)$ on $\bigwedge^4 \C^8$. This is an example
of a $\theta$-group action as introduced and studied by Vinberg. The orbits of a $\theta$-group
are divided into three classes: nilpotent, semisimple and mixed. We consider the 
action of $\SL(8,\R)$ on $\bigwedge^4 \R^8$ and classify the nilpotent and semisimple orbits as well
as the Cartan subspaces. The semisimple orbits are divided into 1452 parametrized classes.
Due to this high number a classification of the mixed orbits does not seem feasible. There are
10 Cartan subspaces. Our methods are based on Galois cohomology. 
\end{abstract}

\maketitle

\section{Introduction}

The elements of $\bigwedge^4 \R^8$ are called 4-vectors of a real 8-dimensional space. This 
paper is concerned with the orbits of $\SL(8,\R)$ in the space of 4-vectors.
These $4$-vectors (or, dually, $4$-forms) play an important role both in differential geometry
and mathematical physics. For instance, they arise naturally from Kähler, hyperKähler, 
quaternion-Kähler, and Spin(7) geometries (in this latter case, the $4$-form is constructed 
from the natural invariant $4$-form and its Hodge dual in $7$-dimensional space, which are 
invariants for the action of the group G(2)). See also \cite{ACD}, where $4$-forms are used to
define and investigate generalizations of the notion of Yang-Mills self-duality in arbitrary 
dimensions. We also refer to our later Remark \ref{rem:geom}, where the representatives of real 
semisimple $4$-vectors with compact semisimple stabilizer are described.
Finally, the bosonic sector of supergravity in $11$-dimensions (arguably the most relevant
among the various supergravity theories) includes a $4$-form field strength in addition to the
metric graviton. In the recent \cite{DGS2} we consider the {\it supersymmery gap problem} for 
$D=11$ supergravity backgrounds and establish a rigidity result in terms of the rank of $4$-
forms, which is amenable to further extensions using the classification carried out in this 
paper.

Antonyan classified the orbits of the group $\SL(8,\C)$ acting on the space $\bigwedge^4 \C^8$ (\cite{Antorig}, see \cite{Antotrad} for a translation). 
For this he used Vinberg's theory of $\theta$-groups (\cite{vinberg}). Here we briefly summarize this, and give more details in the next section. 
The starting point is a $\Z/2\Z$-grading $\g=\g_0\oplus \g_1$ of the simple complex
Lie algebra $\g$ of type $E_7$. Here $\g_0\cong \ssl(8,\C)$ and $\g_1\cong \bigwedge^4\C^8$ (as $\ssl(8,\C)$-module). Let $G$ be the adjoint group of $\g$ and
$G_0$ the connected subgroup corresponding to the subalgebra $\g_0$. There is a surjective morphism $\SL(8,\C)\to G_0$ by which $\g_1$ becomes a $\SL(8,\C)$-module
isomorphic to $\bigwedge^4 \C^8$. A first remark is that $\g_1$ is closed under the Jordan decomposition of $\g$. This divides the elements of $\bigwedge^4 \C^8$, and hence the $\SL(8,\C)$-orbits,
into three groups: nilpotent, semisimple and mixed. Antonyan showed that there are 94 nonzero nilpotent orbits. There are infinitely many semisimple orbits. However, these can be divided into
32 classes, including the class consisting just of 0. The 31 nonzero classes contain an infinite number of orbits and their representatives depend on up to seven parameters. 
The representatives in a fixed class have the {\em same} stabilizer in $\SL(8,\C)$. Finally, for each semisimple class $C$ it is possible to find a finite list $L$ of nilpotent
elements such that the elements $p+e$ ($p\in C$, $e\in L$) are representatives of the orbits of mixed type. In \cite{Antorig} the list $L$ is given for 7 of the 31 nonzero
classes of semisimple elements. 

In this paper we consider the classification of the $\SL(8,\R)$-orbits in the space $\bigwedge^4\R^8$. 
We use the methods developed in \cite{bgl,gl24}. Our main 
technical tool is Galois cohomology. We use the algorithms of \cite{borwdg} to compute the necessary Galois cohomology sets. In order to use these algorithms 
we need to determine the components of the stabilizers of the various elements that we are interested in. The determination of these components 
is the main part of the work that went into this paper. Our methods for tackling this question are described in Sections \ref{sec:nilp}, \ref{sec2}.
Based on our results we can formulate the following theorem.

\begin{theorem}
In $\bigwedge^4 \R^8$ there are 258 (nonzero) nilpotent $\SL(8,\R)$-orbits. Including the zero class there are 1452 classes of semisimple $\SL(8,\R)$-orbits. The nonzero semisimple classes depend on
up to seven parameters. The representatives of the orbits in a fixed class have the same stabilizer in $\SL(8,\R)$. 
\end{theorem}

In order to classify the mixed orbits as well one would need to classify the elements that can be the nilpotent part of a mixed element with semisimple part in one of the 1451 nonzero
semisimple classes. With our current methods this is a huge, if not impossible, task. 

The paper is organized as follows. The next section describes the basic setup as well as the notation and recalls a number of results from the literature.
Section \ref{sec:tab} has two tables: one for the nilpotent orbits, and one with data relative to the semisimple orbits. In Section \ref{sec1} we give some background material
concerning Galois cohomology. Then in Sections \ref{sec:nilp}, \ref{sec2} we go into the classification of the nilpotent and semisimple orbits. In the last section we show
how we classified the Cartan subspaces in $\bigwedge^4 \R^8$. This can be seen as a very rough classification of the semisimple orbits: each semisimple orbit has a point in
(at least) one of the ten given Cartan subspaces. The appendix contains descriptions of the
classes of semisimple elements in which the complex semisimple orbits are divided.

The computations of the components of the stabilizers were mostly performed with the help of the computer algebra system {\sf GAP}4 (\cite{gap4}). In many cases
it sufficed to compute a Gr\"obner basis of an ideal to determine the components. We refer to \cite{clo} for an introduction into the theory of Gr\"obner bases. 
On many occasions we used the computer algebra system {\sc Magma} (\cite{magma}) for computing and working with Gr\"obner bases. The Galois cohomology sets were computed in {\sf GAP}4 
using the implementation of the algorithms of \cite{borwdg} in that system. 

{\bf Acknowledgements:} We thank Simon Salamon for very helpful discussions, and for communicating the content of Remark \ref{rem:geom} to us. 

The third author acknowledges the MIUR Excellence Department Project MatMod@TOV,
which has been awarded to the Department of Mathematics, University of Rome Tor Vergata,
CUP E83C23000330006. This article/publication was also supported by the “National Group
for Algebraic and Geometric Structures, and their Applications” GNSAGA-INdAM (Italy)
and it is based upon work from COST Action CaLISTA CA21109 supported by COST
(European Cooperation in Science and Technology), {\tt https://www.cost.eu}.

\section{Preliminaries}\label{sec:prelim}

\subsection{On Vinbergs' $\theta$-groups}\label{sec:theta}
In \cite{vinberg} Vinberg introduced and studied a class of representations of reductive algebraic groups that since then have become known
as $\theta$-groups in the literature. Their construction starts by considering a semisimple
complex Lie algebra $\g$ and a $\Z/m\Z$-grading 
\begin{equation}\label{eq:grad}
\g = \bigoplus_{i\in \Z/m\Z} \g_i \text{ with } [\g_i,\g_j]\subset \g_{i+j} \text{ for } i,j\in
\Z/m\Z.
\end{equation}

Such gradings are closely related to automorphisms of $\g$ of order $m$. Starting with a grading \eqref{eq:grad}
and a primitive $m$-th root of unity $\omega\in\C$ we define $\theta : \g\to \g$ by $\theta(x) = \omega^ix$ for $x\in \g_i$, and extend $\theta$ to $\g$ by linearity.
Then $\theta$ is an automorphism of order $m$. Conversely, if $\theta$ is an automorphism of order $m$ then we let $\g_i$ be the eigenspace of $\theta$ with eigenvalue
$\omega^i$ and we obtain a grading of the form \eqref{eq:grad}. 

Let a grading of the form \eqref{eq:grad} be given.
Let $G$ be the inner automorphism group of $\g$ (also called the adjoint group). Another way to characterize this group is to say that it is the identity component of the automorphism group of $\g$.
Its Lie algebra is $\ad \g = \{ \ad x \mid x\in \g\}$ where $\ad x : \g\to \g$ is the adjoint map, $\ad x(y) = [x,y]$. It is known that $\g_0$ is a reductive
subalgebra (see \cite[Lemma 8.1]{kac}). Let $G_0\subset G$ be the connected subgroup whose Lie algebra is $\ad \g_0\subset \ad \g$. Then $G_0$ is a reductive 
algebraic group acting naturally on $\g_1$. The group $G_0$ together with its action on $\g_1$ is called a $\theta$-group. It is an important feature 
that by results of Vinberg (\cite{vinberg,vinberg2}) and Vinberg-Elashvili (\cite{elashvin}) it is possible to classify the orbits of these groups. 

In this paper we are concerned with one example of a $\theta$-group. It is constructed from a $\Z/2\Z$-grading of the simple Lie algebra $\g$ of type $E_7$. 
We number the nodes of the extended Dynkin diagram of $\g$ as follows:

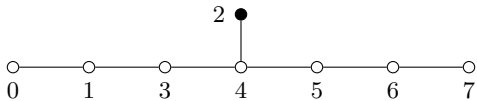
\begin{center}
\begin{tikzpicture} 
\node[dnode,label=below:{\small $0$}] (7) at (0,0) {};  
\node[dnode,label=below:{\small $1$}] (1) at (1,0) {};
\node[dnode,label=left:{\small $2$},fill=black] (2) at (3,0.7) {};
\node[dnode,label=below:{\small $3$}] (3) at (2,0) {};
\node[dnode,label=below:{\small $4$}] (4) at (3,0) {};
\node[dnode,label=below:{\small $5$}] (5) at (4,0) {};
\node[dnode,label=below:{\small $6$}] (6) at (5,0) {};
\node[dnode,label=below:{\small $7$}] (8) at (6,0) {};
\path (1) edge[sedge] (3)
(3) edge[sedge] (4)
(7) edge[sedge] (1)
(4) edge[sedge] (5)
(4) edge[sedge] (2)
(5) edge[sedge] (6)
(6) edge[sedge] (8)
;
\end{tikzpicture}
\end{center}

Here the node labeled 0 corresponds to the lowest root $\alpha_0$ of the root system. The remaining nodes correspond to the simple positive roots
$\alpha_1,\ldots,\alpha_7$. For $0\leq i\leq 7$ we fix a nonzero root vector $e_i$ in $\g_{\alpha_i}$. Then $e_0,\ldots, e_7$ generate $\g$.
By \cite[Theorem 8.6]{kac} sending $e_i\mapsto e_i$ for $i\neq 2$ and $e_2\mapsto -e_2$ uniquely defines an automorphism $\theta$ of $\g$ of order 2.
By \cite[Proposition 8.6]{kac} the subalgebra $\g_0$ is simple of type $A_7$. Hence we have an isomorphism $\psi : \ssl(8,\C)\to \g_0$. Using this isomorphism
we can make $\g_1$ into a $\ssl(8,\C)$-module. Moreover, as $\SL(8,\C)$ is simply connected, the isomorphism lifts to a surjective homomorphism $\psi: \SL(8,\C)\to G_0$.
We can use this to make $\g_1$ into a $\SL(8,\C)$-module. Again using \cite[Proposition 8.6]{kac} we see that the $\SL(8,\C)$-modules $\g_1$ and $\bigwedge^4\C^8$ are 
isomorphic. In the sequel we identify these two modules and we give their elements as linear combinations of the standard basis of $\bigwedge^4 \C^8$. 
For the latter, if $v_1,\ldots,v_8$ denotes the standard basis of $\C^8$, we write 
$$e_{ijkl} = v_i\wedge v_j\wedge v_k\wedge v_l.$$

\subsection{Orbits of a $\theta$-group}

Again consider a grading as in \eqref{eq:grad}. A first observation is that the space $\g_1$ is closed under Jordan decomposition. This means the following.
Let $x\in \g_1$ and write $x=s+n$ where $s$ is semisimple, $n$ is nilpotent (that is, $\ad s$, $\ad n$ are semisimple, respectively nilpotent) and $[s,n]=0$
(cf. \cite[\S 5.4]{hum}). Then $\theta(x)=\theta(s)+\theta(n)$ is the Jordan decomposition of $\theta(x)$, but also $\theta(x) = \omega x = \omega s+\omega n$,
implying that $s,n\in \g_1$. So the elements of $\g_1$ (and hence the $G_0$-orbits) are divided into three groups: semisimple (the elements with $n=0$), nilpotent
(when $s=0$) and mixed (when both $s,n$ are nonzero). Here we briefly comment on the methods to classify the nilpotent and semisimple orbits.

There are finitely many nilpotent orbits and there are several algorithms to classify them (\cite{gra16,vinberg2}). One of the main points here is that a 
nilpotent $e\in \g_1$ lies in a homogeneous $\ssl_2$-triple, that is there are $h\in \g_0$, $f\in \g_{-1}$ such that $(h,e,f)$ is an $\ssl_2$-triple, meaning
$$[h,e]=2e,\, [h,f]=-2f,\, [e,f]=h.$$
Let $e,e'\in\g_1$ be nilpotent lying in homogeneous $\ssl_2$-triples $(h,e,f)$, $(h',e',f')$. Then $e,e'$ are $G_0$-conjugate if and only if there is a $g\in G_0$ with
$g(h)=h'$, $g(e)=e'$, $g(f)=f'$ (cf. \cite[Theorem 8.3.6]{gra16}). In our main example, there are ninety-four nonzero nilpotent orbits, see \cite{Antorig}, \cite{Antotrad}.

For the semisimple orbits the key notion is that of a {\em Cartan subspace}: this is a maximal subspace of $\g_1$ consisting of commuting semisimple elements.
Vinberg has shown that two Cartan subspaces in $\g_1$ are $G_0$-conjugate (\cite[Theorem 1]{vinberg}). Hence every semisimple orbit has a point in a fixed 
Cartan subspace $\h$. Furthermore, consider the following groups
\begin{align*}
Z_{G_0}(\h) &= \{ g\in G_0 \mid g(x) = x \text{ for all } x\in \h\}\\
N_{G_0}(\h) &= \{ g\in G_0\mid g(x)\in \h \text{ for all } x\in \h\}
\end{align*}
and $W_0(\h) = N_{G_0}(\h)/Z_{G_0}(\h)$. Then $W_0(\h)$ is a finite group of linear transformations of $\h$, called the {\em little Weyl group} of the graded
Lie algebra $\g$. Furthermore, two elements of $\h$ are $G_0$-conjugate if and only if they are conjugate under $W_0(\h)$ (\cite[Theorem 2]{vinberg}). 
So a set of representatives of the $W_0(\h)$-orbits on $\h$ is also a set of representatives of the semisimple $G_0$-orbits. However, this can be refined.

Write $W=W_0(\h)$ and for $p\in \h$ let $W_p = \{ w\in W\mid w(p)=p\}$ be its stabilizer. For a subgroup $H$ of $W$ define
\begin{align*}
\h_H &= \{ p\in \h \mid H\subset W_p\}\\
\h_H^\circ &= \{ p\in \h \mid H=W_p\}.
\end{align*}

Then $\h_H$ is the intersection of the subspaces of $\h$ defined by the equations $(w-1)p=0$ for $w\in H$; so it is a subspace of $\h$. Furthermore, 
$\h_H^\circ$ is the subset of $\h_H$ defined by the inequalities $(w-1)p\neq 0$ for $w\in W\setminus H$; so it is a Zariski open subset of $\h_H$.
Also note that $p\in \h$ lies in $\h_H^\circ$ where $H=W_p$. So the sets $\h_H^\circ$ for $H$ a subgroup of $W$ cover all of $\h$.

For subgroups $H,H'\subset W$ and $w\in W$ we have (see \cite[(3.3)]{gl24})
\begin{equation}\label{eq:wmap}
w \h_H^\circ = \h_{H'}^\circ \text{ if and only if } wHw^{-1} = H'.
\end{equation}
We say that a subgroup $H\subset W$ is a stabilizer if there is a $p\in \h$ with $H=W_p$.
Let $H_1,\ldots,H_r$ be the subgroups of $W$, up to conjugacy, that are stabilizers. Then it follows that a semisimple $G_0$-orbit in $\g_1$ 
has a point in exactly one of the $\h_{H_i}^\circ$. Furthermore, it follows that two elements of $\h_{H_i}^\circ$ are $G_0$-conjugate
if and only if they are conjugate under $\Gamma_{H_i} = N_W(H_i)/H_i$. We conclude that the union of the sets of representatives of the 
$\Gamma_{H_i}$-orbits on $\h_{H_i}^\circ$ is also a set of representatives of the semisimple $G_0$-orbits in $\g_1$.
Furthermore, we have the following theorem (\cite[Lemma 2.9]{dgmo22}, \cite[Corollary 3.13]{gl24}).

\begin{theorem}\label{thm:cent}
Let $x,x'\in \h_{H_i}^\circ$. Then $Z_{G_0}(x) = Z_{G_0}(x')$.
\end{theorem}

In our main example a Cartan subspace of $\g_1$ has dimension 7. A particular Cartan subspace $\h$, found by Antonyan (\cite{Antorig}) is spanned by
\begin{align*}
p_1 &= e_{1234}+e_{5678}\\
p_2 &= e_{1357}+e_{2468}\\
p_3 &= e_{1256}+e_{3478}\\
p_4 &= e_{1368}+e_{2457}\\
p_5 &= e_{1458}+e_{2367}\\
p_6 &= -e_{1467}-e_{2358}\\
p_7 &= -e_{1278}-e_{3456}.
\end{align*}    

(Here we could of course erase the minuses in the last two elements; however we preferred to stick to the exact basis given by Antonyan.)

So in this case $\h$ is also a Cartan subalgebra of $\g$. 
The little Weyl group, $W=W_0(\h)$, is equal to the Weyl group of the root system of $\g$ with respect to $\h$ (see \cite[Lemma 2.6]{dgmo22}). 
It has thirty-two conjugacy classes of subgroups that are stabilizers of elements.

\subsection{Real forms}\label{sec:realf}

Throughout this paper we write $\hat \g_0 = \ssl(8,\C)$ and 
$\G_0=\SL(8,\C)$. The construction of the $\G_0$-module $\bigwedge^4\C^8$ using a
$\Z/2\Z$-grading of the simple Lie algebra $\g$ of type $E_7$ has been used by Antonyan
(\cite{Antorig}, see also \cite{Antotrad}) to classify the $\G_0$-orbits in this module. 
Also we write $\G_0(\R)=\SL(8,\R)$. In this paper we classify the nilpotent and semisimple $\G_0(\R)$-orbits on 
$\bigwedge^4\R^8$. For this reason we describe the real forms of the various
objects that we use.

We consider a fixed Chevalley basis of $\g$ (see \cite[Theorem 25.2]{hum})
and let the generators $e_0,\ldots,e_7$ of Section \ref{sec:theta} be elements of that basis.
Let $f_0,\ldots,f_7$ be the elements of the Chevalley basis lying in the root spaces of 
$-\alpha_0,\ldots,-\alpha_7$. 
roots. Denote the $8\times 8$-matrix with a 1 on position $(i,j)$ and zeros elsewhere by
$e_{ij}$. Then mapping $e_{12}, e_{23},\ldots,e_{78}$ to, respectively, $e_0,e_1,e_3,e_4,
\ldots,e_7$ and $e_{21},e_{32},\ldots,e_{87}$ to respectively $f_0,f_1,f_3,f_4,
\ldots,f_7$ extends to a unique isomorphism $\psi : \hat\g_0\to \g_0$.
This isomorphism lifts to a unique surjective homomorphism $\psi :
\G_0\to G_0$. 
We let $\g^\R$ be the real span of the fixed Chevalley basis, and let 
$\g_0^\R$, $\g_1^\R$ be the intersections of $\g^\R$ with $\g_0$, $\g_1$
respectively. Then $\psi$ restricts to an isomorphism $\psi : \ssl(8,\R)
\to \g_0^\R$. We let $G(\R)$ be the set of elements of $G$ that map $\g^\R$ to itself
(recall that $G$ is the inner automorphism group of $\g$), and $G_0(\R)$ be the 
intersection of $G_0$ and $G(\R)$. 
We also consider the restriction of the group homomorphism
$\psi : \G_0(\R)\to G_0(\R)$. We remark that this map is not necssarily
surjective as $G_0(\R)$ could be non-connected.
However, the map $\psi$ makes $\g_1^\R$ into a $\G_0(\R)$-module isomorphic
to $\bigwedge^4 \R^8$. 

We have a canonical conjugation of all complex objects involved. In all cases
we denote this map by $\sigma$. On the real spaces $\g^\R$, $\g_0^\R$, $\g_1^\R$,
$\bigwedge^4\R^8$ it is the identity, and it is extended to the corresponding
complex spaces by sending complex coefficients to their complex conjugates.
On $\G_0$ it is defined by sending the matrix entries of an element to
their complex conjugates. On $G_0$ it is defined by considering the matrix of
a $g\in G_0$ with respect to a basis of $\g^\R$ and sending its matrix
coefficients to their complex conjugates.

\subsection{Two module automorphisms}\label{sec:outer}

Here we define two automorphisms of the $\SL(8,\C)$-module $\bigwedge^4 \C^8$.
They are defined over $\R$ so that they restrict to automorphisms of the
$\SL(8,\R)$-module $\bigwedge^4 \R^8$.

For the first map let
$$g_0 = \diag(-1,1,1,1,1,1,1,1) \text{ and } g_1=\diag(-\omega,\omega,\omega,
\omega,\omega,\omega,\omega,\omega)$$
where $\omega$ is a primitive $16$-th root of unity with $\omega^4=i$.
We define $\nu : \SL(8,\C)\to \SL(8,\C)$ by $\nu(g) = g_0gg_0^{-1}$. Then $\nu$ is
an automorphism of $\SL(8,\C)$ that restricts to an automorphism of
$\SL(8,\R)$. Obviously $g_0$ can be seen as a linear map
$\C^8\to \C^8$ and we define the map $\nu : \bigwedge^4 \C^8\to \bigwedge^4\C^8$
by $\nu(u_1\wedge u_2\wedge u_3\wedge u_4) = (g_0u_1)\wedge (g_0u_2)\wedge
(g_0u_3)\wedge (g_0u_4)$. Then $\nu$ restrictes to a linear map of
$\bigwedge^4 \R^8$. For $g\in \SL(8,\C)$ and $v\in \bigwedge^4 \C^8$ we have
$\nu(g\cdot v ) =\nu(g)\cdot \nu(v)$. So $\nu$ maps the $\SL(8,\C)$-orbit
of $v$ to the $\SL(8,\C)$-orbit of $\nu(v)$. If $u\in \bigwedge^4 \R^8$ then
the same holds with $\SL(8,\C)$ replaced by $\SL(8,\R)$.

Define $\nu_1$ in the same way as $\nu$ but with $g_0$ replaced by $g_1$.
Then $\nu_1$ has the same properties, but it obviously does not restrict to
maps of $\SL(8,\R)$ and $\bigwedge^4\R^8$. Note that $g_1\in \SL(8,\C)$ so
$\nu_1$ maps the $\SL(8,\C)$-orbit of $v\in \bigwedge^4 \C^8$ to itself.
We have that $\nu_1\nu$ is the identity on $\SL(8,\C)$ and maps $v\mapsto
iv$ for $v\in \bigwedge^4\C^8$. It follows that the $\SL(8,\C)$-orbit of
$\nu(v)$ is the $\SL(8,\C)$-orbit of $iv$. So also $\nu$ maps the
$\SL(8,\R)$-orbits contained in $\SL(8,\C)\cdot v\cap \bigwedge^4\R^8$ to
the $\SL(8,\R)$-orbits contained in  $\SL(8,\C)\cdot iv\cap \bigwedge^4\R^8$.

For the second map we consider the centralizer $G^\theta$ of $\theta$ in $G$.
Some calculations show that $G^\theta = G_0 \cup \varphi G_0$. We have that
$\varphi$ is of order 2 and restricts to an outer automorphism of $\g_0$.
It lifts to an automorphism of $\SL(8,\C)$ and for $g\in \SL(8,\C)$ we have
$$\varphi(g) = Q (g^{-1})^t Q^{-1}$$ with
$$Q=\mathrm{antidiag}(1,-1,1,-1,1,-1,1,-1)$$
(the anti-diagonal goes from the top right to the bottom left). Furthermore,
$\varphi$ stabilizes $\g_1$ and hence can be seen as a map $\bigwedge^4 \C^8 \to
\bigwedge^4\C^8$. Since it is of order 2 we have that $\bigwedge^4\C^8$ is the
direct sum of the eigenspaces with eigenvalue 1 and -1 respectively.

A basis of the eigenspace with eigenvalue 1 is as follows
\begin{align*}
& e_{ 1 2 3 4 },\, e_{ 1 2 3 5 },\, e_{ 1 2 4 6 },\, e_{ 1 2 5 6 },\, 
e_{ 1 3 4 7 },\, e_{ 1 3 5 7 },\, e_{ 1 4 6 7 },\, e_{ 1 5 6 7 },\, \\
& e_{ 2 3 4 8 },\, e_{ 2 3 5 8 },\, e_{ 2 4 6 8 },\, e_{ 2 5 6 8 },\, 
e_{ 3 4 7 8 },\, e_{ 3 5 7 8 },\, e_{ 4 6 7 8 },\, e_{ 5 6 7 8 }\\
&e_{ 1 2 4 5 }+ e_{ 1 2 3 6 } ,\, e_{ 1 3 4 5 }+ e_{ 1 2 3 7 } ,\, 
  e_{ 1 3 4 6 }+ e_{ 1 2 4 7 } ,\, e_{ 1 3 5 6 }+ e_{ 1 2 5 7 } \\
&  e_{ 1 4 5 6 }+ e_{ 1 2 6 7 } ,\, e_{ 1 4 5 7 }+ e_{ 1 3 6 7 } ,\, 
  e_{ 2 3 4 5 }+ e_{ 1 2 3 8 } ,\, e_{ 2 3 4 6 }+ e_{ 1 2 4 8 } \\ 
&  e_{ 2 3 4 7 }+ e_{ 1 3 4 8 } ,\, e_{ 2 3 5 6 }+ e_{ 1 2 5 8 } ,\, 
  e_{ 2 3 5 7 }+ e_{ 1 3 5 8 } ,\, e_{ 2 3 6 7 }+ e_{ 1 4 5 8 } \\ 
&  e_{ 2 4 5 6 }+ e_{ 1 2 6 8 } ,\, e_{ 2 4 5 7 }+ e_{ 1 3 6 8 } ,\, 
  e_{ 2 4 5 8 }+ e_{ 2 3 6 8 } ,\, e_{ 2 4 6 7 }+ e_{ 1 4 6 8 } \\ 
&  e_{ 2 5 6 7 }+ e_{ 1 5 6 8 } ,\, e_{ 3 4 5 6 }+ e_{ 1 2 7 8 } ,\, 
  e_{ 3 4 5 7 }+ e_{ 1 3 7 8 } ,\, e_{ 3 4 5 8 }+ e_{ 2 3 7 8 } \\ 
&  e_{ 3 4 6 7 }+ e_{ 1 4 7 8 } ,\, e_{ 3 4 6 8 }+ e_{ 2 4 7 8 } ,\, 
  e_{ 3 5 6 7 }+ e_{ 1 5 7 8 } ,\, e_{ 3 5 6 8 }+ e_{ 2 5 7 8 } \\ 
&  e_{ 4 5 6 7 }+ e_{ 1 6 7 8 } ,\, e_{ 4 5 6 8 }+ e_{ 2 6 7 8 } ,\, 
  e_{ 4 5 7 8 }+ e_{ 3 6 7 8 }.
\end{align*}

For the elements listed here of the form $e_{ijkl}+e_{pqrs}$ we have
$\varphi(e_{ijkl}) = e_{pqrs}$ and vice versa. So with the above list we
can determine the image of $\varphi$ on every element of $\bigwedge^4 \C^8$.

It is clear that $\varphi$ restricts to automorphisms of $\SL(8,\R)$ and
$\bigwedge^4 \R^8$. As $\varphi$ is an automorphism of $\g$ we also have
$\varphi(g\cdot v) = \varphi(g)\cdot \varphi(v)$ for $g$ in $\SL(8,\C)$ or
$\SL(8,\R)$ and $v$ in $\bigwedge^4 \C^8$ or in $\bigwedge^4 \R^8$.
So it maps $\SL(8,\C)$-orbits to $\SL(8,\C)$-orbits and $\SL(8,\R)$-orbits
to $\SL(8,\R)$-orbits.

A quick check shows that 
for the basis vectors $p_1,\ldots,p_7$ of the Cartan subspace $\h$ we have
$\varphi(p_i) = p_i$ for all $i$. Hence $\varphi$ maps a semisimple
$\SL(8,\C)$-orbit to itself. (But it could of course permute the
real orbits contained in the same complex semisimple orbits.)

Some computations show that the group generated by $\nu$ and $\varphi$ has the following
eight elements
$$1,\, \varphi,\, \nu,\, \varphi\nu,\, \nu\varphi,\, \varphi\nu\varphi,\, \nu\varphi\nu,\, \nu\varphi\nu\varphi
$$
with the relation $\nu\varphi\nu\varphi=\varphi\nu\varphi\nu$ (which makes it isomorphic to the
Weyl group of the root system of type $B_2$).

\subsection{Component groups}\label{sec:cmp}

In order to determine the real orbits, given the complex orbits, we use Galois cohomology (see Section \ref{sec1}) of the stabilizers
of the complex orbits. To compute the Galois cohomology set of a stabilizer (with the algorithms of \cite{borwdg}) we need its components.
By this we mean the following. Let $p\in \g_1$ and consider the stabilizer
$$Z_{\G_0}(p) = \{ g\in \G_0 \mid g\cdot p = p\}.$$
Then $Z_{\G_0}(p)$ is an algebraic subgroup of $\G_0$. Let $Z_{\G_0}(p)^\circ$ be the connected component of the identity. 
There are $g_1,\ldots,g_r\in \G_0$ with $g_1=1$ such that $Z_{\G_0}(p)$ is the disjoint union of the sets $g_iZ_{\G_0}(p)^\circ$. 
By determining the components of $Z_{\G_0}(p)$ we mean determining such $g_1,\ldots,g_r$. The identity component $Z_{\G_0}(p)^0$ is
determined by its Lie algebra $\z_{\hat\g_0}(p) = \{ x\in \hat\g_0 \mid x\cdot p = 0\}$. This Lie algebra can be readily determined
by linear algebra techniques. The {\em component group} is the quotient
$Z_{\G_0}(p)/Z_{\G_0}(p)^\circ$; we see that the $g_i$ are representatives of the different cosets in the component group.

We will discuss below how exactly we determine the components of a stabilizer. In many cases we obtain a set of elements 
that could be redundant: it could be possible that different elements of this set lie in the same component. We can detect this situation by an algorithm described in \cite{zarclos}. In the present paper we call
this algorithm {\sf IsEltOf}. The input to this algorithm is a basis
$\mathcal{B}$ of the Lie algebra of a connected algebraic subgroup $A$ of
$\GL(n,\C)$ and an element $g\in \GL(n,\C)$. The output of {\sf IsEltOf}(
$\mathcal{B}$, $g$) is {\sf true} if $g\in A$ and {\sf false} otherwise.
For our purposes, on most occasions, the Lie algebra of the relevant algebraic
subgroup is $\z_{\hat\g_0}(p)$. 

On quite a few occasions a detailed description of the automorphism
group of a complex semisimple Lie algebra $\s$ helps to determine the components of 
a stabilizer. We recall this description here.

We consider the root system of $\s$ (with respect to a fixed Cartan subalgebra)
and fix a set of simple roots. Then there are root vectors $e_1,\ldots,e_r$
(corresponding to the positive simple roots), $f_1,\ldots,f_r$ (corresponding
to the negative simple roots) and $h_1,\ldots,h_r$ in the Cartan subalgebra,
such that
$$[h_i,h_j]=0,\, [e_i,f_j]=\delta_{ij}h_i,\, [h_i,e_j] = C(j,i)e_j,\, [h_i,f_j]=-C(j,i)f_j, \text{ for all $i,j$}$$
where $C$ is the Cartan matrix of the root system. The sequence $e_1,\ldots,e_r$,
$f_1,\ldots,f_r$, $h_1,\ldots,h_r$ is called a {\em canonical set of generators}
of $\s$ (cf. \cite[p. 126]{jac}). 

Let $\pi$ be a permutation of $1,\ldots, r$ such that $C(\pi(i),\pi(j))=C(i,j)$
for all $i,j$. Then  there is a unique automorphism $\sigma_\pi$ of $\s$ such
that $\sigma_\pi(h_i) = h_{\pi(i)}$, $\sigma_\pi(e_i)=e_{\pi(i)}$, $\sigma_\pi(f_i)
= f_{\pi(i)}$ (\cite[\S IX.4]{jac}).
Let $O(\s)$ be the subgroup of $\Aut(\s)$ consisting of all such $\sigma_\pi$.
(Note that $O(\s)$ depends on the choice of Cartan subalgebra of $\s$ and on
the choice of a set of simple roots of the root system.)
Let $\Aut(\s)^\circ$ denote the identity component of $\Aut(\s)$. It is
generated by all $\exp(\ad x)$ where $x\in \s$ is such that $\ad x$ is
nilpotent. Then $\Aut(\s) = O(\s)\ltimes \Aut(\s)^\circ$
(\cite[VIII.5 no 3, Cor 1]{bou3}).

\section{Tables}\label{sec:tab}

Table \ref{tabn} contains representatives of the real nilpotent orbits. We briefly explain its contents. The first column has the number of the complex orbit as in Table 10 of \cite{Antotrad}.
We have followed the numbering of that table, except that our 83, 84, 85 is 85, 83, 84 respectively in \cite{Antotrad}. The reason for this change is the following. Consider the map 
$\varphi$ from Section \ref{sec:outer}. Let $\gamma_1,\ldots,\gamma_7$ be the simple roots of the root system of $\ssl(8,\C)$ with respect to the standard Cartan subalgebra consisting of
diagonal matrices. Let $e\in \g_1$ be nilpotent lying in a homogeneous $\ssl_2$-triple $(h,e,f)$. We view $h$ as element of $\ssl(8,\C)\cong \g_0$. 
After replacing $e$ by a $G_0$-conjugate we may assume that $h$ is diagonal (that is, it lies in the Cartan subalgebra that we use) such that $\gamma_i(h)\geq 0$. The 7-tuple 
$(\gamma_1(h),\ldots,\gamma_7(h))$ is called the {\em characteristic} of the orbit of $e$. The characteristic uniquely determines the nilpotent orbit. 
Now $\varphi(e)$ lies in the homogeneous $\ssl_2$-triple $(\varphi(h),\varphi(e),\varphi(f))$.
From the explicit description of $\varphi$ it easily follows that $(\gamma_1(\varphi(h)),\ldots,\gamma_7(\varphi(h))=(\gamma_7(h),\gamma_6(h),\ldots,\gamma_1(h))$. In other words, the
characteristic of the orbit of $\varphi(e)$ is the reverse of the characteristic of the orbit of $e$. If $e$ and $\varphi(e)$ lie in different orbits (in other words, the characteristic is not
invariant under reversion) then $\varphi$ also maps the real $\G_0(\R)$-orbits contained in the $\G_0$-orbit of $e$ to the real $\G_0(\R)$-orbits contained in the $\G_0$-orbit of $\varphi(e)$.
In this case we just list one of the two orbits in Table \ref{tabn}. (And we note that the number of nonzero complex nilpotent orbits is 94, so the last orbit in the table is also followed 
by its image under $\varphi$.) In Table 10 of \cite{Antotrad} the reverse of a characteristic is always listed immediately after it. So, for example,
orbit number 11 is the image under $\varphi$ of orbit number 10 (and in our table we do not have orbit number 11). The only exception is orbit number 82 whose image (with the reverse
characteristic) is orbit number 85. For this reason, we have changed the numbering, so as to have the reverse of orbit 82 immediately below it as orbit 83.

We also consider the action of $\nu$ on the nilpotent orbits. In Section
\ref{sec:nilp} we show that $\nu$ maps a complex nilpotent orbit to itself.
However, it can nontrivially permute the real nilpotent orbits contained in
a given complex nilpotent orbit. If this happens then of each pair
of $\nu$-conjugates we just list one element. We indicate this situation by
an asterisk, that is, if a representative $e$ of a nilpotent orbit in Table
\ref{tabn} is followed by a $*$ then $\nu(e)$ is a representative of a
second real nilpotent $\SL(8,\R)$-orbit. In Section \ref{sec:nilp} we also
describe how to determine the permutation that $\nu$ induces on the nilpotent orbits.

In the second column of the table we give on row $k$ representatives of the real orbits contained in the complex orbit number $k$. This always starts with the same representative as
in \cite{Antotrad}. The  third column then lists the isomorphism type of the real stabilizer $\z_{g_0^\R}(h,e,f) = \{ x\in \g_0^\R \mid [x,h]=[x,e]=[x,f]=0\}$. This stabilizer is reductive,
so it is the sum of a semisimple subalgebra and a toral centre. For the centre we use the following notation: $\mathfrak{t}(k)$ denotes a $k$-dimensional Lie algebra whose elements
are commuting semisimple elements with real eigenvalues, and $\mathfrak{u}(k)$ denotes the same, with the difference that the eigenvalues are all purely imaginary. 
The last column has the isomorphism type of the component group $Z_{\wG_0}(h,e,f)/Z_{\wG_0}(h,e,f)^\circ$. Here $C_k$ indicates the cyclic group of order $k$.

\begin{longtable}{|c|l|l|c|}
\caption{Nilpotent orbit representatives}\label{tabn}\\
\hline
N & Real representative & $\mathfrak{z}_{\g_0^\R}(h,e,f)$ & K \\ \hhline{|=|=|=|=|}
1   &   \makecell[tl]{$e_{1234}$} &  \makecell[tl]{2$\mathfrak{sl}(4,\mathbb{R})$} & \text{  }\{1\}\text{  } \\ \hline

2    &   \makecell[tl]{$e_{1234}+e_{1256}$} &  \makecell[tl]{2$\mathfrak{sl}(2,\mathbb{R})+\mathfrak{s0}(2,3)+\mathfrak{t}(1)$} & \{1\} \\ \hline

3    &   \makecell[tl]{$e_{1234}+e_{1256}+e_{1278}\text{ }^*$}  &  \makecell[tl]{$\mathfrak{sl}(2,\mathbb{R})+\mathfrak{sp}(3,\mathbb{R})$} & $C_2$ \\ \hline

5    &   \makecell[tl]{$e_{1235}+e_{1246}+e_{1347}$} &  \makecell[tl]{$\mathfrak{sl}(3,\mathbb{R})+\mathfrak{t}(2)$} & \{1\} \\ \hline

6    &   \makecell[tl]{$e_{1235}+e_{1246}+e_{1347}+e_{2348}\text{ }^*$} &  \makecell[tl]{$\mathfrak{sl}(4,\mathbb{R})$} & $C_2$ \\ \hline

7    &   \makecell[tl]{$e_{1345}+e_{1236}+e_{1247}+e_{1258}\text{ }^*$} &  \makecell[tl]{$\mathfrak{sl}(3,\mathbb{R})+\mathfrak{t}(1)$} & $C_2$\\ \hline
  
9    &   \makecell[tl]{$e_{1234}+e_{1567}$\\ \text{ } \\$-\frac{1}{2}e_{1234}-\frac{1}{2}e_{1236}+\frac{1}{2}e_{1247}-\frac{1}{2}e_{1267}$\\$-\frac{1}{2}e_{1345}-\frac{1}{2}e_{1356}+\frac{1}{2}e_{1457}-\frac{1}{2}e_{1567}$} &  \makecell[tl]{2$\mathfrak{sl}(3,\mathbb{R})+\mathfrak{t}(1)$\\ \text{ } \\$\mathfrak{sl}(3,\mathbb{C})+\mathfrak{t}(1)$} & $C_2$ \\ \hline

 10    &  \makecell[tl]{ $e_{1245}+e_{1367}+e_{1238}$\\ \text{ }\\ $-e_{1238}-\frac{1}{2}e_{1245}-\frac{1}{2}e_{1247}+\frac{1}{2}e_{1256}$\\$+\frac{1}{2}e_{1267}+\frac{1}{2}e_{1345}-\frac{1}{2}e_{1347}+\frac{1}{2}e_{1356}$\\$-\frac{1}{2}e_{1367}\text{ }^*$} 
& \makecell[tl]{2$\mathfrak{sl}(2,\mathbb{R})+\mathfrak{t}(2)$\\ \text{ } \\$\mathfrak{sl}(2,\mathbb{C})+\mathfrak{u}(1)+\mathfrak{t}(1)$} & $C_2$\\ \hline
 
 12    &   \makecell[tl]{$e_{1345}+e_{2346}+e_{1256}+e_{1237}+e_{1248}\text{ }^*$\\ \text{ } \\$-e_{1237}-e_{1248}-e_{1256}-e_{1345}-e_{2346}\text{ }^*$} &  \makecell[tl]{2$\mathfrak{sl}(2,\mathbb{R})$\\ \text{ }\\2$\mathfrak{sl}(2,\mathbb{R})$} & $C^2_2$\\
 \hline
 
 13    &   \makecell[tl]{$e_{1246}+e_{1357}+e_{1238}+e_{1458}$\\ \text{ } \\$-e_{1238}-\frac{1}{2}e_{1246}+\frac{1}{2}e_{1247}+\frac{1}{2}e_{1256}$\\$+\frac{1}{2}e_{1257}+\frac{1}{2}e_{1346}+\frac{1}{2}e_{1347}+\frac{1}{2}e_{1356}$\\$-\frac{1}{2}e_{1357}-e_{1458}\text{ }^*$} &
 \makecell[tl]{2$\mathfrak{sl}(2,\mathbb{R})+\mathfrak{t}(1)$\\ \text{ } \\2$\mathfrak{su}(2,\mathbb{R})+\mathfrak{t}(1)$} & $\{1\}$\\ \hline
 
 15    &   \makecell[tl]{$e_{2345}+e_{1246}+e_{1357}+e_{1238}$\\ \text{ } \\$-e_{1238}-\frac{1}{2}e_{1246}-\frac{1}{2}e_{1247}-\frac{1}{2}e_{1256}$\\$+\frac{1}{2}e_{1257}+\frac{1}{2}e_{1346}-\frac{1}{2}e_{1347}-\frac{1}{2}e_{1356}$\\$-\frac{1}{2}e_{1357}-e_{2345}\text{ }^*$\\ \text{ } \\$-e_{1238}-\frac{1}{2}e_{1246}+\frac{1}{2}e_{1247}+\frac{1}{2}e_{1256}$\\$+\frac{1}{2}e_{1257}+\frac{1}{2}e_{1346}+\frac{1}{2}e_{1347}+\frac{1}{2}e_{1356}$\\$-\frac{1}{2}e_{1357}-e_{2345}\text{ }^*$} &  \makecell[tl]{$\mathfrak{t}(3)$\\ \text{ } \\$\mathfrak{u}(2)+\mathfrak{t}(1)$\\ \text{ }\\ \text{ } \\ \text{ } \\$\mathfrak{u}(2)+\mathfrak{t}(1)$} & $C_2$\\ \hline

 16 &  \makecell[tl]{$e_{1246}+e_{1357}+e_{1238}+e_{1458}+e_{1678}\text{ }^*$\\ \text{ } \\$-e_{1238}-\frac{1}{2}e_{1246}+\frac{1}{2}e_{1247}+\frac{1}{2}e_{1256}$\\$+\frac{1}{2}e_{1257}+\frac{1}{2}e_{1346}+\frac{1}{2}e_{1347}+\frac{1}{2}e_{1356}$\\$-\frac{1}{2}e_{1357}-e_{1458}-e_{1678}\text{ }^*$} & \makecell[tl]{$G_{2(2)}$\\ \text{ }\\ \text{ } \\$G$} & $C_4$ \\ \hline
 
 18   &   \makecell[tl]{$e_{2345}+e_{1246}+e_{1357}+e_{1238}+e_{1458}$\\ \text{ } \\$-e_{1238}-\frac{1}{2}e_{1246}+\frac{1}{2}e_{1247}+\frac{1}{2}e_{1256}$\\$+\frac{1}{2}e_{1257}+\frac{1}{2}e_{1346}+\frac{1}{2}e_{1347}+\frac{1}{2}e_{1356}$\\$-\frac{1}{2}e_{1357}-e_{1458}-e_{2345}\text{ }^*$\\ \text{ } \\
$-e_{1238}-2e_{1246}+2e_{1357}-2e_{1458}-\frac{1}{2}e_{2345}$\\ \text{ } \\
$\frac{1}{4}e_{1238}+\frac{1}{8}e_{1246}-4e_{1247}-\frac{1}{8}e_{1256}$\\$-4e_{1257}-e_{1346}-\frac{1}{8}e_{1347}-e_{1356}$\\$+\frac{1}{8}e_{1357}+4e_{1458}+4e_{2345}\text{ }^*$}
&  \makecell[tl]{2$\mathfrak{sl}(2,\mathbb{R})$\\ \text{ }\\ \text{ } \\2$\mathfrak{su}(2)$\\ \text{ }\\  \text{ }\\2$\mathfrak{sl}(2,\mathbb{R})$\\ \text{ }\\ \text{ } \\2$\mathfrak{su}(2)$} & $C_4$\\ \hline

20    &   \makecell[tl]{$e_{1256}+e_{1347}+e_{2348}$} &  \makecell[tl]{$3\mathfrak{sl}(2,\mathbb{R})+\mathfrak{t}(1)$} & \{1\} \\ \hline
 21    &   \makecell[tl]{$e_{2345}+e_{1347}+e_{1567}+e_{1268}$} &  \makecell[tl]{$\mathfrak{sl}(2,\mathbb{R})+\mathfrak{t}(2)$} & $\{1\}$\\ \hline
 22    &   \makecell[tl]{$e_{2456}+e_{2347}+e_{1357}+e_{1348}+e_{1268}$} &  \makecell[tl]{$\mathfrak{t}(2)$} & \{1\} \\ \hline
 
 23    &   \makecell[tl]{$e_{3456}+e_{1347}+e_{1257}+e_{2348}+e_{1268}\text{ }^*$} &  \makecell[tl]{2$\mathfrak{sl}(2,\mathbb{R})$} & $C_2$\\ \hline

 25    &   \makecell[tl]{$e_{2456}+e_{2347}+e_{1457}+e_{1367}+e_{1358}+e_{1268}$} &  \makecell[tl]{$\mathfrak{t}(1)$} & $\{1\}$\\ \hline

 26    &   \makecell[tl]{$e_{2456}+e_{2357}+e_{1457}+e_{1367}$\\$+e_{2348}+e_{1358}+e_{1268}\text{ }^*$} & & $C_8$ \\ 
 \hline
 
 27    &   \makecell[tl]{$e_{1345}+e_{2346}+e_{1257}+e_{1268}$} &  \makecell[tl]{$2\mathfrak{sl}(2,\mathbb{R})+\mathfrak{t}(1)$} & \{1\} \\ \hline
 
 28    &   \makecell[tl]{$e_{2345}+e_{1347}+e_{1267}+e_{1258}+e_{1368}$\\ \text{ } \\$\frac{1}{2}e_{1247}-\frac{1}{2}e_{1248}-\frac{1}{2}e_{1257}-\frac{1}{2}e_{1258}$\\$-e_{1267}-\frac{1}{2}e_{1347}-\frac{1}{2}e_{1348}-\frac{1}{2}e_{1357}$\\$+\frac{1}{2}e_{1358}-e_{1368}-e_{2345}$\\ \text{ } \\$-\frac{1}{2}e_{1247}-\frac{1}{2}e_{1248}+\frac{1}{2}e_{1257}-\frac{1}{2}e_{1258}$\\$-e_{1267}-\frac{1}{2}e_{1347}+\frac{1}{2}e_{1348}-\frac{1}{2}e_{1357}$\\$-\frac{1}{2}e_{1358}-e_{1368}-e_{2345}$}  &  \makecell[tl]{$\mathfrak{t}(2)$\\ \text{ }\\ $\mathfrak{u}(1)+\mathfrak{t}(1)$\\ \text{ } \\
 \text{ } \\  \text{ } \\$\mathfrak{u}(1)+\mathfrak{t}(1)$} &
 $C_2$\\ \hline
 
 30    &   \makecell[tl]{$e_{3456}+e_{1247}+e_{1357}+e_{2358}$\\$+e_{1268}+e_{1368}+e_{2368}\text{ }^*$\\ \text{ } \\$-e_{1247}-e_{1268}-e_{1357}-e_{1368}$\\$-e_{2358}+\frac{1}{2}e_{2367}-e_{2368}-e_{3456}\text{ }^*$}& & $C_4\times S_3$ \\ \hline

 32    &   \makecell[tl]{$e_{2356}+e_{2347}+e_{1457}+e_{1367}+e_{1248}+e_{1358}$} &  \makecell[tl]{$\mathfrak{t}(1)$} & $\{1\}$\\ \hline
 
 33    &  \makecell[tl]{$e_{1356}+e_{2456}+e_{1347}+e_{1257}+e_{2348}+e_{1268}\text{ }^*$}
 &  \makecell[tl]{$\mathfrak{sl}(2,\mathbb{R})$} & $C_4$\\ \hline
 
 34    &   \makecell[tl]{$e_{1256}+e_{3456}+e_{1347}+e_{2348}\text{ }^*$} &
 \makecell[tl]{$3\mathfrak{sl}(2,\mathbb{R})$} & $C_2$\\ \hline
 
 36    &   \makecell[tl]{$e_{1256}+e_{2347}+e_{1357}+e_{1348}$} &  \makecell[tl]{$\mathfrak{t}(3)$} & $\{1\}$\\ \hline
 
 37    &   \makecell[tl]{$e_{2345}+e_{1357}+e_{1467}+e_{1348}+e_{1268}$} &  \makecell[tl]{$\mathfrak{t}(2)$} & $\{1\}$\\ \hline
 
 39    &   \makecell[tl]{$e_{2456}+e_{2347}+e_{1357}+e_{1467}+e_{1348}+e_{1268}\text{ }^*$} &  \makecell[tl]{$\mathfrak{t}(1)$} & $C_2$\\ \hline

 40    &   \makecell[tl]{$e_{2356}+e_{1456}+e_{2347}+e_{1357}+e_{1348}+e_{1278}$\\ \text{ } \\$-e_{1278}-e_{1348}+e_{1357}-e_{1456}-e_{2347}-e_{2356}$} &  \makecell[tl]{$\mathfrak{t}(1)$\\ \text{ }\\ $\mathfrak{t}(1)$} & $C_2$\\ \hline
 
 42    &   \makecell[tl]{$e_{1456}+e_{2347}+e_{1267}+e_{1358}+e_{2358}+e_{1268}$\\ \text{ } \\$-e_{1267}-e_{1268}-\frac{1}{2}e_{1347}-\frac{1}{2}e_{1348}+\frac{1}{2}e_{1357}$\\ $-\frac{1}{2}e_{1358}-e_{1456}-e_{2347}-e_{2358}$} &  \makecell[tl]{$\mathfrak{t}(1)$\\ \text{ } \\ $\mathfrak{t}(1)$} & $C_2$ \\ \hline
 
 43    &   \makecell[tl]{$e_{2345}+e_{1356}+e_{1347}+e_{1267}+e_{1248}$} &  \makecell[tl]{$\mathfrak{t}(2)$} & $\{1\}$\\ \hline
 
 44    &   \makecell[tl]{$e_{1246}+e_{1357}+e_{2348}+e_{2358}$\\ \text{ } \\
   $-e_{1246}-\frac{1}{2}e_{1256}+\frac{1}{2}e_{1258}-e_{1357}-e_{2348}$\\$-\frac{1}{2}e_{2356}-\frac{1}{2}e_{2358}$} &  \makecell[tl]{$\mathfrak{t}(3)$\\ \text{ } \\$\mathfrak{t}(3)$} & $S_3$\\ \hline
 
 45    &  \makecell[tl]{$e_{1258}+e_{1267}+e_{1347}+e_{1368}+e_{1456}+e_{2345}\text{ }^*$\\ \text{ } \\$\frac{1}{2}e_{1247}-\frac{1}{2}e_{1248}-\frac{1}{2}e_{1257}-\frac{1}{2}e_{1258}$\\$-e_{1267}-\frac{1}{2}e_{1347}-\frac{1}{2}e_{1348}-\frac{1}{2}e_{1357}$\\$+\frac{1}{2}e_{1358}-e_{1368}-e_{1456}-e_{2345}\text{ }^*$\\ \text{ } \\$-\frac{1}{2}e_{1247}-\frac{1}{2}e_{1248}+\frac{1}{2}e_{1257}-\frac{1}{2}e_{1258}$\\$-e_{1267}-\frac{1}{2}e_{1347}+\frac{1}{2}e_{1348}-\frac{1}{2}e_{1357}$\\$-\frac{1}{2}e_{1358}-e_{1368}-e_{1456}-e_{2345}\text{ }^*$}
 & \makecell[tl]{$\mathfrak{t}(1)$\\ \text{ }\\ $\mathfrak{u}(1)$\\ \text{ } \\ \text{ }\\ \text{ } \\$\mathfrak{u}(1)$}
 & $C_4\times C_2$ \\ \hline

 47    &   \makecell[tl]{$e_{2345}+e_{1347}+e_{1267}+e_{1258}+e_{1368}+e_{2378}$\\ \text{ } \\$-\frac{1}{4}e_{1247}-\frac{1}{8}e_{1248}+e_{1257}-\frac{1}{2}e_{1258}$\\$-e_{1267}-\frac{1}{2}e_{1347}+\frac{1}{4}e_{1348}
-2e_{1357}$\\$-e_{1358}-e_{1368}-e_{2345}-e_{2378}$\\ \text{ } \\$2e_{1258}-2e_{1267}-2e_{1347}-2e_{1368}-\frac{1}{2}e_{2345}$\\ $-e_{2378}$\\ \text{ } \\$\frac{1}{8}e_{1247}+2e_{1248}-\frac{1}{8}e_{1257}+2e_{1258}$\\$+\frac{1}{4}e_{1267}+2e_{1347}-\frac{1}{8}e_{1348}+
2e_{1357}$\\$+\frac{1}{8}e_{1358}+\frac{1}{4}e_{1368}+4e_{2345}+\frac{1}{4}e_{2378}$} &  \makecell[tl]{$\mathfrak{sl}(2,\mathbb{R})$\\ \text{ }\\$\mathfrak{su}(2)$\\ \text{ }\\ \text{ } \\ \text{ } \\$\mathfrak{sl}(2,\mathbb{R})$\\ \text{ }\\ \text{ } \\$\mathfrak{su}(2)$} & $C_4$\\ \hline
 
 49    &  \makecell[tl]{$e_{2356}+e_{1456}+e_{2347}+e_{1357}$\\$+e_{1267}+e_{1348}+e_{1258}\text{ }^*$\\ \text{ } \\$-2e_{1258}-2e_{1267}+2e_{1348}+4e_{1357}$\\$-2e_{1456}-\frac{1}{2}e_{2347}+\frac{1}{2}e_{2356}\text{ }^*$}& & $C_4\times C_4$ \\ \hline
 
 50    &   \makecell[tl]{$e_{1256}+e_{1357}+e_{1467}+e_{2348}$} &  \makecell[tl]{$\mathfrak{sl}(3,\mathbb{R})+\mathfrak{t}(1)$} & \{1\} \\ \hline
 
 51    &   \makecell[tl]{$e_{2456}+e_{1357}+e_{1467}+e_{2348}+e_{1268}$} &  \makecell[tl]{$\mathfrak{t}(2)$} & \{1\} \\ \hline
 
 52    &   \makecell[tl]{$e_{3456}+e_{2357}+e_{1457}+e_{1367}+e_{2348}+e_{1268}$\\ \text{ } \\$-e_{1268}-e_{1367}-e_{1457}-e_{2348}-e_{2357}-e_{3456}$} &  \makecell[tl]{$\mathfrak{t}(1)$\\ \text{ }\\ $\mathfrak{t}(1)$} & $C_2$\\ \hline

 54    &   \makecell[tl]{$e_{1258}+e_{1267}+e_{1348}+e_{1356}+e_{2348}+e_{2457}\text{ }^*$\\ \text{ } \\$-e_{1258}-e_{1267}-e_{1348}-\frac{1}{2}e_{1356}$\\$+\frac{1}{2}e_{1357}+\frac{1}{2}e_{1456}+\frac{1}{2}e_{1457}-e_{2348}$\\$+\frac{1}{2}e_{2356}+\frac{1}{2}e_{2357}+\frac{1}{2}e_{2456}-\frac{1}{2}e_{2457}\text{ }^*$\\ \text{ } \\$-e_{1258}-e_{1267}-e_{1348}-\frac{1}{2}e_{1356}$\\$-\frac{1}{2}e_{1357}-\frac{1}{2}e_{1456}+\frac{1}{2}e_{1457}-e_{2348}$\\$+\frac{1}{2}e_{2356}-\frac{1}{2}e_{2357}-\frac{1}{2}e_{2456}-\frac{1}{2}e_{2457}\text{ }^*$}
 &  \makecell[tl]{$\mathfrak{t}(1)$\\ \text{ }\\ $\mathfrak{u}(1)$\\ \text{ } \\ \text{ }\\ \text{ }\\ $\mathfrak{u}(1)$}
   & $C_4\times C_2$\\ \hline

 56    &   \makecell[tl]{$e_{1356}+e_{2457}+e_{1267}+e_{3467}+e_{1348}+e_{2348}$\\ \text{ } \\$-e_{1267}-e_{1348}-\frac{1}{2}e_{1356}+\frac{1}{2}e_{1357}$\\$+\frac{1}{2}e_{1456}+\frac{1}{2}e_{1457}-e_{2348}+\frac{1}{2}e_{2356}$\\$+\frac{1}{2}e_{2357}+\frac{1}{2}e_{2456}-\frac{1}{2}e_{2457}-e_{3467}$\\ \text{ } \\$2e_{1267}+2e_{1348}+2e_{1356}+\frac{1}{2}e_{2348}$\\$+\frac{1}{2}e_{2457}-e_{3467}$\\ \text{ } \\$4e_{1267}+4e_{1348}+\frac{1}{8}e_{1356}-4e_{1357}$\\$-e_{1456}-\frac{1}{8}e_{1457}+4e_{2348}-\frac{1}{8}e_{2356}$\\$-4e_{2357}-e_{2456}+\frac{1}{8}e_{2457}+\frac{1}{4}e_{3467}$} &  \makecell[tl]{$\mathfrak{sl}(2,\mathbb{R})$\\ \text{ }\\ $\mathfrak{su}(2)$\\ \text{ }\\ \text{ } \\ \text{ } \\$\mathfrak{sl}(2,\mathbb{R})$\\ \text{ }\\ \text{ } \\$\mathfrak{su}(2)$} & $C_4$\\ \hline

 58    &   \makecell[tl]{$e_{2356}+e_{1456}+e_{1257}+e_{1367}+e_{2348}$\\ \text{ } \\$-e_{1257}-e_{1367}-e_{1456}-e_{2348}-e_{2356}$} &  \makecell[tl]{$\mathfrak{sl}(2,\mathbb{R})+\mathfrak{t}(1)$\\ \text{ } \\$\mathfrak{sl}(2,\mathbb{R})+\mathfrak{t}(1)$} & $C_2$\\ \hline

 60    &   \makecell[tl]{$e_{2356}+e_{1456}+e_{1357}+e_{1267}+e_{2348}+e_{1258}\text{ }^*$\\ \text{ } \\$-e_{1258}-e_{1267}-e_{1357}-e_{1456}-e_{2348}-e_{2356}\text{ }^*$} &  \makecell[tl]{$\mathfrak{t}(1)$\\ \text{ }\\ $\mathfrak{t}(1)$} & $C^2_2$\\ \hline
 
 61    &   \makecell[tl]{$e_{1356}+e_{2457}+e_{1267}+e_{1348}+e_{2348}$\\ \text{ } \\$-e_{1267}-e_{1348}-\frac{1}{2}e_{1356}-\frac{1}{2}e_{1357}$\\$-\frac{1}{2}e_{1456}+\frac{1}{2}e_{1457}-e_{2348}+\frac{1}{2}e_{2356}$\\$-\frac{1}{2}e_{2357}-\frac{1}{2}e_{2456}-\frac{1}{2}e_{2457}$\\ \text{ } \\$-e_{1267}-e_{1348}-\frac{1}{2}e_{1356}+\frac{1}{2}e_{1357}$\\$+\frac{1}{2}e_{1456}+\frac{1}{2}e_{1457}-e_{2348}+\frac{1}{2}e_{2356}$\\$+\frac{1}{2}e_{2357}+\frac{1}{2}e_{2456}-\frac{1}{2}e_{2457}$} &  \makecell[tl]{$\mathfrak{t}(2)$\\ \text{ }\\ $\mathfrak{u}(1)+\mathfrak{t}(1)$\\ \text{ }\\ \text{ } \\ \text{ } \\$\mathfrak{u}(1)+\mathfrak{t}(1)$} & $C_2$\\ \hline

 63    &   \makecell[tl]{$e_{2356}+e_{1457}+e_{1367}+e_{2348}+e_{1358}+e_{1268}$\\ \text{ } \\$2e_{1268}-e_{1358}+2e_{1367}+2e_{1457}$\\$+\frac{1}{2}e_{2348}+\frac{1}{2}e_{2356}$} &  \makecell[tl]{$\mathfrak{t}(1)$\\ \text{ }\\ $\mathfrak{t}(1)$}
 & $C_4$\\ \hline
 
 65    &   \makecell[tl]{$e_{3456}+e_{1357}+e_{2467}+e_{1348}$\\$+e_{2348}+e_{1258}+e_{126
8}$\\ \text{ } \\$-e_{1258}-e_{1268}-e_{1348}-\frac{1}{2}e_{1357}$\\$+\frac{1}{2}e_{1367}+\frac{1}{2}e_{1457}+\frac{1}{2}e_{146
7}-e_{2348}$\\$+\frac{1}{2}e_{2357}+\frac{1}{2}e_{2367}+\frac{1}{2}e_{2457}-\frac{1}{2}e_{2467}$\\$-e_{3456}$\\ \text{ } \\$4e_{1258}+4e_{1268}+4e_{1348}+\frac{1}{8}e_{1357}$\\$-2e_{1367}-2e_{1457}-\frac{1}{8}e_{1467}+4e_{
     2348}$\\$-\frac{1}{8}e_{2357}-2e_{2367}-2e_{2457}+\frac{1}{8}e_{2467}$\\$+\frac{1}{4}e_{3456}$\\ \text{ } \\$-4e_{1258}-4e_{1268}-4e_{1348}-4e_{1357}$\\$-\frac{1}{4}e_{2348}-\frac{1}{4}e_{2467}-\frac{1}{4}e_{3456}$} & & $C_4\times C_2$ \\ \hline
 
 67    &   \makecell[tl]{$e_{2456}+e_{2357}+e_{1457}+e_{1367}$\\$+e_{2348}+e_{1358}+e_{1278}\text{ }^*$\\ \text{ } \\$-4e_{1278}-4e_{1358}-4e_{1367}-4e_{1457}$\\$-\frac{1}{4}e_{2348}-\frac{1}{4}e_{2357}-\frac{1}{4}e_{2456}\text{ }^*$} & & $C_4\times C_2$  \\ \hline
 
69    &   \makecell[tl]{$e_{2456}+e_{2357}+e_{1467}+e_{2348}+e_{1358}+e_{1268}$} &  \makecell[tl]{$\mathfrak{t}(1)$} & $\{1\}$\\ \hline

70    &   \makecell[tl]{$e_{2456}+e_{1457}+e_{1367}+e_{2348}+e_{1258}+e_{1358}$\\ \text{ } \\$-e_{1258}-e_{1358}-e_{1367}-e_{1457}-e_{2348}-e_{2456}$} &  \makecell[tl]{$\mathfrak{t}(1)$\\ \text{ }\\ $\mathfrak{t}(1)$} & $C_2$\\ \hline

72    &   \makecell[tl]{$e_{2346}+e_{1456}+e_{1357}+e_{1267}+e_{1248}+e_{1258}\text{ }^*$\\ \text{ } \\$-e_{1248}-e_{1258}-e_{1267}-e_{1357}-e_{1456}-e_{2346}\text{ }^*$} &  \makecell[tl]{$\mathfrak{t}(1)$\\ \text{ }\\ $\mathfrak{t}(1)$} & $C^2_2$\\ \hline

73    &   \makecell[tl]{$e_{2346}+e_{1456}+e_{1257}+e_{1348}+e_{1358}\text{ }^*$\\ \text{ } \\ $-\frac{1}{2}e_{1247}+\frac{1}{2}e_{1248}-e_{1257}-\frac{1}{2}e_{1347}$\\$-\frac{1}{2}e_{1348}-e_{1358}-e_{1456}-e_{2346}\text{ }^*$}
 &  \makecell[tl]{$\mathfrak{t}(2)$\\ \text{ }\\ $\mathfrak{u}(1)+\mathfrak{t}(1)$} & $C_2^2$\\ \hline

75    &   \makecell[tl]{$e_{2456}+e_{2357}+e_{1457}+e_{1367}$\\$+e_{2348}+e_{1258}+e_{1358}\text{ }^*$\\ \text{ } \\$-e_{1258}-e_{1358}-e_{1367}-e_{1457}$\\$-e_{2348}+e_{2357}-e_{2456}\text{ }^*$}
& & $C_4\times C_2$  \\ \hline

 77    &   \makecell[tl]{$e_{2456}+e_{1357}+e_{1267}+e_{2348}$\\$+e_{1458}+e_{1368}+e_{146
8}$\\ \text{ } \\$-2e_{1258}-e_{1267}-e_{1357}-e_{1368}$\\$-e_{1458}-e_{1468}-e_{2348}-e_{2456}$\\ \text{ } \\$-4e_{1267}-4e_{1357}-4e_{1368}-4e_{1458}$\\$-4e_{1468}-\frac{1}{4}e_{2348}-\frac{1}{4}e_{2456}$\\ \text{ } \\$8e_{1258}+4e_{1267}+4e_{1357}+4e_{1368}$\\$+4e_{1458}+4e_{1468}+\frac{1}{4}e_{2348}+\frac{1}{4}e_{
2456}$} & & $C_4\times C_2$  \\ \hline

 79    &   \makecell[tl]{$e_{2456}+e_{2347}+e_{1567}+e_{1348}+e_{1358}+e_{1268}$} &  \makecell[tl]{$\mathfrak{t}(1)$} & $\{1\}$\\ \hline
 
 80    &   \makecell[tl]{$e_{2456}+e_{2347}+e_{1367}+e_{1348}+e_{1258}+e_{1358}\text{ }^*$} &  \makecell[tl]{$\mathfrak{t}(1)$} & $C_2$\\ \hline

 82    &   \makecell[tl]{$e_{2456}+e_{2347}+e_{1457}+e_{1367}$\\$+e_{1348}+e_{1258}+e_{1358}\text{ }^*$\\ \text{ } \\$-4e_{1258}-4e_{1348}-4e_{1358}-4e_{1367}$\\$-4e_{1457}-\frac{1}{4}e_{2347}-\frac{1}{4}e_{2456}\text{ }^*$} & & $C_4\times C_2$  \\ \hline

 84    &   \makecell[tl]{$e_{2456}+e_{2357}+e_{1467}+e_{2348}$\\$+e_{1358}+e_{1368}+e_{1278}$\\ \text{ } \\$-4e_{1278}-4e_{1358}-4e_{1368}-4e_{1467}$\\$-\frac{1}{4}e_{2348}-\frac{1}{4}e_{2357}-\frac{1}{4}e_{2456}$} & & $C_4$  \\ \hline
 86    &   \makecell[tl]{$e_{2356}+e_{2457}+e_{1367}+e_{2348}$\\$+e_{1458}+e_{1268}+e_{127
8}$\\ \text{ } \\$-4e_{1268}-4e_{1278}-4e_{1367}-4e_{1458}$\\$-\frac{1}{4}e_{2348}-\frac{1}{4}e_{2356}-\frac{1}{4}e_{2457}$} & & $C_4$   \\ \hline

 88    &   \makecell[tl]{$e_{3456}+e_{2457}+e_{2367}+e_{1467}$\\$+e_{2348}+e_{1358}+e_{126
     8}$\\ \text{ } \\$-4e_{1268}-4e_{1358}-4e_{1467}-\frac{1}{4}e_{2348}$\\$-\frac{1}{4}e_{2367}-\frac{1}{4}e_{2457}-\frac{1}{4}e_{3456}$} & & $C_4$   \\ \hline
 
 90    &   \makecell[tl]{$e_{1356}+e_{2456}+e_{2347}+e_{1257}+e_{1248}\text{ }^*$} &  \makecell[tl]{$\mathfrak{t}(2)$} & $C_2$\\ \hline

 92    &   \makecell[tl]{$e_{1256}+e_{2347}+e_{1357}+e_{1467}+e_{1348}\text{ }^*$} &  \makecell[tl]{$\mathfrak{sl}(2,\mathbb{R})+\mathfrak{t}(1)$} & $C_2$\\ \hline
 
 93    &   \makecell[tl]{$e_{2345}+e_{1356}+e_{1457}+e_{1267}+e_{1348}+e_{1258}\text{ }^*$\\ \text{ } \\$-e_{1258}-e_{1267}-e_{1348}+e_{1356}-e_{1457}-e_{2345}\text{ }^*$}
 &  \makecell[tl]{$\mathfrak{sl}(2,\mathbb{R})$\\ \text{ }\\$\mathfrak{sl}(2,\mathbb{R})$} & $C_4\times C_2$\\ \hline
\end{longtable}
\vspace{1cm}

Table \ref{tab1} shows some data concerning the complex and real semisimple orbits. As noted above, the Weyl group $W$ has thirty-two conjugacy classes of 
subgroups that are stabilizers of elements. We denote these by $H_1,\ldots,H_{32}$. Here $H_1=W$ and $H_{32}=1$. 
The first column of Table \ref{tab1} has the index $i$. The second column has a basis of $\h_{H_i}$. In some cases these bases are different from 
the ones found by Antonyan (in other words, we work with different subgroups of $W$). The reason is the following. Write $\hat \g_0 = \ssl(8,\C)$ and for $p\in \h$
set $\z_{\hat\g_0} (p) = \{ x\in \hat\g_0 \mid x\cdot p=0\}$. Then with the bases in Table \ref{tab1} we have that $\z_{\hat\g_0}(p)$ has a Cartan subalgebra 
consisting of diagonal matrices, so that we do not have to diagonalize when working with it. This simplifies a lot of computations later on.

The fifth column of Table \ref{tab1} has a description of the isomorphism type of the component group of the stabilizer $Z_{\G_0}(p)$ for $p\in \h_{H_i}^\circ$. Here the notation $C_k$ indicates the
cyclic group of order $k$. The numbers in the last column have the following meaning. The space $\h_{H_i}$ has a finite number of subsets
$U_1,\ldots, U_r$ (where each $U_j$ is the real span of some basis of $\h_{H_i}$) such that the set of $p\in \h_{H_i}^\circ$ with the property that $p$ is 
$\G_0$-conjugate to a $q\in \g_1$ which is real (i.e., lies in $\g_1^\R$), is exactly the union of $U_1^\circ,\ldots,U_r^\circ$, where $U_i^\circ$ is an open set in $U_i$ (see Theorem \ref{th1}, and the
explanation below it). 
Let $p\in U_j^\circ$ be conjugate to $q\in \g_1^\R$. Then the $\G_0$-orbit of $q$ contains $k_j$ real $\G_0(\R)$-orbits (which we have determined using Galois cohomology).
The last column of Table \ref{tab1} has the $r$-tuple $(k_1,\ldots, k_r)$. This means that the single class of semisimple $\G_0$-orbits $\h_{H_i}^\circ$ contains 
$k_1+\cdots +k_r$ classes of real $\G_0(\R)$-orbits.
    
\begin{longtable}{|c|l|l|l|l|l|}
\caption{Classes of semisimple elements}\label{tab1}\\
\hline
i & basis of $\h_{H_i}$ & $\mathfrak{z}_{\hat{\mathfrak{g}}_0}(p)$ & $|H_i|$ & K &\\ \hhline{|=|=|=|=|=|=|}
 1  & -  &  $A_7$ & 2 903 040 & $\{1\}$ & - \\ \hline
 2  & $p_1+p_2+p_4$  &  $C_4$ & 51 840 & $C_2$ & $(2,2)$ \\ \hline
 3  & $p_1$  &  $2A_3$ & 23 040 & $C_2$ & $(2,2)$\\ \hline
 4  & $p_1+p_2$  &  $2B_2+\mathfrak{t}(1)$ & 3 840 & $C_2$ & $(3,3)$ \\ \hline
 5  & $-2p_1-p_2+p_3+p_5$  &  $B_3$ & 5 040 & $C_2$ & $(4,4)$\\ \hline
 6  & $\frac{1}{2}p_1-\frac{1}{2}p_2+\frac{1}{2}p_3+p_5+\frac{1}{2}p_6+p_7$  &  $A_3+\mathfrak{t}(1)$ & 1 440 & $C_2$ & $(4,4)$ \\ \hline
 7  & $3p_1+2p_2-p_3+p_7$  &  $B_2+A_1$ & 720 & $C_2$ & $(6,6)$\\ \hline
 8  & $-2p_1-p_2+p_3$  &  $3A_1+\mathfrak{t}(1)$ & 288 & $C_2$ & $(7,7)$\\ \hline
 9  & $p_1+p_2,p_4$  &  $2B_2$ & 1 920 & $C^2_2$ & $(4,1,1,4)$\\ \hline
 10 & $p_1,p_2$  &  $4A_1+\mathfrak{t}(1)$ & 384 & $C^2_2$ & $(6,3,1,6)$ \\ \hline
 11 & $p_1+p_2+p_4,p_3+p_5-p_6+p_7$  & $A_3$ & 720 & $C^2_2$ & $(6,4,4,6)$\\ \hline
 12 & $p_4,-p_1+p_3+p_5$  &  $A_3$ & 720 & $C^2_2$ & $(3,1,1,3)$ \\ \hline
 13 & $\makecell[tl]{\frac{1}{2}p_1-\frac{1}{2}p_2+\frac{1}{2}p_4+p_5,\\-p_1-p_3+p_7}$ &  $B_2+\mathfrak{t}(1)$ & 240 & $C^2_2$ & $(8,8)$ \\ \hline
 14 & $\makecell[tl]{-2p_1+p_2+2p_3+2p_4+p_6+p_7,\\p_1-p_3-p_4+p_5}$  &  $3A_1$ & 144 & $C^2_2$  & $(10,3,3,10)$ \\ \hline
 15 & $p_6,-2p_1+p_2-p_3+p_4+p_7$  &  $2A_1+\mathfrak{t}(2)$ & 96 & $C^2_2$ & $(5,5,5,5)$\\ \hline
 16 & $-2p_1-p_2+p_3,p_1+p_2+p_4$  &  $2A_1+\mathfrak{t}(1)$ & 72 & $C^2_2$ & $(4,12,12,4)$ \\ \hline
 17 & $p_1+p_2,-p_1+p_3$  &  $A_1+\mathfrak{t}(3)$ & 48 & $C^2_2$ & $(1,13,1,13)$ \\ \hline
 18 & $p_1,p_2,p_4$  &  $4A_1$ & 192 & $C^3_2$ & $(8,2,2,2,2,8)$\\ \hline
 19 & $p_1+p_2,p_4,p_1-p_3+p_7$  &  $B_2$ & 120 & $C^3_2$ & $(10,2,2,10)$ \\ \hline
 20 & $p_1-p_2+p_3,p_1-p_2+p_4,p_2+p_7$  &  $2A_1+\mathfrak{t}(1)$ & 48 & $C^3_2$ & $(4,14,3,6,6,3,14,4)$\\ \hline
 21 & $p_4,-p_2+p_5+p_6,p_2-p_3+p_7$  &  $2A_1+\mathfrak{t}(1)$ & 48 & $C^3_2$ & $(13,2,3,3,2,13)$\\ \hline
 22 & $-2p_1-p_2+p_3,p_1+p_2+p_4,p_5$  &  $2A_1$ & 36 & $C^3_2$ & $(6,2,16,4,4,2,16,6)$\\ \hline
 23 & $p_1+p_2,-p_1+p_3,p_4$  &  $A_1+\mathfrak{t}(2)$ & 24 & $C^3_2$ & $(18,5,2,2,2,2,5,18)$\\ \hline
 24 & $p_2,p_3,p_6$  &  $\mathfrak{t}(4)$ & 16 & $C^3_2$ & $(20,7,2,7,2,20)$\\ \hline
 25 & $\makecell[tl]{p_1+p_2,p_4,-p_1-p_3+p_5,\\p_1+p_3+p_7}$  &  $2A_1$ & 24 & $C^4_2$  & $\makecell[tl]{(4,18,4,1,4,5,1,5,4,\\4,4,18)}$\\ \hline
 26 & $p_1+p_2,-p_1+p_3,p_4,p_7$  &  $A_1+\mathfrak{t}(1)$ & 12 & $C^3_2$ & $(8,2,2,2,2,8)$ \\ \hline
 27 & $p_1,p_4,p_6,p_7$  &  $\mathfrak{t}(3)$ & 8  &$C^4_2$ & $(27,7,3,3,1,7,3,27)$ \\ \hline
 28 & $-p_1+p_4,p_5,p_1+p_6,p_1+p_7$  &  $\mathfrak{t}(3)$ & 8 & $C^4_2$ & $\makecell[tl]{(4,28,4,8,4,4,10,4,4,\\10,8,28)}$ \\ \hline
 29 & $p_1+p_2,-p_1+p_3,p_4,p_5,p_7$  &  $A_1$ & 6 &  $C^5_2$ & $(2,2,8,32,2,2,8,32)$ \\ \hline
 30 & $p_1,p_2,p_3,p_4,p_7$  &  $\mathfrak{t}(2)$ & 4 & $C^5_2$  & $\makecell[tl]{(38,10,2,6,2,6,3,6,11,\\11,3,2,6,2,6,10,6,38)}$\\ \hline
 31 & $p_1,p_3,p_4,p_5,p_6,p_7$  &  $\mathfrak{t}(1)$ & 2 & $C^6_2$ & $\makecell[tl]{(52,14,3,4,1,4,4,4,3,\\14,52)}$ \\ \hline
 32 & $p_1,p_2,p_3,p_4,p_5,p_6,p_7$  &  $0$ & 1 & $U_8$ & $\makecell[tl]{(72, 72, 6, 6, 6, 6, 20, 20,\\ 2, 2)}$  \\ \hline

\end{longtable}

\begin{rmk}
The group $U_8$ on line 32 of this table is of order $2^8$. It has a centre of order 4, generated by $\diag(i,\ldots,i)$. The quotient of 
$U_8$ by its centre is isomorphic to $C_2^6$. This is in accordance with \cite[Table 1]{popov}, where the author determines the stabilizer of a
generic element. However, the group that he considers is not $\SL(8,\C)$ but rather its image in $\GL(\bigwedge^4 \C^8)$.
\end{rmk}

From Table \ref{tab1} we see that there are very many classes of semisimple orbits. For this reason we cannot list them all. Table \ref{tab:semsimreps} contains 
representatives of the semisimple orbits that have a stabilizer of large dimension ($\geq 10$) (with the exception of the semisimple orbits with number 18 in 
Table \ref{tab1}, as here the representatives turned out to be rather bulky). Also we have given representatives of the orbits up to the action of the
group generated by $\varphi$ and $\nu$ (Section \ref{sec:outer}). In Table \ref{tab:semsimreps} the first column has the index of the corresponding complex
semisimple orbit (as in Table \ref{tab1}). The second column has the representatives of the real semisimple orbits. As outlined before, these are representatives 
of classes of orbits. They can depend on parameters, and the restrictions on the parameters are also listed. Furthermore, if $u$ is a representative in the table
then the corresponding class is given by $a u$ where $a\in \R\setminus \{0\}$. So we see that a class in the table can depend on up to three parameters. 
All elements of a given class share the same stabilizer in $\G_0(\R)$ and $\ssl(8,\R)$. In the third column we give the isomorphism type of the stabilizer in 
$\ssl(8,\R)$. The last column has the orbit of the element under the group generated by $\varphi$ and $\nu$. If $u$ is the given representative, then 
I means that the orbit is $\eta(u)$ for $\eta \in \{1,\nu,\varphi\nu, \nu\varphi\nu\}$. For II and III this last set is $\{1,\nu\}, \{1\}$ respectively. 
By putting all these orbits together we get representatives of the semisimple $\SL(8,\R)$-orbits. We refer to Remark \ref{rem:phinusemsim} for the methods
we used to determine the permutation of the orbits induced by the maps $\nu$ and $\varphi$.

\begin{longtable}{|l|p{.7\textwidth}|p{.2\textwidth}|l|}
    \caption{representatives of real semisimple orbits with large stabilizer}\label{tab:semsimreps}

\endfirsthead

\endhead

\endfoot

\endlastfoot
    
\hline    
$i$ & representative & stabilizer & orbit \\ \hline
2 & $e_{ 1 2 3 4 }+e_{ 1 3 5 7 }+e_{ 1 3 6 8 }+e_{ 2 4 5 7 }+e_{ 2 4 6 8 }+
 e_{ 5 6 7 8 }$ & $\mathfrak{sp}(4,\R)$ & I  \\ \hline
3 & $(5+3\sqrt{2})e_{ 1234 }-3e_{ 1237 }+3e_{ 1248 }+ 5e_{ 1278 }+ 3e_{ 1346 }-
 5e_{ 1367 }+ 5e_{ 1468 }+ (3+5\sqrt{2})e_{ 1678 }-(3+5\sqrt{2})e_{ 2345 }
 -5e_{ 2357 }+ 5e_{ 2458 }-3e_{ 2578 }+ 5e_{ 3456 }+ 3e_{ 3567 }-3e_{ 4568 }+
 (5+3\sqrt{2})e_{ 5678 }$ & $\ssl(4,\C)$ & II \\ \cline{2-4}
 & $e_{1234 }+e_{5678 }$ & $2\ssl(4,\R)$ & II \\ \hline
4 & $-2e_{1234}-e_{1357}+e_{1368}+e_{1458}+e_{1467}+e_{2358}+e_{2367}+e_{2457}-e_{2468}
 -2e_{5678}$ & $\mathfrak{so}(5,\C)+\mathfrak{u}(1)$ & I \\ \cline{2-4}
 & $e_{1234}+e_{1357}+e_{2468}+e_{5678}$ & $2\mathfrak{so}(2,3)+\mathfrak{t}(1)$ & II\\
 \hline 
5 &  $-2e_{1234}+e_{1256} -e_{1357} +e_{1458} +e_{2367} -e_{2468} 
  +e_{3478} -2e_{5678}$ & $\mathfrak{so}(4,3)$ & I \\  \cline{2-4} 
 & $e_{ 1 2 3 4 }- e_{ 1 2 5 6 }+ e_{ 1 2 7 8 }+ e_{ 1 3 5 7 }+e_{ 1 3 6 8 }
-e_{ 1 4 5 8 }+ e_{ 1 4 6 7 }+e_{ 2 3 5 8 }- e_{ 2 3 6 7 }+ e_{ 2 4 5 7 }+e_{ 2 4 6 8 }+
e_{ 3 4 5 6 }-e_{ 3 4 7 8 }+ e_{ 5 6 7 8 }$ 
& $\so(7)$ & I \\ \hline
6 & $e_{1234}+e_{1256}-2e_{1278} -e_{1357} +2e_{1458}
-e_{1467} -e_{2358} +2e_{2367} -e_{2468} -2e_{3456}
+e_{3478}+e_{5678}$ & $\ssl(4,\R)+\mathfrak{t}(1)$ & II \\ \cline{2-4}
& $-2e_{1234} -e_{1256} +e_{1278} +2e_{1357} -e_{1458}
+e_{1467} +e_{2358} -e_{2367} +2e_{2468} +e_{3456}
-e_{3478} -2e_{5678}$ &  $\ssl(4,\R)+\mathfrak{u}(1)$ & II \\ \cline{2-4}
& $80e_{1234} -48e_{1237} +64e_{1256} -128e_{1278}
+48e_{1345} +80e_{1357} +128e_{1458} -64e_{1467}
-4e_{2358} +8e_{2367} -5e_{2468} +3e_{2678} -8e_{3456}
+4e_{3478} -3e_{4568} +5e_{5678}$ & $\ssl(2,\mathbb{H})+\mathfrak{u}(1)$ & I \\ \hline
7 & $3e_{1234} -e_{1256} -e_{1278} +2e_{1357} +2e_{2468}
-e_{3456} -e_{3478} +3e_{5678}$ & $\ssl(2,\R)+\mathfrak{so}(2,3)$ & I \\
\cline{2-4}
& $-3e_{1234} +e_{1256} +e_{1278} -e_{1357} +e_{1368}
-e_{1458} -e_{1467} -e_{2358} -e_{2367} +e_{2457}
-e_{2468} +e_{3456} +e_{3478} -3e_{5678}$ & $\mathfrak{su}(2)+
\mathfrak{so}(4,1)$ & I \\ \cline{2-4}
 & $-3e_{ 1 2 3 4 }+e_{ 1 2 5 6 }+e_{ 1 2 7 8 }-e_{ 1 3 5 7 }+e_{ 1 3 6 8 }+
e_{ 1 4 5 8 }+e_{ 1 4 6 7 }+e_{ 2 3 5 8 }+e_{ 2 3 6 7 }+e_{ 2 4 5 7 }-e_{ 2 4 6 8 }
+e_{ 3 4 5 6 }+e_{ 3 4 7 8 }-3e_{ 5 6 7 8 }$
& $\mathfrak{su}(2)+\so(5)$ & I \\
\hline
8 & $-2e_{1234}+e_{1256}-e_{1357}-e_{2468}+e_{3478}-2e_{5678}$ &
$3\ssl(2,\R)+\mathfrak{t}(1)$ & II \\
\cline{2-4}
& $e_{1234}+e_{1256} +2e_{1357}+ 3e_{1368} +3e_{1458} +3e_{2367} 
+3e_{2457} +2e_{2468} +e_{3478} +e_{5678}$ &
$\ssl(2,\R)+\ssl(2,\C)+\mathfrak{u}(1)$ & I \\ \cline{2-4}
& $-128e_{1234}+64e_{1256} -17e_{1357} +15e_{1358} +15e_{1367} +17e_{1368} 
  +15e_{1457} +17e_{1458} +17e_{1467} -15e_{1468} -15e_{2357} +17e_{2358} +17
  e_{2367} +15e_{2368} +17e_{2457} +15e_{2458} +15e_{2467} -17e_{2468} +4
  e_{3478} -8e_{5678}$ & $\mathfrak{su}(2)+\ssl(2,\C)+\mathfrak{u}(1)$ & I \\
  \cline{2-4}
  & $e_{1234} -e_{1256} +e_{1278} +e_{1357} +e_{1368} +e_{1467} 
  +e_{2358} +e_{2457} +e_{2468} +e_{3456} -e_{3478} +e_{5678}$ & $3\mathfrak{su}(2)+\mathfrak{t}(1)$ & I \\
  \hline
9 & $e_{1234}+e_{1357}+e_{2468}+e_{5678}+b(
  e_{2457}+e_{1368})$; $b\neq 0,1,-1$ & $2\mathfrak{so}(2,3)$& I \\ \cline{2-4}
  &  $-e_{1357}+e_{1467} +4e_{2358} -4e_{2468}+
b(e_{1278} -4e_{3456})$; $b\neq 0,1,-1$ &  $2\mathfrak{so}(2,3)$ & II \\ \cline{2-4}
  & $-2e_{1234}-(1-b)e_{1357} +(1-b)e_{1368} +(1+b)e_{1458}
  +(1+b) e_{1467} +(1+b) e_{2358}+(1+b) e_{2367}+(1-b) e_{2457}
  -(1-b)e_{2468} -2e_{5678}$; $b\neq 0,1,-1$ & $\mathfrak{so}(5,\C)$ & I \\
  \hline
10 &  $e_{1234}+e_{5678} +b(e_{1357} +e_{2468})$ ; $b\neq 0,-1,1$
  & $4\ssl(4,\R)+\mathfrak{t}(1)$ & II \\ \cline{2-4}
& $-5e_{1234} -3e_{1236} -3e_{1247} +5e_{1267} +3e_{1348} -5e_{1368} 
  -5e_{1478} -3e_{1678} +3e_{2345} -5e_{2356} -5e_{2457} -3e_{2567} +5
  e_{3458} +3e_{3568} +3e_{4578} -5e_{5678}
  +b(8e_{1357} +8e_{2468 })$, $b\neq 0,1,-1$ &
   $2\ssl(2,\C)+\mathfrak{u}(1)$ & I \\ \cline{2-4}
& $-3e_{1356} +5e_{1357} +5e_{1368} +3e_{1378} +5e_{1456} +3e_{1457} 
  +3e_{1468} -5e_{1478} -5e_{2356} +3e_{2357} +3e_{2368} +5e_{2378} +3e_{2456}
  +5e_{2457} +5e_{2468} -3e_{2478}+b(32e_{1234} +2e_{5678})$; $b\neq 0,1,-1$ &
   $2\ssl(2,\C)+\mathfrak{u}(1)$ & I \\ \cline{2-4}
  & $-(1+b)e_{1234}+( 1-b)e_{ 1278} -(1+b)e_{ 1357}-(1-b)e_{ 1458} -(1-b)e_{ 2367}
  -(1+b)e_{ 2468} +(1-b)e_{ 3456} -(1+b)e_{ 5678 }$; $b\neq 0,1,-1$ & $2\ssl(2,\C)+\mathfrak{t}(1)$ & II
  \\ \cline{2-4}
& $5e_{1234} +3e_{1236} +3e_{1247} -5e_{1267} -3e_{1348} +5e_{1368} 
  +5e_{1478} +3e_{1678} -3e_{2345} +5e_{2356} +5e_{2457} +3e_{2567} -5
  e_{3458} -3e_{3568} -3e_{4578} +5e_{5678}
  +b(-8e_{1357} +8e_{2468})$; $b\neq 0,1,-1$ &  $2\ssl(2,\C)+\mathfrak{u}(1)$ & III \\ \cline{2-4}
& $-5e_{1357} -3e_{1358} -3e_{1367} +5e_{1368} -3e_{1457} +5e_{1458} 
  +5e_{1467} +3e_{1468} -3e_{2357} -5e_{2358} -5e_{2367}+3e_{2368} -5
  e_{2457} +3e_{2458} +3e_{2467} +5e_{2468}
  +b(8e_{1234} +8e_{5678})$; $b\neq 0,1,-1$ & $2\ssl(2,\C)+\mathfrak{u}(1)$ & III \\ \cline{2-4}
  &   $e_{1234}+ e_{5678}+b(e_{1357}-e_{2468})$; $b\neq 0,1,-1$ & $4\ssl(4,\R)+\mathfrak{t}(1)$ & III\\
  \cline{2-4}
& $2e_{1234} +4e_{1278} +e_{3456} +2e_{5678}+b(-2e_{1346}+4e_{1678}-e_{2345}+2e_{2578})$;
$b\neq 0,1,-1$   & $2\ssl(2,\R)+\ssl(2,\C)+\mathfrak{t}(1)$ & III \\ \hline   
11 &  $e_{1256}-e_{1278}+e_{1458}+e_{1467}+e_{2358}+e_{2367}- e_{3456}+e_{3478}
  +b(e_{1234}+e_{1357}+e_{1368}+e_{2457}+e_{2468}+e_{5678})$; $b\neq 0, 1, -1$ &
  $\ssl(4,\R)$ & I \\ \cline{2-4}
& $-e_{ 1 2 5 6 }+e_{ 1 3 5 7 }+e_{ 1 4 5 8 }+e_{ 2 3 6 7 }+e_{ 2 4 6 8 }
-e_{ 3 4 7 8 }+b(e_{ 1 2 3 4 }-e_{ 1 2 7 8 }-e_{ 1 3 6 8 }+e_{ 1 4 6 7 }+e_{ 2 3 5 8 }-
e_{ 2 4 5 7 }-e_{ 3 4 5 6 }+e_{ 5 6 7 8 })$, $b\neq 0, 1, -1$ & $\mathfrak{su}(4)$ & I \\
\cline{2-4}

& $(1+b)e_{1357} +(1+b)e_{1368} +e_{1458} +e_{1467} +e_{2358} +e_{2367} 
-(1-b)e_{2457} +(1+b)e_{2468} +b(e_{1234}+e_{5678})$; $b\neq 0,1,-1$ & 
$\mathfrak{su}(2,2)$ & I \\ \cline{2-4}

& $e_{1257}-e_{1268}+e_{1458}-e_{1467}+e_{2358}-e_{2367} -e_{3457} +e_{3468}
+b(e_{1234}+e_{1356}-e_{1378}+e_{2456}-e_{2478}+e_{5678})$; $b\neq 0,1,-1$ &
$\ssl(4,\R)$ & I \\
\cline{2-4}
& $-e_{1236} -e_{1238} +4e_{1245} -4e_{1247} +e_{1257} -4e_{1268}+ 
  4e_{1345} -4e_{1347} -e_{1357} -4e_{1368} +4e_{1456} +4e_{1458} -4e_{1467} +4e_{1478} +4
  e_{1567} +4e_{1578} +e_{2346} +e_{2348} -e_{2356} +e_{2358} 
  -e_{2367} -e_{2378} +e_{2457} +4e_{2468} -e_{2568} +e_{2678} -e_{3457} +4e_{3468} +
  e_{3568} -e_{3678} +4e_{4567} +4e_{4578}+b(-4e_{1235} -4e_{1237} +4e_{1246} -4e_{1248}
  +2e_{1258} +2e_{1267} -4e_{1346} +4e_{1348} -2e_{1356} +2e_{1378} -e_{1568} -e_{1678}
  -4e_{2345} -4e_{2347} -2e_{2456} +2e_{2478} +e_{2567} -e_{2578} +2e_{3458} +2e_{3467} +
  e_{3567} -e_{3578} +e_{4568} +e_{4678})$; $b\neq 0,1,-1$ &
  $\mathfrak{su}(1,3)$ & I \\ \hline
12 & $(1+b) e_{1234} -(1+b)e_{1256} (1-3b)e_{1368} -(1+b)e_{1458}
  -(1+b)e_{2367} (1-3b)e_{2457} -(1+b)e_{3478} (1+b)e_{5678}$; $b\neq 0, 1,-1,
  \tfrac{1}{3}$ & $\ssl(4,\R)$ & II \\ \cline{2-4}
  & $-e_{1234} +e_{1256} +e_{1458} +e_{2367} +e_{3478} -e_{5678}+b(e_{1368}+e_{2457})$;
  $b\neq 0,1,-1,3$ & $\ssl(4,\R)$ & II \\ \cline{2-4}
  & $-16e_{1234}+16e_{1256}+16e_{1458}+e_{2367}+e_{3478}-e_{5678}-b(16e_{1368}+e_{2457})$;
  $b\neq 0,1,-1,3$ & $\ssl(4,\R)$ & II \\ \cline{2-4}
  & $-256e_{1258} -(768-512\sqrt{2}) e_{1267}
  +(768-512\sqrt{2}) e_{1456} -(768-512\sqrt{2})e_{1478} -
  e_{2356} +e_{2378} -e_{3458} -(3+-2\sqrt{2})e_{3467}
  +b(16e_{1256} +48e_{1278} -16e_{1458}
  -(48-32\sqrt{2})e_{1467} +(\tfrac{3}{16}+\tfrac{1}{8}\sqrt{2})e_{2358} 
  +\tfrac{1}{16}e_{2367} -\tfrac{3}{16}e_{3456} -\tfrac{1}{16}e_{3478})$;
  $b\neq 0, \pm(16-16\sqrt{2}), 48-48\sqrt{2}$ & $\ssl(4,\R)$ & III\\ \cline{2-4}

  & $(-\tfrac{17}{64}+\tfrac{3}{16}\sqrt{2})e_{1256}
  -(\tfrac{17}{64}-\tfrac{3}{16}\sqrt{2})e_{1357}
  -(\tfrac{17}{64}-\tfrac{3}{16}\sqrt{2})e_{1467} -e_{2358} -e_{2468} -e_{3478}+b(
  (-\tfrac{3}{8}+\tfrac{1}{4}\sqrt{2}) e_{1234} +(24+16\sqrt{2})
  e_{5678})$; $b\neq 0,\pm \alpha, 3\alpha$, $\alpha= -\tfrac{3}{8}+\tfrac{1}{4}
  \sqrt{2}$ & $\ssl(4,\R)$ & III \\ \hline

13 & $(1+b)e_{1234} +(1+b)e_{1256} -(1-b)e_{1278} +2e_{1458} +2e_{2367} -(1-b)e_{3456}
  +(1+b)e_{3478}+ (1+b)e_{5678}$; $b\neq 0,1,-1,-2$ & $\mathfrak{so}(2,3)+
  \mathfrak{t}(1)$ & I \\ \cline{2-4}
  & $e_{1256} -e_{1368} -2e_{1458} -(1-b)e_{1467} -(1-b)e_{2358} -2e_{2367} -e_{2457} 
    +e_{3478}+b(e_{1278} +e_{1357}+e_{2468}+e_{3456})$; $b\neq 0,1,-1,2$ &
     $\mathfrak{so}(2,3)+  \mathfrak{u}(1)$ & I \\ \cline{2-4}

  & $(1+b) e_{1234} -e_{1357} +e_{1368} +2e_{1458} +2e_{2367} +e_{2457} -e_{2468}+ 
    (1+b)e_{5678}+b(e_{1256}+e_{1278}+e_{3456}+e_{3478})$, $b\neq 0,1,-1,-2$ &
    $\mathfrak{so}(4,1)+  \mathfrak{u}(1)$ & I \\ \cline{2-4}

    & $(1+b)e_{ 1234} +e_{ 1256} +e_{1278} -e_{1357} +e_{1368} +(1+b)e_{1458} +
    (1-b)e_{1467}+(1-b)e_{ 2358} +(1+b)e_{ 2367} +e_{ 2457} -e_{ 2468} 
    +e_{3456}+e_{ 3478} +(1+b)e_{5678}$; $b\neq 0,1,-1,-2$ &
    $\mathfrak{so}(5)+  \mathfrak{u}(1)$ & I \\ \hline
    
19 & $(1+c)e_{1234}+e_{1357}+e_{2468} +(1+c)e_{5678}+b(e_{1368}+e_{2457})
    -c(e_{1256}+e_{1278}+e_{3456} +e_{3478})$; $b\neq 0,\pm 1$, $c\neq 0,-1$, 
    $\varepsilon b+2\delta c \neq 1$ where $\varepsilon,\delta \in \{1,-1\}$ &
    $\mathfrak{so}(2,3)$ & I  \\ \cline{2-4}

    & $(2+c)e_{1234} -(1+b)e_{1256} -(1+b)e_{1278} +(1-b+c) e_{1357} -(1-b-c) e_{1368}
    -(1-b-c)e_{2457} +(1-b+c)e_{2468} -(1+b)e_{3456} -(1+b)e_{3478} +(2+c)e_{5678}$;
    $b\neq 0,\pm 1$, $c\neq 0,-2$,  $\varepsilon b+\delta c \neq 1$ where
    $\varepsilon,\delta \in \{1,-1\}$ & $\mathfrak{so}(2,3)$ & I \\ \cline{2-4}

    & $(1+c)e_{1234} + e_{1357} + e_{2468} +(1+c)e_{5678}
+b(e_{1368} +e_{2457}) -c(e_{1256}+e_{1278}+e_{3456} +e_{3478})$;  $b\neq 0,\pm 1$, $c\neq 0,-1$, 
    $\varepsilon b+2\delta c \neq 1$ where $\varepsilon,\delta \in \{1,-1\}$ &
    $\mathfrak{so}(2,3)$ & I  \\ \cline{2-4}

    &  $(1-b+c)e_{1234} -(1+b+c)e_{1256} +2e_{1357} -(1-b)e_{1368} +(1+b)e_{1458}
    +(1+b)e_{2367} -(1-b)e_{2457} +2e_{2468} -(1+b+c)e_{3478} +(1-b+c)e_{5678}
    -c( e_{1278} + e_{3456} )$; $b\neq 0,\pm 1$, $c\neq 0,-2$,
    $\varepsilon b+\delta c \neq 1$ where
    $\varepsilon,\delta \in \{1,-1\}$ & $\mathfrak{so}(4,1)$ & I \\ \cline{2-4}
    
  & $-(2+c)e_{ 1 2 3 4 }+(b-1)e_{ 1 3 5 7 }-(b-1)e_{ 1 3 6 8 }+(b+1)e_{ 1 4 5 8 }+
(b+1)e_{ 1 4 6 7 } +(b+1)e_{ 2 3 5 8 }+(b+1)e_{ 2 3 6 7 }-(b-1)e_{ 2 4 5 7 }+
    (b-1)e_{ 2 4 6 8 }-(2+c)e_{ 5 6 7 8 }+c(e_{ 1 2 5 6 }+e_{ 1 2 7 8 }+e_{ 3 4 5 6 }+
    e_{ 3 4 7 8 })$; $b\neq 0,\pm 1$, $c\neq 0,-2$, $\varepsilon b+\delta c
    \neq 1$ where $\varepsilon,\delta \in \{1,-1\}$ &
    $\mathfrak{so}(5)$ & I  \\ \cline{2-4}
 & $(4+c)e_{1257} -4e_{1367} +e_{2458} -(1+c)e_{3468} +b(-4e_{1478} +e_{2356})
+c( 2e_{1234} +4e_{1268} -  e_{3457} +2e_{5678})$;  $b\neq 0,\pm 1$, $c\neq 0,-1$, 
    $\varepsilon b+2\delta c \neq 1$ where $\varepsilon,\delta \in \{1,-1\}$ &
    $\mathfrak{so}(2,3)$ & I  \\ \hline
\end{longtable}  

\begin{rmk}\label{rem:geom}
The fourvectors in Table \ref{tab:semsimreps} that have a compact stabilizer play an important role in differential geometry.
There are four of them with compact semisimple stabilizer corresponding to the lines 5, 7, 11 and 19. There is one with a compact reductive (but not semisimple) stabilizer on line 13.
The compact semisimple stabilizers are Lie algebras of Riemannian holonomy groups that occur in Berger's classification, \cite{berg}.
They define $\mathrm{Spin}(7)$, quaternion-Kaehler, Calabi-Yau, and hyperkaehler geometries respectively, all described in \cite{salamon}. We look at each of the fourvectors in somewhat more detail.

Set
\begin{align*}
\Phi_5 =& e_{ 1 2 3 4 }- e_{ 1 2 5 6 }+ e_{ 1 2 7 8 }+ e_{ 1 3 5 7 }+e_{ 1 3 6 8 }
-e_{ 1 4 5 8 }+ e_{ 1 4 6 7 }+e_{ 2 3 5 8 }- e_{ 2 3 6 7 }+ \\
& e_{ 2 4 5 7 }+e_{ 2 4 6 8 }+e_{ 3 4 5 6 }-e_{ 3 4 7 8 }+ e_{ 5 6 7 8 }
\end{align*}
 which equals $\tfrac{1}{2}( - \omega_1^2 + \omega_2^2 - \omega_3^2)$ where
$$ \omega_1 = e_{12}-e_{34} + e_{56}-e_{78},\, \omega_2 = e_{13}-e_{42} + e_{57}-e_{86},\,\omega_3 = e_{14}-e_{23} + e_{58}-e_{67}.$$
We decompose $\R^8=\R^4+\R^4$ and each $\omega_i$ is the sum of two anti-self-dual (ASD) 2-forms on each summand. The stabilizer of $\Phi_5$ is the subgroup 
$\mathrm{Spin}(7) \subset \mathrm{SO}(8)$ resulting from its spin representation. We note that $\Phi_5$ is equivalent to the fundamental 4-form $\Phi_K$ considered
in \cite[Definition 2.9]{bh89}. Indeed, the element $g=\diag(1,1,1,1,1,1,-1,-1)$ maps $\Phi_5$ to $\Phi_K$ (see \cite[(2,23)]{bh89}).

Next we look at 
\begin{align*}
\Phi_7 =& -3e_{ 1 2 3 4 }+e_{ 1 2 5 6 }+e_{ 1 2 7 8 }-e_{ 1 3 5 7 }+e_{ 1 3 6 8 }+
e_{ 1 4 5 8 }+e_{ 1 4 6 7 }+e_{ 2 3 5 8 }+e_{ 2 3 6 7 }+\\
& e_{ 2 4 5 7 }-e_{ 2 4 6 8 }+e_{ 3 4 5 6 }+e_{ 3 4 7 8 }-3e_{ 5 6 7 8 }
\end{align*}
which equals $-\tfrac{1}{2}(\omega_1^2 + \omega_2^2 + \omega_3^2)$ where
$$\omega_1 = e_{12}+e_{34} + e_{65}+e_{87},\, \omega_2 = -e_{13}-e_{42} + e_{68}-e_{75},\, \omega_3 = -e_{14}-e_{23} + e_{67}+e_{58}.$$
Here each $\omega_i$ is the sum of two ASD 2-forms as above (after applying the double transposition $(5,6)(7,8)$ which is induced by an element of $\SL(8,\R)$). The stabilizer of 
$\Phi_7$ is the subgroup $\mathrm{Sp}(2)\mathrm{Sp}(1) \cong (\mathrm{Sp}(2)\times \mathrm{Sp}(1))/Z_2 \subset \mathrm{SO}(8)$
where $Z_2=\{(1,1),(-1,-1)\}$. The subgroup $\mathrm{Sp}(1)=\mathrm{SU}(2)$ acts as $\mathrm{SO}(3)$ rotating the $\omega_i$. Linear deformations of $\Phi_7$ are investigated in \cite{cms}.

Now we consider
\begin{align*}
\Phi_{11} =& -e_{ 1 2 5 6 }+e_{ 1 3 5 7 }+e_{ 1 4 5 8 }+e_{ 2 3 6 7 }+e_{ 2 4 6 8 }
-e_{ 3 4 7 8 }+\\
& b(e_{ 1 2 3 4 }-e_{ 1 2 7 8 }-e_{ 1 3 6 8 }+e_{ 1 4 6 7 }+e_{ 2 3 5 8 }-
e_{ 2 4 5 7 }-e_{ 3 4 5 6 }+e_{ 5 6 7 8 }), \, b\neq 0, 1, -1.
\end{align*}

We have that $\Phi_{11}$ equals  $\tfrac{1}{2}\omega^2 - b \mathrm{Re}(\Omega)$ where $\omega = e_{15}-e_{37}+e_{26}-e_{48}$ is a symplectic 2-form with stabilizer 
$\mathrm{Sp}(4,\R)$ and $\Omega = (e_1+ie_5)(e_3-ie_7)(e_2+ie_6)(e_4-ie_8)$   is a complex volume form with stabilizer $\mathrm{SL}(4,\C)$.
The intersection of these two groups is $\mathrm{SU}(4)$, a subgroup of $\mathrm{Spin}(7)$. If $b=\pm 1$ then the form is conjugate to $-\Phi_5$. Indeed 
$\diag(-1,-1,-1,1,-1,1,1,1)$ (if $b=1$) and $\diag(1,-1,-1,1,1,1,1,1)$ (if $b=-1$) map $\Phi_{11}$ to $-\Phi_5$. 

Finally, let
\begin{align*}
\Phi_{19}=& -(2+c)e_{ 1 2 3 4 }+(b-1)e_{ 1 3 5 7 }-(b-1)e_{ 1 3 6 8 }+(b+1)e_{ 1 4 5 8 }+
(b+1)e_{ 1 4 6 7 } +\\
& (b+1)e_{ 2 3 5 8 }+(b+1)e_{ 2 3 6 7 }-(b-1)e_{ 2 4 5 7 }+(b-1)e_{ 2 4 6 8 }-(2+c)e_{ 5 6 7 8 }+\\
& c(e_{ 1 2 5 6 }+e_{ 1 2 7 8 }+e_{ 3 4 5 6 }+e_{ 3 4 7 8 }),\\
& b\neq 0,\pm 1,\, c\neq 0,-2,\, \varepsilon b+\delta c
    \neq 1\,\text{ where } \varepsilon,\delta \in \{1,-1\}.
\end{align*}
With $\omega_1,\omega_2,\omega_3$ as for $\Phi_7$ we have $\Phi_{19}=\tfrac{1}{2}(-c\omega_1^2
+(b-1)\omega_2^2-(b+1)\omega_3^2)$.
\end{rmk}

\section{Galois cohomology}\label{sec1}

A significant tool that we will use throughout next sections is Galois cohomology. Here we summarize the main results that we use for determining the real orbits within a complex one.

In this section we let $\cG$ be any group with a fixed automorphism $\sigma$ of order 2. An element $g\in \cG$ is called a \textit{cocycle} if $g\sigma(g)=1$. The set of cocycles is denoted by $Z^1(\cG,\sigma)$ (we omit $\sigma$ when it is clear which automorphism is being used). We can define an equivalence relation on $Z^1(\cG,\sigma)$: $g,g'\in \cG$ are equivalent if $g'=hg\sigma(h)^{-1}$ for some $h\in \cG$. The set of equivalence classes is denoted by $H^1(\cG,\sigma)$, or also by $H^1 \cG$ if it is clear which automorphism $\sigma$ is used. 
It is called the first Galois cohomology set of $\cG$. Suppose $\cG$ acts on a set $X$ which also has a map $\sigma:X\to X$ such that $\sigma(\sigma(x))=x$ for each $x\in X$. We suppose that the action of
$\cG$ is compatible with $\sigma$, that is, $\sigma(g)\sigma(x)=\sigma(gx)$ for all $g\in \cG,x\in X$. Let $x_0\in X$ be invariant under $\sigma$, i.e., $\sigma(x_0)=x_0$. Let $Z_{\cG}(x_0)=\{g\in \cG\mid gx_0=x_0\}$ be its stabilizer. We have a map $i_*:H^1(Z_{\cG}(x_0),\sigma)\to H^1(\cG,\sigma)$ by $i_*([g])=[g]$, where the latter is the class of the cocyle $g$ in $H^1(\cG,\sigma)$. The kernel of this map is $\ker(i_*)=\{[g]\in H^1(Z_{\cG}(x_0),\sigma)\mid g \text{ is equivalent to } 1 \text{ in } Z^1(\cG,\sigma)\}$. By $\cG^\sigma,X^\sigma$ we denote the elements of $\cG,X$ that are fixed under $\sigma$. We are now ready to state a crucial result (see \cite[Section I.5.4, Corollary 1 of Proposition 36]{serre} for a proof):
\begin{theorem}\label{th3}
    Let $x_0\in X^\sigma$ and 
    let $Y=\cG\cdot x_0$ be its $\cG$-orbit. Let $Y^\sigma=\{y\in Y\mid\sigma(y)=y\}$. Then the $\cG^\sigma$-orbits in $Y^\sigma$ are in bijection with $\ker(i_*)$. The bijection is given as follows. Let $[g]\in \ker(i_*)$, then there is a $h\in \cG$ such that $g=h^{-1}\sigma(h)$. The class $[g]$ corresponds to the $\cG^\sigma$-orbit of $h\cdot x_0$.
    \end{theorem}

Now we outline how to determine the first Galois cohomology set of a group, assuming we can do it for a normal subgroup and the quotient by it. We sketch the method as it is also 
contained in other sources (for example \cite{borwdg}). It will be used in Section \ref{sec3}.

Let $\cG$ be a group with involution $\sigma$. Let $\cN$ be a normal subgroup that is stable under $\sigma$. Then $\sigma$ induces an involution on the quotient
$\cC=\cG/\cN$ which we also denote by $\sigma$. Let $\pi : \cG\to \cC$ be the projection map. Hence we have an exact sequence
$$1\to \cN \labelto{i} \cG\labelto{\pi} \cC\to 1$$
where $i$ is the inclusion map. The maps $i$ yields a map $i_* : H^1(\cN,\sigma) \to H^1(\cG,\sigma)$, $i_*([n]) = [i(n)]$; similarly we have a map 
$\pi_* : H^1(\cG,\sigma) \to H^1(\cC,\sigma)$, $\pi_*([g]) = [\pi(g)]$. 

\begin{lemma}\label{lem:galc1}
We have $\im(i_*)=\ker (\pi_*)$ (where $\ker(\pi_*)=\{[g]\in H^1(\cG,\sigma) \mid \pi_*([g]) = [1] \text{ in } H^1(\cC,\sigma)\}$).
\end{lemma}

This lemma is contained in \cite[I \S 5, Proposition 38]{serre}. We also sketch its proof. First of all, $\pi(i(n)) = 1$ for all $n\in \cN$ implying that $\im(i_*)\subset \ker(\pi_*)$.
For the converse let $g\in \cG$ be such that $\pi_*([g])=[1]$; then there is a $c_1\in \cC$ such that $\pi(g)=c_1^{-1}\sigma(c)$. Write $c_1=\pi(g_1)$, $g_1\in \cG$. As 
$\pi(g)=\pi(g_1^{-1}\sigma(g_1))$ we get $g=g_1^{-1}\sigma(g_1)n_1$ for some $n_1\in \cN$. We have $\sigma(g_1)n_1\sigma(g_1)^{-1}=n_2\in \cN$ so that $g=g_1^{-1}n_2\sigma(g_1)$.
So $n_2$ is a cocycle and $i_*([n_2]) = [g]$. 

Let $\cC^\sigma = \{c\in \cC\mid \sigma(c)=c\}$. We now define a right action of $\cC^\sigma$ on $H^1(\cN,\sigma)$. 
For $c\in \cC^\sigma$ write $c=\pi(g)$ for a $g\in \cG$. For $n\in Z^1(\cN,\sigma)$ we have that $g^{-1}n\sigma(g)\in Z^1(\cN,\sigma)$. We define
\begin{equation}\label{eq1}
[n]\cdot c = [g^{-1}n\sigma(g)].
\end{equation}
Short calculations show that $[g^{-1}n\sigma(g)]\in H^1(\cN,\sigma)$ is independent of the choice of $g$ and that this indeed defines an action of $\cC^\sigma$. 

\begin{lemma}\label{lem:galc2}
Let $n_1,n_2\in Z^1(\cN,\sigma)$. Then $i_*([n_1]) = i_*([n_2])$ if and only if there exists a $c\in \cC^\sigma$ with $[n_1]\cdot c = [n_2]$.    
\end{lemma}

This lemma is contained in \cite[I \S 5, Proposition 39(ii)]{serre}). 
The proof is again a short verification. If $[n_1]$, $[n_2]$ are equivalent in $Z^1(\cG,\sigma)$ then there is a $g\in \cG$ with $g^{-1}n_1\sigma(g) = n_2$.
Applying $\pi$ we see that this entails $\pi(\sigma(g)) = \pi(g)$, so if $c=\pi(g)$ then $c\in \cC^\sigma$ and $[n_1]\cdot c=[n_2]$ in $H^1(\cN,\sigma)$. The converse
is even more straightforward. 

Now fix $[c]\in H^1(\cC,\sigma)$; we wish to compute $\pi_*^{-1}([c])$; if we can do this then we can obviously determine $H^1(\cG,\sigma)$. 
We assume that there is a $g_c\in Z^1(\cG,\sigma)$ such that $\pi(g_c) = c$. (If this is not the case then $\pi_*^{-1}([c])=\emptyset$.) In the sequel we fix such a $g_c$.

Define the involution $\tau : \cG\to \cG$ by $\tau(g) = g_c\sigma(g) g_c^{-1}$. Then for $g\in Z^1(\cG,\tau)$ we have $gg_c\in Z^1(\cG,\sigma)$. Furthermore,
if the cocycles $g_1,g_2\in Z^1(\cG,\tau)$ are equivalent, then $g_1g_c$ and $g_2g_c$ are equivalent in $Z^1(\cG,\sigma)$. So we get a map 
$\eta : H^1(\cG,\tau)\to H^1(\cG,\sigma)$. Its inverse is $[g]\mapsto [gg_c^{-1}]$, hence $\eta$ is a bijection. Similarly we have a bijection
$\xi : H^1(\cC,\tau) \to H^1(\cC,\sigma)$, $\xi([d]) = [dc]$. We get the following diagram

$$\xymatrix{ 
H^1(\cN,\tau)\ar[r]^{i_*} & H^1(\cG,\tau)\ar[r]^{\pi_*}\ar[d]_{\eta} & H^1(\cC,\tau)\ar[d]^{\xi}\\
& H^1(\cG,\sigma)\ar[r]^{\pi_*}  & H^1(\cC,\sigma)          
}$$

where the square on the right commutes.

\begin{proposition}\label{prop1}
Let $\{[n_1],\ldots,[n_r]\}$  be a set of $\cC^\tau$-orbit
representatives in $H^1(\cN,\tau)$, then  $\pi_*^{-1}([c]) = \{[n_1g_c],\ldots,[n_rg_c]\}$.
\end{proposition}

\begin{proof}
Let $M=\{[g]\in H^1(\cG,\sigma) \mid \pi_*([g]) = [c]\}$. So $\eta^{-1}(M)$ is the preimage in $H^1(\cG,\tau)$ of $\xi^{-1}([c])$. But $\xi^{-1}([c]) = [1]$, so 
$\eta^{-1}(M)=\ker(\pi_*)$ which by Lemma \ref{lem:galc1} is equal to $\im(i_*)$.  By Lemma \ref{lem:galc2} we see that $\im(i_*)$ corresponds to the $\cC^\tau$-orbits on
$H^1(\cN,\tau)$; the claim follows.
\end{proof}

We end this section with some comments on the use of Galois cohomology in the situation that we consider. For this we revert to the notation of Section \ref{sec:prelim}. 
Let $x\in \g_1^\R$ and let $\O = \G_0\cdot x$ be its $\G_0$-orbit. For the involution $\sigma$ we use complex conjugation, that is we use the maps described in 
Section \ref{sec:realf}.  We note that $H^1 \G_0$ is trivial (see \cite[Corollary III.8.26]{berhuy}). So by Theorem \ref{th3} 
there is a 1-to-1 correspondence between the elements of $H^1(Z_{\G_0}(x))$ and the $\G_0(\mathbb{R})$-orbits in $\mathcal{O}$.
In the following sections we will see how this is applied to nilpotent and semisimple orbits. In both cases there are some differences to the straightfoirward approach just outlined.
For the nilpotent orbits we prefer to work with so-called $\ssl_2$-triples and the actions of the various groups on them. For the semisimple orbits we have the added difficulty
that they come in infinite parametrized classes. Finally, we compute Galois cohomology sets with the algorithm of \cite{borwdg}. For this we need one element of each component
of $Z_{\G_0}(x)$ and one of the main problems we will be concerned with in the next sections is to determine such elements. We also remark that \cite{borwdg} contains algorithms to
determine $g\in \G_0$ such that $g^{-1}\overline{g}=c$, where $c$ is a given cocycle in $\G_0$.

\section{Real nilpotent orbits}\label{sec:nilp}

In this section we describe the classification of real nilpotent orbits.
It is based on the fact that also real nilpotent orbits are classified
by $\ssl_2$-triples. More precisely, we have the following theorem, for which
we use the notation introduced in Section \ref{sec:realf}. Note that the
map $\psi$ of Section \ref{sec:realf} yields a $\G_0$-action on $\g_0$ and $\g_1$, and a
$\G_0(\R)$ action on $\g_0^\R$ and $\g_1^\R$. 

\begin{theorem}\label{thm:sl2tr}
\begin{enumerate}
\item Let $e\in \g_1$ be nilpotent, then there are $h\in \g_0$ and $f\in \g_1$
  such that $(h,e,f)$ is an $\ssl_2$-triple. Moreover, $e,e'\in\g_1$ lying in
  $\ssl_2$-triples $(h,e,f)$, $(h',e',f')$ are $\G_0$-conjugate if and only if
  there is a $g\in \G_0$ such that $g\cdot h' = h$, $g\cdot e'=e$, $g\cdot f'=
  f$.
\item  Let $e\in \g_1^\R$ be nilpotent, then there are $h\in \g_0^\R$ and
  $f\in \g_1^\R$
  such that $(h,e,f)$ is an $\ssl_2$-triple. Moreover, $e,e'\in\g_1^\R$ lying in
  $\ssl_2$-triples $(h,e,f)$, $(h',e',f')$ are $\G_0(\R)$-conjugate if and
  only if
  there is a $g\in \G_0(\R)$ such that $g\cdot h' = h$, $g\cdot e'=e$,
  $g\cdot f'=f$.  
\end{enumerate}
\end{theorem}

\begin{proof}
The first part is well known see for example
\cite[Lemma 8.3.5, Theorem 8.3.6]{gra16}. The proof of the second part is analogous.
The existence of the $\ssl_2$-triple is shown as in \cite[Lemma 8.3.5]{gra16}; in that
proof the field is assumed to be algebraically closed, but that assumption plays no role in the proof.

Let $(h'',e,f'')$ be an $\ssl_2$-triple with $h''\in \g_0^\R$, $e,f''\in \g_1^\R$. Then the arguments
in the proofs of \cite[Proposition 8.1.3, Theorem 8.3.6]{gra16} show that there is a nilpotent
$z\in \g_0^\R$ such that with $\tau = \exp(\ad z)$ we have $\tau(h'')=h$, $\tau(e)=e$, $\tau(f'')=f$.

Suppose that there is a $g_0\in \G_0(\R)$ with $g_0(e')=e$. Then with $h''=g_0\cdot h'$, 
$f''=g_0\cdot f'$ we have that $(h'',e,f'')$ is an $\ssl_2$-triple. Let $\tau= \exp(\ad z)$ be 
as above.  Let $\hat z$ be a pre-image of $z$ under
$\varphi$ in $\ssl(8,\R)$. Then also $\hat z$ is nilpotent and hence 
$\hat\tau=\exp(\hat z)\in \G_0(\R)$. Furthermore, $\psi(\hat \tau) = \exp(\ad z)$. So with 
$g=\hat \tau g_0$ we get the second statement of the theorem.
\end{proof}

Now let $\TT$ be the set of all $\ssl_2$-triples $(h,e,f)$ with $h\in \g_0$,
$e,f\in \g_1$. Let $\TT^\R$ be the set of all $\ssl_2$-triples $(h,e,f)$
with $h\in \g_0^\R$, $e,f\in \g_1^\R$. Consider the conjugation of Section
\ref{sec:realf}. Define $\TT^\sigma = \{ (h,e,f)\in \TT \mid \sigma(h)=h,
\sigma(e)=e, \sigma(f)=f\}$. Then $\TT^\sigma = \TT^\R$. Furthermore
$\G_0$ acts on $\TT$ by $g\cdot (h,e,f) = (g\cdot h,g\cdot e, g\cdot f)$.
The action of $\G_0(\R)$ on $\TT^\R$ is defined by the same formula.
Hence Theorems \ref{th3}, \ref{thm:sl2tr} and the fact that the first Galois cohomology 
set of $\G_0$ is trivial immediately imply the following theorem. 

\begin{theorem}\label{thm:realnilp}
For $(h,e,f)\in \TT^\R$ set
$$Z_{\G_0}(h,e,f) = \{g\in \G_0 \mid g\cdot h = h , g\cdot e=e, g\cdot f = f\}.$$
Then the $\G_0(\R)$-orbits contained in $\G_0\cdot e$ are in bijection with
the Galois cohomology set $H^1 (Z_{\G_0}(h,e,f),\sigma)$. In particular, a
class $[c]$ in $H^1 (Z_{\G_0}(h,e,f),\sigma)$ corresponds to the
$\G_0(\R)$-orbit of $g\cdot e$, where $g\in \G_0$ is such that $g^{-1}\sigma(g) =
c$. 
\end{theorem}  

In order to apply this theorem we need to compute the Galois cohomology sets of the stabilizers $Z_{\G_0}(h,e,f)$. 
The algorithm of \cite{borwdg} for computing the first Galois cohomology set of a linear algebraic group requires as input a basis of 
the Lie algebra of the group, as well as one element of each component of the group. Writing $\hat\g_0= \ssl(8,\C)$ we have that the Lie algebra of $Z_{\G_0}(h,e,f)$ is
$$\z_{\hat\g_0}(h,e,f) = \{ x\in \hat\g_0 \mid x\cdot h=x\cdot e=x\cdot f=0\},$$
which we can compute by linear algebra techniques. 

So the first problem is to compute one element of each component of the stabilizers $Z_{\G_0}(h,e,f)$ where $e$ runs through a set of representatives of the nilpotent orbits.
We use the representatives listed in \cite[Table 10]{Antotrad}. Observe that in the cited list all representatives are real (which is necessary to use Theorem \ref{thm:realnilp}). 
Note that the equations $g\cdot h=h$, $g\cdot e=e$, $g\cdot f=f$ directly translate to polynomial equations on the entries of $g\in \G_0$. So we consider these 
polynomials in 64 indeterminates $a_{ij}$, $1\leq i,j\leq 8$, add the polynomial corresponding to the condition $\det(g)=1$, and compute a Gr\"obner basis (for which we used the computer algebra
system {\sc Magma}). By studying the Gr\"obner basis it is often possible to identify elements of $\G_0$ such that $Z_{\G_0}(h,e,f)$ is the union of the components 
containing these elements. With the algorithm {\sf IsEltOf} (see Section \ref{sec:cmp})
it is then possible to see which elements lie in the same component.
For the majority of the nilpotent orbits we were able to determine elements of each component of $Z_{\G_0}(h,e,f)$ from this Gr\"obner basis. 

However, due to the high complexity of the Gr\"obner basis algorithm, for some of the orbits this computation failed. 
For those cases we turned to a slightly different strategy that we describe as follows. 

Fix an $\ssl_2$-triple $t=(h,e,f)$ with $h\in \g_0$, $e,f\in \g_1$. Consider the stabilizer $Z_{\G_0}(t)$ and its Lie algebra $\mathfrak{z}_{\hat\g_0}(t)$. 
The latter is reductive (cf. \cite[Lemma 8.3.9]{gra16}), so it can be written as the direct sum of its semisimple part $\mathfrak{s}$ and its center $\mathfrak{c}$.
We assume that $\s$ is nonzero, and we fix a Cartan subalgebra of $\mathfrak{s}$, a set of simple roots in the root system and a set of canonical generators
$e_1,\ldots,e_s$, $f_1,\ldots,f_s$, $h_1,\ldots,h_s$ (see Section \ref{sec:cmp}). Relative to these choices we consider the group $O(\s)$ (see Section \ref{sec:cmp}). 
Also we remark that a $g\in Z_{\G_0}(t)$ acts on $\z_{\hat\g_0}(t)$ by the adjoint action (cf. \cite[\S 3.13]{borel}). Furthermore, since $\s$ is the derived subalgebra of
$\z_{\hat\g_0}(t)$, it is stabilized by the adjoint action of the elements of $Z_{\G_0}(t)$. 
\begin{proposition}
    Each component of $Z_{\G_0}(t)$ has an element whose action on $\mathfrak{s}$ induces an element of $O(\s)$.
\end{proposition}
\begin{proof}
    Let $S\subseteq Z^\circ_{\G_0}(t)$ be the connected subgroup with Lie algebra $\s$. The action of $S$ on $\s$ induces a morphism of algebraic groups $S\to \Aut(\s)^\circ$.
    Since the latter is the adjoint group of $\s$ this morphism is surjective (\cite[Chapter V, Corollary 1]{steinberg}). 
    Let $g\in Z_{\G_0}(t)$ and consider the action of $g$ on $\mathfrak{z}_{\g_0}(t)$. Clearly, this induces an automorphism of $\mathfrak{s}$, say $\sigma_{g}$. As 
    $\Aut(\s) = O(\s)\ltimes \Aut(s)^\circ$ we can write $\sigma_g = \sigma_\pi \varphi$ with $\varphi\in \Aut(\s)^\circ$ and $\sigma_\pi\in O(\s)$. Let $h\in S$ be a preimage of $\varphi$.
    Then $gh^{-1}$ lies in the same component as $g$ and acts on $\s$ as $\sigma_\pi$.
\end{proof}
In other words, the above proposition shows that representatives of the component group are contained in the set of solutions of the systems of equations given by $g\cdot h_i=h_{\pi(i)}$, 
$g\cdot e_i=e_{\pi(i)}$ and $g\cdot f_i=f_{\pi(i)}$ (one for each possible diagram automorphism $\pi$). These equations yield additional polynomial conditions on the 64 indeterminates $a_{ij}$. 
Write $\hat h_i = \psi^{-1}(h_i)$. Then the equations $g\cdot h_i = h_{\pi(i)}$ are equivalent to $g\hat h_i = \hat h_{\pi(i)} g$. We just take these equations and add 
those deriving from $g\cdot h=h$, $g\cdot e=e$, $g\cdot f=f$. This leads to larger systems of equations which are often easier to compute. Note that the polynomials 
corresponding to $g\cdot h_i = h_{\pi(i)}$ are linear, thus essentially reducing the number of indeterminates. 

Using these procedures we managed to determine the components of all $Z_{\G_0}(h,e,f)$ for all representatives $e$ of the nilpotent orbits. Using Theorem \ref{thm:realnilp} and the algorithms 
of  \cite{borwdg} we were able to find the nilpotent orbits of $\G_0(\R)$. 

\begin{example}
Let $e=e_{1246}+e_{1357}+e_{1258}+e_{1458}+e_{1678}$ and let $(h,e,f)$ be a homogeneous $\ssl_2$-triple containing $e$ (this is number 16 in Table \ref{tabn}).  Computing 
a Gr\"obner basis of the ideal defining $Z_{\G_0}(h,e,f)$ proved computationally too difficult. The centralizer $\z_{\hat\g_0}(h,e,f)$ is simple of type $G_2$. Since this Lie 
algebra has no diagram automorphisms every component of $Z_{\G_0}(h,e,f)$ has an element that acts as the identity on $\z_{\hat\g_0}(h,e,f)$. Adding the polynomials that are
equivalent to this condition we did get an easily computable Gr\"obner basis. It consists of $a_{ij}$ for $i\neq j$ and
$$a_{11} - a_{88},a_{22} - a_{55}a_{77}a_{88}^3,a_{33} - a_{44}a_{66}a_{88}^3,
    a_{44}a_{55} - a_{88}^2,
    a_{66}a_{77} - a_{88}^2,
    a_{88}^4 - 1.
$$
We see that the set of matrices satisfying these conditions forms a 3-dimensional variety of diagonal matrices consisting of four components. It is straightforward to write down an
element of each component. Using the algorithm indicated in Section \ref{sec:cmp} it is then a small calculation to show that these elements lie in different components of $Z_{\G_0}(h,e,f)$.
We see that the component group is cyclic of order 4. 
\end{example}

We end this section by showing how to determine the permutation induced by the map $\nu$
defined in Section \ref{sec:outer} on the sets of complex and real nilpotent orbits.

\begin{lemma}\label{lem:nucomp}
We have that $\nu$ maps a complex nilpotent $\G_0$-orbit to itself. 
\end{lemma}  

\begin{proof}
Let $e \in \bigwedge^4 \C^8$ be a representative of such an orbit, lying in the
homogeneous $\ssl_2$-triple $(h,e,f)$. We may choose $e$ such that $h$ lies in
the Cartan subalgebra of $\ssl(8,\C)$ consisting of all traceless diagonal
matrices. Then $\nu(h) = h$. So $(h,\nu(e),\nu(f))$ is a homogeneous
$\ssl_2$-triple containing $\nu(e)$. By Theorem \ref{thm:sl2tr}(1) it follows
that $e$ and $\nu(e)$ are $\G_0$-conjugate.
\end{proof}

Now we show how we can identify the real nilpotent orbit to which $\nu$ maps a
given nilpotent orbit.
Let $e\in \bigwedge^4 \R^8$ lie in the homogeneous $\ssl_2$-triple $(h,e,f)$.
First of all we note that the element $h$ spans the Lie algebra of a
$1$-dimensional torus $H$ in $\G_0$. Morever, as shown in \cite[Section 6]{borwdg}
we can compute an explicit isomorphism $\chi : \C^\times \to H$. Then for
$t\in \C^\times$ we have $\chi(t)\cdot e = t^m \cdot e$ for a certain $m\in \Z$.
We can find a $t_0\in \C$ such that $t_0^m = i$, and thus we have the element
$\chi(t_0)\in \G_0$ with $\chi(t_0)\cdot e = ie$.

Write $Z=Z_{\G_0}(h,e,f)$ and $H^1 Z = \{[c_1],\ldots,[c_r]\}$. Let $g_i\in \G_0$
be such that $g_i^{-1}\sigma(g_i)=c_i$. Then the real $\G_0(\R)$-orbits contained
in the orbit $\G_0\cdot e$ have representatives $g_1\cdot e,\ldots,g_r\cdot e$.
By Lemma \ref{lem:nucomp} it follows that $\nu$ permutes these orbits. 
Using the notation of Section \ref{sec:outer} we compute
$$\nu(g_i\cdot e) = g_0g_ig_0^{-1}g_0\cdot e = g_0g_ig_0^{-1}g_1^{-1}g_1g_0\cdot e=
g_0g_ig_0^{-1}g_1^{-1} \cdot ie =  g_0g_ig_0^{-1}g_1^{-1}\chi(t_0) \cdot e.$$
Set $\hat g_i =  g_0g_ig_0^{-1}g_1^{-1}\chi(t_0)$; then $\hat g_i\in \G_0$ and
$\nu(g_i\cdot e) = \hat g_i\cdot e$. Set $\hat c_i = \hat{g_i}^{-1}
\sigma(\hat g_i)$ then $\hat c_i$ is a cocyle in $Z$. By algorithms of
\cite{borwdg} we can find $j$ such that $\hat c_i$ and $c_j$ are equivalent
cocycles. Then $\nu$ maps the $\G_0(\R)$-orbit of $g_i\cdot e$ to the
$\G_0(\R)$-orbit of $g_j\cdot e$.

\section{Real semisimple orbits}\label{sec2}

In this section we describe our methods for classifying the real semisimple orbits. We use the complex classification by Antonyan (\cite[Table 1]{Antotrad}), in which there are 32 distinct classes of complex semisimple orbits. However, in many cases we work with a different space $\mathfrak{h}_{H_i}$, see Table \ref{tab1}.

\subsection{Determining the components of a stabilizer}

The first step is to compute the component group of the
stabilizer $Z_{\G_0}(p)$, where $p\in \h_{H_i}^\circ$. By Theorem \ref{thm:cent} all elements of $\h_{H_i}^\circ$ have the same stabilizer, so it does not matter which $p$ in that set we
choose. Computing representatives of the elements of the component group works exactly as in Section 2 for most cases, that is, we use direct Gr\"obner basis computations. However, there are some exceptions that require a different approach. In these cases the standard procedure does not work because the centralizer is a torus, so we cannot work with diagram automorphisms to simplify the computations.
The classes for which this happens are the ones numbered 24, 28, 27, 30, 31, 32. Here we show how to deal with the first five, postponing the last case to Section \ref{sec3}.

Let $p\in \h_{H_i}^\circ$ and set $\c=\z_{\hat{\g}_0}(p)$, so that $\Lie(Z_{\G_0}(p))=\c$, and $\mathfrak{c'}=\mathfrak{z}_{\mathfrak{\hat{g}_0}}(\mathfrak{c})$. Then we have the following.
   \begin{proposition}\label{p1}
        With the notation above we have that $g\in Z_{\G_0}(p)$ normalizes both $\c$ and $\c'$.
    \end{proposition}
    \begin{proof}
        At first, let $g\in Z_{\G_0}(p)$ and observe that, if $c\in\mathfrak{c}$, then $gcg^{-1}\in\mf{c}$, as it is just the result of the action of $Z_{\G_0}(p)$ on its Lie algebra. Moreover, if $t\in\mathfrak{c'}$ and $c\in\mathfrak{c}$, we have $[gtg^{-1},c]=g[t,g^{-1}cg]g^{-1}=0$, which means that $g$ normalizes $\mathfrak{c'}$.
    \end{proof}

Let $p$ lie in $\h_{H_i}^\circ$ with $i=24$ or $i=28$. Then a direct computation shows that $\c'$ is a Cartan subalgebra of $\hat\g_0$. In fact, it is the standard Cartan subalgebra
consisting of diagonal matrices. 
    Since $\mf{c'}$ is a Cartan subalgebra of $\mf{\hat{g}}_0$, $N_{\G_0}(\mathfrak{c'})/Z_{\G_0}(\mathfrak{c'})\cong W_0$, the Weyl group of $\mf{\hat{g}}_0$. Let $\beta_1,\dots,\beta_7$ be a choice of simple roots of the root system of $\mf{\hat{g}}_0$ with respect to $\mf{c'}$. Fix a corresponding canonical generating set $x_1,\dots,x_7$, $y_1,\dots,y_7$, $h_1,\ldots,h_7$. Set $s_i=\text{exp}(x_i)\text{exp}(-y_i)\text{exp}(x_i)$; then $s_i\in N_{\G_0}(\mf{c'})$ induces the simple reflection corresponding to $\beta_i$. For $w\in W_0$ we fix a reduced expression as product of simple reflections and let $\dot{w}$ be the analogous product of the $s_i$. Proposition \ref{p1} assures that each element of $Z_{\G_0}(p)$ lies in $N_{\G_0}(\c')$ and hence can be written as $\dot{w}z$ with $z\in Z_{\G_0}(\mf{c'})$. Since $\c'$ is the standard Cartan subalgebra of $\hat\g_0$ consisting of traceless diagonal matrices we have that $Z_{\G_0}(\c')$ is the maximal torus of $\G_0$ consisting of 
    diagonal matrices. Now let $w\in W_0$ and $z\in Z_{\G_0}(\c')$ be such that $\dot{w}z\cdot p = p$. By Proposition \ref{p1} $\dot{w}z$ normalizes $\c$. Since $\c$ is abelian we have $\c\subset\c'$
    so that also $z$ normalizes (in fact, centralizes) $\c$. It follows that $\dot{w}$ normalizes $\c$. 
    We remark that $|W_0|=8!$. We have written a little program checking for which $\dot{w}$ we have that $\dot{w}$ normalizes $\mf{c}$. It turns out that there are 48 such elements. 
    Now we fix one of these 48 elements $\dot{w}$ and try to determine all $z\in Z_{\G_0}(\c')$ such that $\dot{w}z\cdot p=p$. We claim that if $z_1,z_2\in Z_{\G_0}(\c')$ satisfy this, then 
    $(\dot{w}z_1)^{-1}(\dot{w}z_2)$ lies in the identity component of $Z_{\G_0}(p)$. Indeed, $(\dot{w}z_1)^{-1}(\dot{w}z_2)=z_1^{-1}z_2\in Z_{\G_0}(p)$, hence, in particular, $z_1^{-1}z_2\in  Z_{\G_0}(p)\cap Z_{\G_0}(\mf{c'})=\{h\in Z_{\G_0}(\mf{c'})\mid hp=p\}$. It is straightforward to compute the defining equations of the latter variety. An immediate computation in {\sc Magma} shows that the associated ideal is prime, so that the variety $Z_{\G_0}(p)\cap Z_{\G_0}(\mf{c'})$ is connected and hence contained in the identity component of $Z_{\G_0}(p)$. 
    To check whether a solution exists we write $z$ as a diagonal matrix with indeterminates on the diagonal and consider the polynomial equations equivalent to 
    $(\dot{w}z)\cdot p=p$. We compute a Gr\"obner basis and, if that is nontrivial, a solution exists. With this strategy, we get a set of 8 elements, say $\dot{w}_1,\dots,\dot{w}_8$ such that the equations have solutions. However, it turns out that $\dot{w}_i\cdot p=p$ for $i=1,\dots,8$. The set of the classes represented by the $\dot{w}_i's$ is in fact a group of order 8. We can then conclude just observing that the $\dot{w}_i$ are distinct modulo $Z_{\G_0}(\mf{p})^\circ$ because they are distinct modulo $Z_{\G_0}(\mf{c'})$ and $Z_{\G_0}(\mf{p})^\circ\subseteq Z_{\G_0}(\mf{c'})$. 
    So we get a component group of order 8.

Now let $p$ lie in $\h_{H_i}^\circ$ with $i=27, 30$ or $31$. Again we have that $\c$ is a torus (of dimension 3,2,1 respectively),
but this time $\mathfrak{c'}=2A_3+\mathfrak{c}$ (for $i=31$), or $\mathfrak{c'}=4A_1+\mathfrak{t''}$, where $\mathfrak{c}\subseteq \mathfrak{t''}$ and $\mathfrak{t''}$ is a torus of dimension $3$ (for $i=27,30$). From now on we fix $i=31$; the procedure is similiar for $p$ in the other two subalgebras. First observe that there is a Cartan subalgebra of $\hat\g_0$, say $\mathfrak{h'}$, such that $\mathfrak{c}\subseteq\mathfrak{h'}\subseteq\mathfrak{c'}$ (clearly, $\mathfrak{h'}$ is also a Cartan subalgebra of $\mathfrak{c'}$). Again this implies that $N_{\G_0}(\mathfrak{h'})/Z_{\G_0}(\mathfrak{h'})\cong W_0$. In the sequel we write $H'=Z_{\G_0}(\h')$; then $H'$ is the connected torus whose Lie algebra is $\h'$. This also implies that 
$H'\subset Z_{\G_0}(\c)$. For each $w\in W_0$ we fix, in the same way as above, a $\dot{w}\in N_{\G_0}(\h')$ inducing $w$. We also observe that, as $\c$ is a torus, we have that the
group $Z_{\G_0}(\c)$ is connected (\cite[Corollary 3.11]{steinberg75}). 

Let $g\in Z_{\G_0}(p)$; by Proposition \ref{p1} $g$ normalizes $\mathfrak{c}$ and $\mathfrak{c'}$, so that $g\mathfrak{h'}g^{-1}\subseteq\mathfrak{c'}$. Then, since $g\mathfrak{h'}g^{-1}$ is another Cartan subalgebra of $\mathfrak{c'}$, it has to be $Z_{\G_0}(\c)$-conjugate to $\h'$; in other words, there is an
$h\in Z_{\G_0}(\c)$ with $h(g\mathfrak{h'}g^{-1})h^{-1}=\mathfrak{h'}$. Hence $hg\in N_{\G_0}(\mathfrak{h'})$ so that $hg=t\dot{w}$ for some $w\in W_0$ and $t\in H'$. As $h\in Z_{\G_0}(\mathfrak{c})$ we have $(hg)\mathfrak{c}(hg)^{-1}=\mathfrak{c}$, hence $t\dot{w}\mathfrak{c}\dot{w}^{-1}t^{-1}=\mathfrak{c}$, which is equivalent to $\dot{w}\mathfrak{c}\dot{w}^{-1}=\mathfrak{c}$. As in the previous paragraph, we can then compute the subset $W^*$ of $W_0$ consisting of all $w\in W_0$ such that $\dot{w}$ normalizes $\mf{c}$. It turns out that $|W^*|=1152$. Summarizing: for any $g\in Z_{\G_0}(p)$ we have $g=k\dot{w}$, where $k\in Z_{\G_0}(\mathfrak{c})$ and $w\in W^*$. With the algorithm {\sf IsEltOf} we can check for each pair $w_1,w_2\in W^*$ whether $\dot{w}_1$, $\dot{w}_2$ are equal modulo $Z_{\G_0}(\c)$, that is, whether $\dot{w}_1\dot{w}_2^{-1}\in Z_{\G_0}(\c)$. 
It turns out that $W^*$splits into two disjoint classes, $W^*=W^*_0\cupdot W^*_1$ where  $|W^*_i|=512$ and, if $w_1,w_2\in W^*_i$ then $\dot{w}_1$ and $\dot{w}_2$ are equal modulo $Z_{\G_0}(\mf{c})$. 
For $w\in W^*$ set $X_w = \{ k\dot{w} \mid k\in Z_{\G_0}(\c), (k\dot{w})\cdot p = p\}$. It is straightforward to see that if $\dot{w}_1$, $\dot{w}_2$ are equal modulo $Z_{\G_0}(\c)$ then 
$X_{w_1} = X_{w_2}$. The set $W_0^*$ contains the identity and let $w_1$ be a fixed element of $W_1^*$. Then it follows that $Z_{\G_0}(p) = X_1 \cup X_{w_1}$. Therefore, we deduce that it is sufficient to solve the two systems of equations $kp=p$ and $k\dot{w}_1p=p$ for $k\in Z_{\G_0}(\c)$. By inspecting the Lie algebra $\c'$ of $Z_{\G_0}(\c)$ we see that the latter group $Z_{\G_0}(\c) = 
(\SL(4,\C)\times \SL(4,\C))\cdot T_1$ where each $\SL(4,\C)$ occupies a $4\times 4$-block in the matrix of an element of $Z_{\G_0}(\c)$ and the $T_1$ is a 1-dimensional torus with Lie 
algebra $\c$. So we can write a general element $k$ of $Z_{\G_0}(\c)$ as a matrix depending on 33 parameters $a_1,\ldots,a_{32}, s_1$.
Then, we compute $kp$ and $k\dot{w}_1p$, two elements in $\bigwedge^4\C^8$ of the form $\sum_{1\le i<j<k<l\le 8} \gamma^{ijkl}(a_1;\dots;a_{32};s_1)e_{ijkl}$; now, if $p=\sum_{1\le i<j<k<l\le 8} q^{ijkl}e_{ijkl}$, we deduce the equations $\gamma^{ijkl}(a_1;\dots;a_{32};s_1)=q^{ijkl}$ (at this point, we also add the equations expressing that the determinants of the two $4\times 4$-blocks is 1). In particular, we will have two large systems of equations that we can reduce computing their Gr\"obner bases. Now, observe that by the construction above the number of solutions is not finite. Indeed, $T_1=Z_{\G_0}(p)^\circ\subset Z_{\G_0}(c)$ (as $\c\subset \c'$) so if $g\in T_1$ then $g$ is a solution of the first set of equations and $g\dot{w}_1^{-1}$ is a solution of the second set of 
equations. However, as $T_1=Z_{\G_0}(p)^\circ$ we can fix the parameter $s_1$ arbitrarily, after which the solution sets are finite. It turns out that the first system gives $64$ solutions, the second one $128$.
After that, we can just conclude computing the group generated by those solutions and then computing it modulo the identity component $Z_{\G_0}(\mathfrak{c})^\circ$ using again the algorithm 
{\sf IsEltOf}. In the end we find a component group of 64 elements.

\subsection{Determining the semisimple orbits}

Now we turn to the problem to determine the semisimple $\G_0(\R)$-orbits. We fix a subgroup $H$ of $W$ and write $\A=\h_{H}^\circ$. We wish to determine the $\G_0(\R)$-orbits 
$\G_0(\R)\cdot q$ such that the complex orbit $\G_0\cdot q$ has a point $q'$ in $\A$. The problem is that $q'$ may be non-real, but we need the real point $q$ in order to use
Galois cohomology. We use the theorem below to deal with this. In order to formulate it we first need some notation. Set
\begin{align*}
N_{\G_0}(\A) &= \{ g\in \G_0\mid g\cdot p\in \A \text{ for all } p\in \A\}\\
Z_{\G_0}(\A) &= \{ g\in \G_0\mid g\cdot p=p \text{ for all } p\in \A\}.
\end{align*}
Let $N_W(H)$ denote the normalizer of $H$ in $W$. Let $g\in N_{\G_0}(\A)$. Then there is a $w_g\in N_W(H)$ such that $g\cdot p = w_g\cdot p$ for all $p\in \A$ (\cite[Lemma 4.13]{gl24}).
We set $\Gamma_H = N_W(H)/H$ and define a map $\varphi : N_{\G_0}(\A)\to \Gamma_H$ by $\varphi(g)=w_gH$. By \cite[Lemma 4.13]{gl24} this is well defined and a surjective group homomorphism
with kernel $Z_{\G_0}(\A)$. Now we can state the theorem; for the proof we refer to \cite[Proposition 4.14]{gl24}.

\begin{theorem}\label{th1}
    Let $H^1(\Gamma_H)=\{[\gamma_1],\dots,[\gamma_r]\}$. Then: 
    \begin{enumerate}
        \item Let $\O$ be a $\G_0$-orbit with a representative in $\A$. If $\O$ has real points then there is 
        $q\in\A\cap \O$ and  $i$ with $1\leq i\leq r$ such that $\overline{q}=\gamma_i^{-1}q$.
        \item Let $q\in\A$ be such that $\overline{q}=\gamma_i^{-1}q$ for some $1\le i\le r$. Then the orbit of $q$ has real points if and only if there exists a cocycle $c_i\in Z^1(N_{\G_0}(\A))$ 
        such that $\varphi(c_i)=\gamma_i$. If the latter holds, then $gq$ is a real point of the orbit of $q$, where $g\in \G_0$ is such that $g^{-1}\overline{g}=c_i$.
    \end{enumerate}
\end{theorem}

We note that the element $g\in \G_0$ in the second statement exists because $H^1 \G_0$ is trivial.

In order to use this theorem we have to do a few things:
\begin{enumerate}
\item Compute $H^1 \Gamma_H = \{[\gamma_1],\ldots,[\gamma_r]\}$.
\item For $1\leq i\leq r$ describe the sets $\A_{\gamma_i}=\{q\in \A \mid \overline{q} = \gamma_i^{-1}q
\}$.
\item For each $i$ find a cocycle $c_i\in Z^1 N_{\G_0}(\A)$ with $\varphi(c_i) = \gamma_i$, or decide that no such
cocycle exists.
\item Compute a $g_{\gamma_i}\in \G_0$ such that $g_{\gamma_i}^{-1}\overline{g}_{\gamma_i} = c_i$ and describe the set  $g_{\gamma_i}\cdot \A_{\gamma_i}$. 
\end{enumerate}

Write, as in the theorem, $H^1(\Gamma_H)=\{[\gamma_1],\dots,[\gamma_r]\}$. Then every orbit with real points and a representative in $\A$ has a representative
in $\A_{\gamma_i}$ for some $i$. Furthermore, if $q\in \A_{\gamma_i}$ then $g_{\gamma_i}\cdot q$ is a real point of the orbit of $q$.

We first comment on the steps above. Then we say how we find the real orbits. The first step
is easy: $\Gamma_H$ is a finite group of size bounded by the size of $W$ (which is 2903040).
So we can compute $H^1 \Gamma_H$ by brute force (first listing all cocycles in $\Gamma_H$, then
dividing them into equivalence classes). For the second step we note that $\A$ is an open set
in the subspace $\h_H$. 
Let $q_1,\ldots,q_k$ be a basis of $\h_H$; it is given in Table \ref{tab1}. From that table 
we see that the basis consists of real elements. We now view $\h_H$ as a vector space over
$\R$ of dimension $2k$. We note that the matrix of $\gamma_i$ with respect to the given basis of $\h_H$ has real coefficients.
So we can write $\gamma_i^{-1} \cdot q_j= \sum_{s=1}^k \delta_{sj} q_s$ with $\delta_{sj}\in \R$. Then 
the set of $q\in \h_H$ with $\bar q = \gamma_i^{-1} q$ consists of the linear combinations
$\sum_{j=1}^k (a_j+ib_j) q_j$ where $a_j,b_j\in \R$ and the vectors with coordinates $a_1,\ldots,a_k$ and $b_1,\ldots,b_k$ are
eigenvectors with eigenvalue 1 and -1 respectively of the matrix $(\delta_{sj})$. 
Since $(\delta_{sj})$ is real and $\gamma_i$ is a cocycle it
follows that the matrix is of order 2. Hence the solution space has $\R$-dimension $k$. It also follows that 
$\gamma_i=\gamma_i^{-1}$ and that the solution space is closed under complex conjugation. So we 
can find a basis $q_1',\ldots,q_k'$ of the real space $\{ q\in \h_H \mid \bar q = \gamma_i^{-1} q\}$ and 
$\A_{\gamma_i}$ is an open set in it. 

For the third step we note that even if the cocycle $c_i$ exists it may not be easy to find it because the set $\varphi^{-1}(\gamma_i)$ is not easy to describe.
We used a few ad-hoc methods that are not guaranteed to succeed, but for us they always did. We also remark that if such methods do not succeed in finding
$c_i$ then it is possible to use the more systematic methods of \cite{bor21}. For our first method we let $\delta : N_{\G_0}(\h) \to W$ be the usual 
projection. Let $w_{i}\in N_W(H)$ be such that $w_{i}H = \gamma_i$. Then $\delta^{-1}(w_{i}) \subset \varphi^{-1}(\gamma_i)$. Indeed, let $g\in N_{\G_0}(\h)$ be such that
$\delta(g) = w_{i}$; as $w_i\in N_W(H)$ by \eqref{eq:wmap} we see that $g$ stabilizes $\h_H^\circ$ and acts on it in the same way as $w_i$, so that $g\in \varphi^{-1}(\gamma_i)$. 
Now suppose that $[\gamma_i]$ has $m$ elements $\gamma_i^1,\ldots,\gamma_i^m$ (observe that we computed these in the first step). Let $w_i^j\in W$ be such that $\gamma_i^j = w_i^jH$.
Then for each $i,j$ and for each $w'\in w_i^jH$ we computed an element of $\delta^{-1}(w')$. In the majority of cases at least one of these elements turned out to be a cocycle.
For the remaining cases we succeeded with brute force: taking a general matrix in $g\in\G_0$ whose entries are indeterminates and considering the 
polynomial equations that are equivalent to $g\cdot q_i=\gamma_i\cdot q_i$ for $1\leq i\leq r$. We computed a Gr\"obner basis of the ideal
generated by these polynomials, and looked for a cocycle in the solution set. 

The fourth step can be tackled with algorithms given in \cite{borwdg}. However, in our case the following simple method always worked. 
Let $c'_j=P^{-1}c_jP$ be a diagonal matrix, where $P$ is the invertible matrix whose columns are the eigenvectors of $c_j$. In all cases $P$ turns to be a real matrix, so that we can proceed in the following way. It is a straightforward calculation to find a diagonal matrix $s\in \G_0$ such that $s^{-1}\overline{s}=c'_j$ and such that the determinant of $sP^{-1}$ equals 1, and then $(sP^{-1})^{-1}\overline{sP^{-1}}=Ps^{-1}\overline{s}P^{-1}=Pc'_jP^{-1}=c_j$. So we can set $g_{\gamma_i} = sP^{-1}$. Furthermore $g_{\gamma_i}\cdot \A_{\gamma_i}$ is an open set 
in the subspace spanned by $g_{\gamma_i}\cdot q_1',\ldots, g_{\gamma_i}\cdot q_k'$. 

Let $p\in g_{\gamma_i}\cdot \A_{\gamma_i}$; then $p$ is real. By Theorem \ref{th3} the $\G_0(\R)$-orbits contained in the $\G_0$-orbit of $p$ are in bijection with 
$H^1 Z_{\G_0}(p)$. Furthermore, all elements $q\in \A_{\gamma_i}$ have the same stabilizer in $\G_0$ (Theorem \ref{thm:cent}). Denote this stabilizer by $Z$. Above we have commented on how we determined it. Then 
$Z_{\G_0}(p) = g_{\gamma_i}Zg_{\gamma_i}^{-1}$. 

\begin{example}
Let $H=H_{10}$ then from Table \ref{tab1} we see that $\h_H$ is spanned by $p_1,p_2$. Furthermore, $\A=\h_H^\circ$ consists of $xp_1+yp_2$ with $xy(x^2-y^2)\neq 0$.   
The group $\Gamma_H$ is of order 8 and the actions of two generators on $\h_H$ is given by 
$$\begin{pmatrix} -1 & 0\\ 0&1\end{pmatrix},\, \begin{pmatrix}0&1\\-1&0\end{pmatrix}. $$
These two matrices are real, so the first Galois cohomology of $\Gamma_H$ consists of the conjugacy classes of elements of order dividing 2. Representatives of these classes are
$$\begin{pmatrix} 1&0\\0&1\end{pmatrix},\, \begin{pmatrix} -1&0\\0&-1\end{pmatrix},\, \begin{pmatrix} -1&0\\0&1 \end{pmatrix},\, \begin{pmatrix} 0&-1\\ -1&0\end{pmatrix}.$$
Denoting these elements by $\gamma_1,\ldots,\gamma_4$ we have that the real spaces $\{ q\in \h_H \mid \bar q = \gamma_i^{-1} q\}$ are spanned by respectively $\{p_1,p_2\}$, 
$\{ip_1,ip_2\}$, $\{ip_1,p_2\}$, $\{p_1-p_2,i(p_1+p_2)\}$. 

Let us consider the third case. We have that $\A_{\gamma_3}$ consists of $aip_1+bp_2$ with $a,b\in \R$ and $ab(a^2+b^2)\neq 0$. Since $a,b\in \R$ the latter is equivalent to $ab\neq 0$.
Set 
$$c_3 = \diag(1,-1,1,1,1,1,1,-1),\, g_3 = \diag(1,i,1,1,1,1,1,-i).$$
Then $c_3,g_3\in \G_0$, $c_3$ is a cocycle inducing the action of $\gamma_3$ on $\h_H$ and $g_3^{-1} \bar g_3 = c_3$. In fact, 
$$g_3\cdot ip_1 = -e_{1234}+e_{5678},\, g_3\cdot p_2 = e_{1357}+e_{2468}.$$
Write $q_1'=e_{1234}-e_{5678}$, $q_2'=e_{1357}+e_{2468}$. Then $g_3\cdot \A_{\gamma_3}$ consists of $aq_1'+bq_2'$ with $a,b\in \R$ and $ab\neq 0$.
Let $Z = Z_{\G_0}(q)$ with $q\in \A_{\gamma_3}$. Then $Z$ is of type $4A_1+\mathfrak{t}(1)$ and has four components. We have that $Z_{\G_0}(q')$ with $q'\in g_3\cdot \A_{\gamma_3}$ is 
$g_3Zg_3^{-1}$. We computed its Galois cohomology with the algorithms of \cite{borwdg}. The first Galois cohomology set has three elements. One is the class
of the identity leading to the orbits with representatives $aq_1'+bq_2'$ with $ab\neq 0$. Set 
$$h_2= \begin{pmatrix}
\tfrac{5}{4}+\tfrac{3}{4}i & 0 & 0 & 0 & \tfrac{3}{4}+\tfrac{5}{4}i & 0 & 0 & 0 \\ 
  0 & \tfrac{1}{2}-\tfrac{1}{2}i & 0 & 0 & 0 & -\tfrac{1}{2}-\tfrac{1}{2}i & 0 & 0 \\ 
  0 & 0 & -\tfrac{1}{2}i & 0 & 0 & 0 & -\tfrac{1}{2} & 0 \\ 
  0 & 0 & 0 & \tfrac{1}{2}-\tfrac{1}{2}i & 0 & 0 & 0 & -\tfrac{1}{2}-\tfrac{1}{2}i \\ 
  \tfrac{3}{4}+\tfrac{5}{4}i & 0 & 0 & 0 & \tfrac{5}{4}+\tfrac{3}{4}i & 0 & 0 & 0 \\ 
  0 & -\tfrac{1}{2}-\tfrac{1}{2}i & 0 & 0 & 0 & \tfrac{1}{2}-\tfrac{1}{2}i & 0 & 0 \\ 
  0 & 0 & -\tfrac{1}{2} & 0 & 0 & 0 & -\tfrac{1}{2}i & 0 \\ 
  0 & 0 & 0 & -\tfrac{1}{2}-\tfrac{1}{2}i & 0 & 0 & 0 & \tfrac{1}{2}-\tfrac{1}{2}i \end{pmatrix}.$$
The second element of  $H^1 Z_{\G_0}(q')$ is $h_2^{-1}\bar h_2$. Set $q_1''= -8h_2 \cdot q_1'$ and $q_2''=-h_2\cdot q_2'$ (here the coefficients $-8$ and $-1$ have been added 
to make the formulas below come out a bit nicer). Then
\begin{align*}
q_1'' &= 5e_{1234}+3e_{1238}-3e_{1247}-5e_{1278}+3e_{1346}+5e_{1368}-5e_{1467}-3e_{1678}-3e_{2345}\\
&\phantom{=} +5e_{2358}-5e_{2457}+3e_{2578}+5e_{3456}-3e_{3568}+3e_{4567}-5e_{5678}\\
q_2'' &= e_{1357}+e_{2468}.
\end{align*}
The third element of the first Galois cohomology set is $h_3^{-1}\bar h_3$. We do not report $h_3$ here, but if we set $q_1''' = h_3\cdot q_1'$ and $q_2'''= 16h_3\cdot q_2'$ then
\begin{align*}
q_1''' &=   4e_{1234}-\tfrac{1}{4} e_{5678} \\
q_2''' &=  17e_{1357}-15e_{1358}-15e_{1367}  -17e_{1368}-15e_{1457}-17e_{1458}  -17e_{1467}+15e_{1468}\\
&\phantom{=} +15e_{2357} -17e_{2358}-17e_{2367}-15e_{2368}-17e_{2457}-15e_{2458}-15e_{2467}+17e_{2468}.
\end{align*}
We conclude that the $\G_0$-orbit of $aip_1+bp_2$ with $a,b\in \R$, $ab\neq 0$ has three real $\G_0(\R)$-orbits with representatives $-aq_1'+bq_2'$, $\tfrac{1}{8}aq_1''-bq_2''$, $-aq_1'''
+\tfrac{1}{16}bq_3'''$. In particular, the element $\gamma_3$ corresponds to three classes of real orbits and the elements of each class share the same stabilizer in 
$\G_0(\R)$. The elements $\gamma_1$, $\gamma_2$, $\gamma_4$ yield 6,6,1 classes respectively. We conclude that the single complex class $\h_H^\circ$ gives rise to 16 classes
of real orbits. 
\end{example}

\begin{rmk}\label{rem:phinusemsim}
Now we show how to determine the permutations of the semisimple orbits induced by the maps
$\nu$ and $\varphi$ from Section \ref{sec:outer}. We use the notation from
above and consider the semisimple orbits that lie in the complex orbits of the
sets $g_{\gamma_j}\cdot \A_{\gamma_j}$ for $1\leq j\leq r$.

As seen in Section \ref{sec:outer} $\nu$ maps the real orbits contained in
the $\G_0$-orbit of $v\in \bigwedge^4 \C^8$ to the real orbits contained in
the $\G_0$-orbit of $iv$. 
It is often the case that $\A_{\gamma_l} = i \A_{\gamma_j}$ for certain $l,j$.
In that case we have that the map $\nu$ from Section \ref{sec:outer}
maps the $\SL(8,\C)$-orbits of an element in $\A_{\gamma_j}$ to the
$\SL(8,\C)$-orbits of an element of $\A_{\gamma_l}$.

As $\varphi$ is the identity on the Cartan subspace $\h$, 
it follows that $\varphi$ maps each semisimple complex
orbit to itself. Consider an element of a set $\A_{\gamma_j}$. We can compute
the action of 
$\varphi$ on the real orbits contained in the complex orbit of a $u\in \A_{\gamma_j}$ as follows.
Let $h=g_{\gamma_j}$ then $h\cdot u$ is real. Let $H^1 Z_{\G_0}(h\cdot u)=\{[c_1],\ldots,[c_r]\}$ with $c_i\in \G_0$ such that $c_i\sigma(c_i) = 1$. Let $g_i\in \G_0$ be 
such that $g_i^{-1}\sigma(g_i) = c_i$; then $g_1h\cdot u,\ldots,g_rh\cdot u$ are representatives 
of the real $\G_0(\R)$-orbits in the $\G_0$-orbit of $u$. Applying $\varphi$ we get 
$\varphi(g_kh\cdot u) = \varphi(g_kh) h^{-1} \cdot hu$. Write $a_k = \varphi(g_kh)h^{-1}$ and
$b_k = a_k^{-1}\sigma(a_k)$. Then the orbit of $\varphi(g_kh\cdot u)$ corresponds to the cocycle
$b_k$. We can compute $l$ such that $b_k$ is equivalent to $c_l$ in $Z^1 Z_{\G_0}(h\cdot u)$.
Then $\varphi$ maps the orbit with representative $g_kh\cdot u$ to the orbit with representative
$g_lh\cdot u$.
\end{rmk}

\begin{rmk}
The numbers contained in the sixth column of Table \ref{tab1} often display
a certain symmetry. This is explained by the previous remark.
Consider the sets
$\A_{\gamma_1},\ldots, \A_{\gamma_r}$. Then the list in the sixth column is
$(k_1,\ldots,k_r)$. This means that the $\G_0$-orbit of an element in
$\A_{\gamma_j}$ contains $k_j$ real $\SL(8,\R)$-orbits for $1\leq j\leq r$.
If $\A_{\gamma_l} = i \A_{\gamma_j}$ for certain $l,j$
then $\nu$ maps the $\G_0$-orbit of an element in $\A_{\gamma_j}$ to the
$\G_0$-orbit of an element of $\A_{\gamma_l}$. So since $\nu$ also maps
real orbits to real orbits we have $k_l=k_j$. In many cases  we have that
multiplication by $i$ permutes the sets $\A_{\gamma_j}$. So if that happens we
see that the list $(k_1,\ldots,k_r)$ has a nontrivial symmetry of order 2. 
\end{rmk}

\section{The Cartan subspaces in $\g_1^\R$}\label{sec3}

A Cartan subspace in $\g_1^\R$ is a maximal abelian subspace consisting of semisimple elements. 
In this section we classify the Cartan subspaces in $\g_1^\R$. Let $\c$ be a subspace of 
$\g_1^\R$ and let $\c^\C$ be the subspace of $\g_1$ spanned by $\c$.
From \cite{gl24} we recall that $\c$ is a Cartan subspace of $\g_1^\R$ if and only if 
$\c^\C$ is a Cartan subspace of $\g_1$. 

For the remainder of this section we set $Z=Z_{\G_0}(\h)$ and $N=N_{\G_0}(\h)$. By 
\cite[Theorem 4.7]{gl24} and the fact that $H^1 \G_0 = 1$ we have the following theorem.

\begin{theorem}
There is a bijection between the elements of $H^1 N$ and the Cartan subspaces in $\g_1^\R$. 
More precisely, let $[c]\in H^1 N$ and let $g\in \G_0$ be such that $g^{-1}\bar g = c$; then
$[c]$ corresponds to the Cartan subspace $g\cdot \h$ in $\g_1^\R$.
\end{theorem}

It follows that we need to determine $H^1 N$. The group $N$ is finite but very large (below we show that it is of order $2^8\cdot 2903040$), so we cannot compute $H^1N$ by brute force.
Therefore we use the strategy outlined in  Section \ref{sec1} using the exact sequence
$$ 1\to Z \to N \labelto{j} W\to 1$$
where $W=N/Z$ which is equal to the Weyl group of the root system of $\g$ with respect to $\h$. So we know $W$, and now we outline how we computed $Z$. 
First we observe that $G_0 \subset G^\theta = \{ g\in G \mid g\theta = \theta g\}$ and both groups have the same Lie algebra (which is $\g_0$).
Set $H= Z_G(\h)$ then $H$ is a connected torus and since its Lie algebra is $\h$, which is contained in $\g_1$, the intersection of $H$ with $G^\theta$
is a finite group. We first compute this intersection. Let $x_1,\ldots,x_7$ be the positive simple root vectors in a fixed canonical generating set of $\g$ relative to the 
root system of $\g$ with respect to $\h$ (cf. Section \ref{sec:cmp}). Define the elements $w_i(t) = \exp( t\ad x_i ) \exp(-t^{-1}\ad x_i)\exp( t\ad x_i )$, $h_i(t)=w_i(t)w_i(1)^{-1}$
where $t\in \C^*$. Then $H=\{h_1(t_1)\cdot\dots\cdot h_7(t_7)\mid t_i \in\mathbb{C}^*\}$ (cf. \cite[Chapter 3, Lemma 28]{steinberg}).  So we can write an arbitrary element of $H$ as a $133\times 133$ matrix depending on seven parameters $t_1,\dots,t_7$.
Then the condition that an element of $H$ commutes with $\theta$ translates to polynomial equations in $t_1,\ldots,t_7$. We treat $t_1,\ldots,t_7$ as indeterminates and solve these equations by computing a Gr\"obner basis and
find the set $H\cap G^\theta$ consisting of 128 elements. Now with the algorithm {\sf IsEltOf} (see Section \ref{sec:cmp}) we find the subset of elements that lie in $G_0$; it has 64
elements. Now we consider the map $\psi : \G_0\to G_0$ (see Section \ref{sec:prelim}) and observe that $\psi(Z) = H\cap G_0$. For $g\in H\cap G_0$ we compute a preimage $\psi^{-1}(g)$
in the following way. By the algorithm of \cite{cmt} (see also \cite[\S 5.7]{gra16}) we can write $g=\exp( \ad z_1 )\cdot \ldots \cdot\exp(\ad z_s)$ where $z_1,\dots,z_j\in\mathfrak{g}_0$ are 
nilpotent. Let $\hat z_i\in \hat\g_0$ be such that $\psi(\hat z_i) = z_i$. Then $\exp(\hat z_1)\cdot\ldots\cdot \exp(\hat z_s)$ is a preimage of $g$. So we get 64 preimages. To this set we add the
kernel of $\psi$ which is of order 4 and generated by the $8\times 8$ diagonal matrix with $i$ on the diagonal. As a result we find the group $Z$ of order $2^8$.

Thanks to the results described in Section \ref{sec1}, we were able to develop an algorithm for computing $H^1N$, which we summarize in the following steps:
\begin{enumerate}
    \item Compute $H^1 W$. Since the complex conjugation acts trivially on $W$, the cohomology classes coincide with the conjugacy classes of elements of order 2 in $W$. These 
    are easily calculated by {\sf GAP}. In particular, it turns out that $H^1 W$ has exactly ten elements, and we denote fixed representatives of the classes in $H^1 W$ by 
    $\pi_{i}$ for $1\le i\le 10$.
    \item Determine cocycles $n_i\in N$ such that $j(n_i)=\pi_i$; note that $j^{-1}(\pi_i)$ has $2^8$ elements, and so for each $j^{-1}(\pi_i)$ we easily find such a cocycle; in particular, we can pick cocycles represented by diagonal matrices. 
    \item For each $i$ define $\tau_i: N\mapsto N$ with $\tau_i(n)=n_i\bar n n_i^{-1}$. This clearly induces a map on $W$ which we denote again by $\tau_i$, by abuse of notation, such that $\tau_i(w)=\pi_iw\pi_i^{-1}$.
    \item By direct brute force computation we compute $H^1(Z,\tau_i)$ and $W^{\tau_i}=\{w\in W\mid \tau_i(w)=w\}$.
    \item Observe that $W^{\tau_i}$ acts on $H^1(Z,\tau_i)$ as described in (\ref{eq1}) and, in particular, it turns out (again by direct computation) that the action is transitive. This implies, 
    thanks to Proposition \ref{prop1}, that $j_*^{-1}([\pi_i])=\{[n_i]\}$.
    \item We conclude that $H^1 N = \{ [n_1],\ldots, [n_{10}] \} $.
\end{enumerate}

Proceeding as in Section \ref{sec2}, for each $i$ we compute $g_i\in \G_0$ with $g_i^{-1} \bar g_i = n_i$. Then $g_i\cdot \h$ is the Cartan subspace corresponding to $[n_i]$.
It follows that $\g_1^\R$ has ten Cartan subspaces up to conjugacy. 

We consider computing the action of the map $\nu$ on the set of Cartan subspaces. 
It turns out that the $n_i$ and $g_i$ can be chosen so that they are diagonal matrices.
Hence $\nu(g_i\cdot \h)= g_0g_ig_0^{-1} g_1^{-1}g_1 g_0\cdot \h = g_i g_1^{-1}\cdot 
i\h = g_ig_1^{-1}\cdot \h$ (as $i\h = \h$). Write $h_i = g_ig_1^{-1}$, then, as all involved
matrices commute, $h_i^{-1} \bar h_i = g_i^{-1}\bar g_i g_1 \bar g_1^{-1}=n_ig_1\bar g_1^{-1}$.
But $g_1\bar g_1^{-1} = \diag(\omega^2,\ldots,\omega^2)$ (notation as in Section \ref{sec:outer}).
Hence $g_1\bar g_1^{-1} \cdot u = -u$ for all $u\in \bigwedge^4 \C^8$. Now the longest element
$w_0$ of the Weyl group acts as $-1$ on $\h$. It follows that the restriction of $h_1^{-1}
\bar h_i$ to $\h$ is equal to $\pi_i w_0$. So the cocycle corresponding to $\nu(g_i\cdot \h)$
lies in the conjugacy class of $\pi_iw_0$. A computation shows that the conjugacy classes of 
$\pi_i$ and $\pi_iw_0$ are always different. Hence up to the action of $\nu$ only 5 Cartan subspaces remain.
Their bases are as follows:

\begin{enumerate} 
\item $e_{1234}+e_{5678}, e_{1357}+e_{2468}, e_{1256}+e_{3478}, e_{1368}+e_{2457}, e_{1458}+e_{2367}, e_{1467}+e_{2358}, e_{1278}+e_{3456}$
\item $e_{1257}-4e_{3468},e_{1358}-4e_{2467},e_{1268}-4e_{3457},e_{1456}-4e_{2378},e_{1478}-4e_{2356},e_{1367}-4e_{2458},e_{1234}+16e_{5678}$
\item $e_{1234}+e_{5678},e_{1257}-e_{3468},e_{1356}+e_{2478},e_{1268}-e_{3457},e_{1458}-e_{2367},e_{1467}-e_{2358},e_{1378}+e_{2456}$
\item $e_{1256}+128e_{3478},e_{1368}+128e_{2457},e_{1357}+128e_{2468},e_{1467}+128e_{2358},e_{1234}-128e_{5678},e_{1278}+128e_{3456},e_{1458}+128e_{2367}$
\item $e_{1256}+128e_{3478},e_{1678}-192e_{2345},e_{1458}+128e_{2367},e_{1346}+192e_{2578},e_{1238}+192e_{4567},e_{1247}+192e_{3568},e_{1357}+128e_{2468}$
\end{enumerate}

\section*{Appendix}\label{sec:app}

In this appendix we describe the classes $\h_H^\circ$ of semisimple orbits. Excluding the class consisting only of
0, there are 31 such classes, numbered $2,\ldots,32$. In each case we use the basis $u_1,\ldots,u_k$
given in Table \ref{tab1}. In each case $\h_H^\circ$ consists of $\lambda_1u_1+\cdots +\lambda_k
u_k$ where the coefficients $\lambda_i$ are such that certain polynomials in the $\lambda_i$ are nonzero.
We give these polynomials, along with generators and the size of the group $\Gamma_H$. 

The classes 2-9 are all 1-dimensional. Let $q$ be a basis element. Then in each case
$\h_H^\circ$ consists of $\mu q$ with $\mu\neq 0$ and $\Gamma_H=\{1,-1\}$.

\begin{enumerate}
\item[Class 10] Generators of $\Gamma_H$:
$$\begin{pmatrix} -1 & 0\\ 0&1\end{pmatrix},\, \begin{pmatrix}0&1\\-1&0\end{pmatrix}. $$
Order of $\Gamma_H$: 8. Polynomial conditions:
$$\lambda_1\lambda_2(\lambda_1^2-\lambda_2^2)\neq 0.$$
\item[Class 11] Generators of $\Gamma_H$:
$$\begin{pmatrix} -1 & 0\\ 0&1\end{pmatrix},\, \begin{pmatrix}1&0\\0&-1\end{pmatrix}. $$
Order of $\Gamma_H$: 4. Polynomial conditions:
$$\lambda_1\lambda_2(\lambda_1^2-\lambda_2^2)\neq 0.$$
\item[Class 12] Generators of $\Gamma_H$:
$$\begin{pmatrix} -1 & 0\\ 0&1\end{pmatrix},\, \tfrac{1}{2}\begin{pmatrix}-1&-3
\\-1&1\end{pmatrix}. $$
Order of $\Gamma_H$: 12. Polynomial conditions:
$$\lambda_1\lambda_2(\lambda_1^2-\lambda_2^2)(\lambda_1^2-9\lambda_2)^2\neq 0.$$
\item[Class 13] Generators of $\Gamma_H$:
$$\begin{pmatrix} -1 & 0\\ 0&-1\end{pmatrix}. $$
Order of $\Gamma_H$: 2. Polynomial conditions:
$$\lambda_1\lambda_2(\lambda_1^2-4\lambda_2)^2\neq 0.$$
\item[Class 14] Generators of $\Gamma_H$:
$$\begin{pmatrix} 1 & 0\\ 3&-1\end{pmatrix},\, \begin{pmatrix}-1&0
\\-3&1\end{pmatrix}. $$
Order of $\Gamma_H$: 4. Polynomial conditions:
$$\lambda_1\lambda_2(3\lambda_1-2\lambda_2)(\lambda_1-\lambda_2)(2\lambda_1-\lambda_2)(3\lambda_1-\lambda_2)\neq 0.$$
\item[Class 15] Generators of $\Gamma_H$:
$$\begin{pmatrix} -1 & 0\\ 0&1\end{pmatrix},\, \begin{pmatrix}1&0\\0&-1\end{pmatrix}. $$
Order of $\Gamma_H$: 4. Polynomial conditions:
$$\lambda_1\lambda_2(\lambda_1^2-4\lambda_2^2)(\lambda_1^2-16\lambda_2^2)\neq 0.$$
\item[Class 16] Generators of $\Gamma_H$:
$$\begin{pmatrix} -1 & 0\\ 0&-1\end{pmatrix},\, \begin{pmatrix}1&0\\2&-1\end{pmatrix}. $$
Order of $\Gamma_H$: 4. Polynomial conditions:
$$\lambda_1\lambda_2(2\lambda_1-\lambda_2)(3\lambda_1-2\lambda_2)(\lambda_1-2\lambda_2)(\lambda_1-\lambda_2)\neq 0.$$
\item[Class 17] Generators of $\Gamma_H$:
$$\begin{pmatrix} 0 & 1\\ 1&0\end{pmatrix},\, \begin{pmatrix}0&1\\-1&1\end{pmatrix}. $$
Order of $\Gamma_H$: 12. Polynomial conditions:
$$\lambda_1\lambda_2(\lambda_1^2-\lambda_2^2)(\lambda_1-2\lambda_2)(2\lambda_1-\lambda_2)\neq 0.$$
\item[Class 18] Generators of $\Gamma_H$:
$$\begin{pmatrix} 0 & 0 & -1\\ 1&0&0 \\ 0&1&0\end{pmatrix},\, \begin{pmatrix}0&-1&0
\\1&0&0\\ 0&0&-1\end{pmatrix}. $$
Order of $\Gamma_H$: 48. Polynomial conditions:
$$\lambda_1\lambda_2\lambda_3(\lambda_1^2-\lambda_2^2)(\lambda_1^2-\lambda_3^2)(\lambda_2^2-
\lambda_3^2)\neq 0.$$
\item[Class 19] Generators of $\Gamma_H$:
$$\begin{pmatrix} -1 & 0 & 0\\ 0&1&0 \\ 0&0&-1\end{pmatrix},\, \tfrac{1}{2}\begin{pmatrix}-1&-1&-2
\\1&1&-2\\ -1&1&0\end{pmatrix}. $$
Order of $\Gamma_H$: 12. Polynomial conditions:
$$\lambda_1\lambda_2\lambda_3(\lambda_1^2-\lambda_2^2)(\lambda_1+\lambda_3)\neq 0,
\lambda_1+\epsilon\lambda_2+2\delta\lambda_3\neq 0 \text{ where } \epsilon,\delta \in \{1,-1\}.$$
\item[Class 20] Generators of $\Gamma_H$:
$$\begin{pmatrix} 0 & -1 & 0\\ -1&0&0 \\ -1&-1&1\end{pmatrix},\,\begin{pmatrix}
0&0&1\\1&1&-1\\ 1&0&0\end{pmatrix},\, \begin{pmatrix} 1&1&-1\\0&0&1\\ 0&1&0
\end{pmatrix}.$$
Order of $\Gamma_H$: 8. Polynomial conditions:
\begin{align*}
&\lambda_1\lambda_2\lambda_3(\lambda_2-\lambda_3)(\lambda_1^2-\lambda_3^2)(\lambda_1+\lambda_2)\neq 0,\\
&(\lambda_1+\lambda_2-\lambda_3)(3\lambda_1+2\lambda_2-\lambda_3)\neq 0, \,
\lambda_3 \neq \pm (\lambda_1+2\lambda_2).
\end{align*}
\item[Class 21] Generators of $\Gamma_H$:
$$\begin{pmatrix} -1 & 0 & 0\\ 0&1&0 \\ 0&0&1\end{pmatrix},\,\begin{pmatrix}
1&0&0\\0&0&1\\ 0&1&0\end{pmatrix},\, \tfrac{1}{2}\begin{pmatrix} 1&3&-1\\1&-1&1\\ 0&0&1
\end{pmatrix}.$$
Order of $\Gamma_H$: 48. Polynomial conditions:
$$
\lambda_1\lambda_2\lambda_3(\lambda_2^2-\lambda_3^2)\neq 0,\,
\lambda_1-\delta(3\lambda_2-\lambda_3)\neq 0,\,
\lambda_1-\delta(\lambda_2-3\lambda_3)\neq 0,\,
\lambda_1+\delta\lambda_2+\epsilon\lambda_3\neq 0,$$
where $\delta,\epsilon\in \{1,-1\}$.
\item[Class 22] Generators of $\Gamma_H$:
$$\begin{pmatrix} 1 & 0 & 0\\ 0&1&0 \\ 0&0&-1\end{pmatrix},\,\begin{pmatrix}
1&0&0\\2&-1&0\\ 0&0&1\end{pmatrix},\, \tfrac{1}{2}\begin{pmatrix} 1&0&1\\-1&2&1\\ 3&0&-1
\end{pmatrix}.$$
Order of $\Gamma_H$: 24. Polynomial conditions:
$$
\lambda_1\lambda_2\lambda_3(\lambda_1^2-\lambda_3^2)(9\lambda_1^2-\lambda_3^2)
(\lambda_1-\lambda_2)(2\lambda_1-\lambda_2)\neq 0,\,
\lambda_1-2\lambda_2+\delta\lambda_3\neq 0,\,
3\lambda_1-2\lambda_2+\delta\lambda_3\neq 0,$$
where $\delta\in \{1,-1\}$.
\item[Class 23] Generators of $\Gamma_H$:
$$\begin{pmatrix} 1 & 0 & 0\\ 0&1&0 \\ 0&0&-1\end{pmatrix},\,\begin{pmatrix}
1&-1&0\\0&-1&0\\ 0&0&1\end{pmatrix},\, \begin{pmatrix} -1&1&0\\0&1&0\\ 0&0&1
\end{pmatrix}.$$
Order of $\Gamma_H$: 8. Polynomial conditions:
$$
\lambda_1\lambda_2\lambda_3(\lambda_1^2-\lambda_3^2)
(\lambda_1-\lambda_2)(2\lambda_1-\lambda_2)\neq 0,\,
\lambda_1-2\lambda_2+\delta\lambda_3\neq 0,\,
\lambda_1+\delta \lambda_2+\epsilon \lambda_3\neq 0,$$
where $\epsilon,\delta\in \{1,-1\}$.
\item[Class 24] Generators of $\Gamma_H$:
$$\begin{pmatrix} 1 & 0 & 0\\ 0&1&0 \\ 0&0&-1\end{pmatrix},\,\begin{pmatrix}
1&0&0\\0&0&1\\ 0&1&0\end{pmatrix},\, \begin{pmatrix} 0&1&0\\1&0&0\\ 0&0&1
\end{pmatrix}.$$
Order of $\Gamma_H$: 48. Polynomial conditions:
$$
\lambda_1\lambda_2\lambda_3(\lambda_1^2-\lambda_2^2)
(\lambda_1^2-\lambda_3^2)(\lambda_2^2-\lambda_3^2)\neq 0,\,
\lambda_1+\delta \lambda_2+\epsilon \lambda_3\neq 0,$$
where $\epsilon,\delta\in \{1,-1\}$.
\item[Class 25] Generators of $\Gamma_H$:
$$\begin{pmatrix} 0&0&0&-1\\ 0&1&0&0\\ -1&0&1&-1\\ -1&0&0&0\end{pmatrix},\,
\begin{pmatrix} 0&0&1&0\\ 0&1&0&0\\ 1&0&0&0\\ 1&0&-1&1\end{pmatrix},\,
\tfrac{1}{2}\begin{pmatrix} 1&-1&-1&2\\ -1&1&-1&2\\ 0&0&2&0\\ 1&1&1&0\end{pmatrix},\,
\tfrac{1}{2}\begin{pmatrix} 1&-1&1&0\\-2&0&2&0\\1&1&1&0\\1&1&-1&2\end{pmatrix}.$$
Order of $\Gamma_H$: 96. Polynomial conditions:
\begin{align*}
&\lambda_1\lambda_2\lambda_3\lambda_4(\lambda_1-\lambda_3)
(\lambda_1+\lambda_4)(\lambda_3-\lambda_4)(\lambda_1-\lambda_3+\lambda_4)\neq 0,\,\\
& \lambda_1+\delta \lambda_2+\epsilon \lambda_3\neq 0,\,
\lambda_1+\delta\lambda_2-3\lambda_3+2\lambda_4\neq 0,\,
\lambda_1+\delta \lambda_2 +\epsilon (\lambda_3-2\lambda_4)\neq 0,
\end{align*}
where $\epsilon,\delta\in \{1,-1\}$.
\item[Class 26] Generators of $\Gamma_H$:
$$\tfrac{1}{2}
\begin{pmatrix}1&-1&-1&-1 \\-1&1&-1&-1\\-1&-1&1&-1\\-1&-1&-1&1\end{pmatrix},\,
\tfrac{1}{2}
\begin{pmatrix}-1&1&1&-1\\1&-1&1&-1\\1&1&-1&-1\\-1&-1&-1&-1\end{pmatrix},\,
\begin{pmatrix}0&1&0&0 \\ 1&0&0&0 \\ 0&0&0&1 \\ 0&0&1&0\end{pmatrix}.$$
Order of $\Gamma_H$: 48. Polynomial conditions:
\begin{align*}
&\lambda_1\lambda_2\lambda_3\lambda_4(\lambda_1-\lambda_2)
(\lambda_2^2-\lambda_4^2)(\lambda_1^2-\lambda_3^2)\neq 0,\,\\
& \lambda_1+\delta \lambda_2+\epsilon \lambda_3+\nu\lambda_4\neq 0,\,
\lambda_1-2\lambda_2+\delta \lambda_3\neq 0,\,
2\lambda_1-\lambda_2 +\delta\lambda_4\neq 0,
\end{align*}
where $\epsilon,\delta,\nu\in \{1,-1\}$.
\item[Class 27] Generators of $\Gamma_H$:
$$\tfrac{1}{2}
\begin{pmatrix} -1&-1&-1&-1 \\ -1&1&-1&1 \\
      1&-1&-1&1 \\ 1&1&-1&-1 \end{pmatrix},\,
\tfrac{1}{2}
\begin{pmatrix} -1&1&1&1 \\ 1&-1&1&1 \\ 1&1&1&-1 \\-1&-1&1&-1\end{pmatrix}.$$
Order of $\Gamma_H$: 1152. Polynomial conditions:
\begin{align*}
&\lambda_1\lambda_2\lambda_3\lambda_4(\lambda_1^2-\lambda_2^2)
(\lambda_1^2-\lambda_3^2)(\lambda_1^2-\lambda_4^2)(\lambda_2^2-\lambda_3^2)
(\lambda_2^2-\lambda_4^2)(\lambda_3^2-\lambda_4^2)\neq 0,\,\\
& \lambda_1+\delta \lambda_2+\epsilon \lambda_3+\nu\lambda_4\neq 0,
\end{align*}
where $\epsilon,\delta,\nu\in \{1,-1\}$.
\item[Class 28] Generators of $\Gamma_H$:
$$\tfrac{1}{2}
\begin{pmatrix}
-2&-1&1&1\\ 0&0&-2&-2\\ -2&1&1&1\\0&-1&1&-1 \end{pmatrix},\,
\begin{pmatrix}0&0&0&-1\\0&-1&0&0\\1&0&-1&-1\\-1&0&0&0\end{pmatrix},\,
\tfrac{1}{2}\begin{pmatrix}
1&1&-1&-2\\2&0&-2&0\\1&-1&-1&-2\\-1&1&-1&0\end{pmatrix}.$$
Order of $\Gamma_H$: 96. Polynomial conditions:
\begin{align*}
&\lambda_1\lambda_2\lambda_3\lambda_4(\lambda_1-\lambda_3)
(\lambda_1^2-\lambda_4^2)(\lambda_3+\lambda_4)(\lambda_1-2\lambda_3-\lambda_4)
(\lambda_1-\lambda_3-\lambda_4)\neq 0,\,\\
& \lambda_2+\delta\lambda_3+\epsilon\lambda_4\neq 0,\, \lambda_1+\delta\lambda_2
+\epsilon\lambda_3\neq 0,\, \lambda_1+\delta\lambda_2-\lambda_3-2\lambda_4
\neq 0,\, 2\lambda_1+\delta\lambda_2-\lambda_3-\lambda_4\neq 0,
\end{align*}
where $\epsilon,\delta\in \{1,-1\}$.
\item[Class 29]
Generators of $\Gamma_H$:
$$\tfrac{1}{2}\begin{pmatrix}
 1 & 0 & 1 & 1 & 0\\2 & 0 & 0 & 0 & 0\\0 & 0 & 0 & 0 & 2\\ 
 1 & -2 & 1 & -1 & 0\\-1 & 2 & 1 & -1 & 0
\end{pmatrix},\,
\tfrac{1}{2}\begin{pmatrix}
  1 & 0 & 1 & -1 & 0\\0 & 1 & 0 & -1 & 1 \\
  2 & -1 & 0 & 1 & 1\\-1 & -1 & 1 & 0 & 1 \\
  1 & -2 & -1 & -1 & 0
\end{pmatrix},\,
\tfrac{1}{2}\begin{pmatrix}
  -1 & 0 & 1 & 1 & 0\\-1 & 1 & 1 & 0 & -1\\
  -1 & -1 & -1 & 0 & -1\\-2 & 1 & 0 & -1 & 1\\
     -1 & 2 & -1 & 1 & 0 
\end{pmatrix}.$$
Order of $\Gamma_H$: 1440. Polynomial conditions:
\begin{align*}
&\lambda_1\lambda_2\lambda_3\lambda_4\lambda_5(\lambda_1-\lambda_2)\neq 0,\,\\
& \lambda_2+\delta\lambda_4+\epsilon\lambda_5\neq 0,\, \lambda_1+\delta\lambda_3
+\epsilon\lambda_4\neq 0,\, \lambda_1-2\lambda_2+\delta\lambda_3+\epsilon \lambda_4\neq 0,\\
&\lambda_1+\delta\lambda_2+\epsilon \lambda_3+\nu\lambda_5\neq 0,\,
2\lambda_1-\lambda_2+\delta\lambda_4+\epsilon\lambda_5\neq 0,
\end{align*}
where $\epsilon,\delta,\nu\in \{1,-1\}$.
\item[Class 30]
Generators of $\Gamma_H$:
$$
\tfrac{1}{2}\begin{pmatrix}
-2 & 0 & 0 & 0 & 0\\0 & -1 & 1 & -1 & -1 \\ 
0 & 1 & 1 & -1 & 1\\0 & 1 & 1 & 1 & -1 \\ 0 & -1 & 1 & 1 & 1
\end{pmatrix},\,
\tfrac{1}{2}\begin{pmatrix}
-2 & 0 & 0 & 0 & 0\\0 & 1 & -1 & 1 & 1\\
0 & 1 & -1 & -1 & -1\\0 & -1 & -1 & -1 & 1\\0 & 1 & 1 & -1 & 1 
\end{pmatrix},\,
\tfrac{1}{2}\begin{pmatrix}
2 & 0 & 0 & 0 & 0\\0 & 1 & 1 & 1 & -1\\0 & 1 & 1 & -1 & 1 \\
0 & 1 & -1 & 1 & 1\\0 & -1 & 1 & 1 & 1 
\end{pmatrix}.$$
Order of $\Gamma_H$: 768. Polynomial conditions:
\begin{align*}
&\lambda_1\lambda_2\lambda_3\lambda_4\lambda_5(\lambda_2^2-\lambda_4^2)
(\lambda_3^2-\lambda_5^2)\neq 0,\,\\
& \lambda_2+\delta\lambda_3+\epsilon\lambda_4+\nu\lambda_5\neq 0,\,
\lambda_1+\delta\lambda_4+\epsilon\lambda_5\neq 0,\, \lambda_1+\delta\lambda_3
+\epsilon\lambda_4\neq 0,\\
& \lambda_1+\delta\lambda_2+\epsilon \lambda_5\neq 0,\,
\lambda_1+\delta\lambda_2+\epsilon \lambda_3\neq 0,
\end{align*}
where $\epsilon,\delta,\nu\in \{1,-1\}$.
\item[Class 31]
Generators of $\Gamma_H$:
$$
\begin{pmatrix}
0 & 0 & 1 & 0 & 0 & 0\\0 & 0 & 0 & 0 & 1 & 0\\1 & 0 & 0 & 0 & 0 & 0 \\
0 & 0 & 0 & 0 & 0 & -1\\0 & 0 & 0 & 1 & 0 & 0\\0 & 1 & 0 & 0 & 0 & 0
\end{pmatrix},\,
\tfrac{1}{2}\begin{pmatrix}
0 & 1 & 0 & 1 & 1 & -1\\1 & 1 & 1 & -1 & 0 & 0 \\
0 & -1 & 0 & -1 & 1 & -1\\-1 & 1 & -1 & -1 & 0 & 0\\
-1 & 0 & 1 & 0 & -1 & -1\\-1 & 0 & 1 & 0 & 1 & 1 
\end{pmatrix}
$$
Order of $\Gamma_H$: 23040. Polynomial conditions:
\begin{align*}
&\lambda_1\lambda_2\lambda_3\lambda_4\lambda_5\lambda_6\neq 0,\,\\
&\lambda_i+\delta\lambda_j+\epsilon\lambda_k\neq 0,\,\,(i,j,k)\in \{(1,2,5),
(1,4,6),(2,3,6),(3,4,5)\},\\
& \lambda_i+\delta\lambda_j+\epsilon\lambda_k+\nu\lambda_l\neq 0,\,\,(i,j,k,l)
\in \{((1,2,3,4),(1,3,5,6),(2,4,5,6)\},
\end{align*}
where $\epsilon,\delta,\nu\in \{1,-1\}$.
\item[Class 32]
Generators of $\Gamma_H$:
$$
\tfrac{1}{2}\begin{pmatrix}
1 & 1 & -1 & 0 & 0 & -1 & 0\\1 & 0 & 0 & 1 & 0 & 1 & -1 \\ 
1 & -1 & 0 & 0 & -1 & 0 & 1\\0 & 1 & 1 & 1 & 0 & 0 & 1 \\ 
0 & -1 & 0 & 1 & 1 & -1 & 0\\1 & 0 & 1 & -1 & 1 & 0 & 0 \\ 
0 & 0 & 1 & 0 & -1 & -1 & -1
\end{pmatrix},\,
\tfrac{1}{2}\begin{pmatrix}
0 & 1 & -1 & 1 & 0 & 0 & -1\\1 & 0 & -1 & -1 & -1 & 0 & 0 \\ 
0 & -1 & 0 & 1 & -1 & 1 & 0\\1 & -1 & 0 & 0 & 1 & 0 & -1 \\ 
1 & 1 & 1 & 0 & 0 & 1 & 0\\1 & 0 & 0 & 1 & 0 & -1 & 1 \\ 
0 & 0 & 1 & 0 & -1 & -1 & -1 
\end{pmatrix}
$$
Order of $\Gamma_H$: 2903040. Polynomial conditions:
\begin{align*}
&\lambda_1\lambda_2\lambda_3\lambda_4\lambda_5\lambda_6\lambda_7\neq 0,\,\\
& \lambda_i+\delta\lambda_j+\epsilon\lambda_k+\nu\lambda_l\neq 0,\,\,(i,j,k,l)\in \{((1,2,3,6),
(1,2,5,7),\\
& (1,3,4,5),(1,4,6,7),(2,3,4,7),(2,4,5,6),
		(3,5,6,7)\},
\end{align*}
where $\epsilon,\delta,\nu\in \{1,-1\}$.
\end{enumerate}

\begin{thebibliography}{10}

\bibitem{ACD}
Dmitri V. Alekseevsky, Vicente Cortés and Chandrashekar Devchand.
 Yang-Mills connections over manifolds with Grassmann structure
{\em J. Math. Phys.} 44 no. 12, 6047--6076, 2003.

\bibitem{Antorig} L.V. Antonyan. The classification of four-vectors of the eight-dimensional space. Trudy Seminara po Vektornomu i Tenzornomu Analizu. 20, 1981.  
\bibitem{berhuy} G.\ Berhuy.
An Introduction to Galois Cohomology and its Applications. 
London Mathematical Society Lecture Note Series. Cambridge: Cambridge University Press, 2010.

\bibitem{berg}
M. Berger. Sur les groupes d'holonomie homogènes des variétés a connexion affines et des variétés riemanniennes, {\em Bull. Soc. Math. France}, 83, 279--330, 1955.

\bibitem{borel}
Armand Borel.
\newblock {\em Linear algebraic groups}.
\newblock Springer-Verlag, Berlin, Heidelberg, New York, second edition, 1991.

\bibitem{bor21}
Mikhail Borovoi. Real points in a homogeneous space of a real algebraic group.
{\tt  	arXiv:2106.14871 [math.AG]}, 2021.

\bibitem{borwdg}
Mikhail Borovoi and Willem A.~de Graaf.
\newblock Computing {G}alois cohomology of a real linear algebraic group.
{\em J. Lond. Math. Soc.} (2) 109 no. 5, Paper No. e12906, 2024.

\bibitem{bgl}
Mikhail Borovoi, Willem A.~de Graaf and H\^ong V\^an L\^e. 
Classification of real trivectors in dimension nine.
{\em J. Algebra} 603 118-163, 2022.

\bibitem{magma}
Wieb Bosma, John Cannon, and Catherine Playoust.
\newblock The {M}agma algebra system. {I}. {T}he user language.
\newblock {\em J. Symbolic Comput.}, 24(3-4):235--265, 1997.
\newblock Computational algebra and number theory (London, 1993).

\bibitem{bou3}
N.~Bourbaki.
\newblock {\em {Groupes et Alg\`{e}bres de Lie, Chapitres VII et VIII}}.
\newblock Hermann, Paris, 1975.

\bibitem{bh89}
R.L. Bryant and R. Harvey. Submanifolds in hyper-Kähler geometry. {\em J. Amer. Math. Soc.} 2 
1--31, 1989.

\bibitem{cmt}
Arjeh M. Cohen, Scott H. Murray and Donald E. Taylor.
Computing in groups of Lie type.
{\em Math. Comp.} 73, no. 247, 1477–1498, 2004.

\bibitem{cms}
D. Conti, T.B. Madsen and S. Salamon. Quaternionic Geometry in Dimension 8. In {\em Geometry and Physics: A Festschrift in honour of Nigel Hitchin} Vol. 1 (eds: A. Dancer, J.E. Andersen and  O. García-Prada), Oxford University Press, 91--113, 2018.

\bibitem{clo}
David~A. Cox, John Little, and Donal O'Shea.
\newblock {\em Ideals, varieties, and algorithms}.
\newblock Undergraduate Texts in Mathematics. Springer, Cham, fourth edition,
  2015.
\newblock An introduction to computational algebraic geometry and commutative
algebra.

\bibitem{DGS2}
Emanuele Di Bella, Willem A. de Graaf and Andrea Santi.
Some rigidity results for supergravity backgrounds in 11 dimensions, {\em preprint}.

\bibitem{dgmo22}  H. Dietrich, Willem A. de Graaf, A. Marrani, M. Origila. 
Classification of four qubit states and their stabilisers under SLOCC operations
{\it J. Phys. A} 55 no. 9, Paper No. 095302, (2022).

\bibitem{gap4}
  The GAP~Group, \emph{GAP -- Groups, Algorithms, and Programming, 
  Version 4.14.0}; 
  2024,
  {\tt https://www.gap-system.org}.


\bibitem{gra16}
Willem A.~de Graaf.
\newblock {\em Computation with linear algebraic groups}.
\newblock Monographs and Research Notes in Mathematics. CRC Press, Boca Raton,
  FL, 2017.

\bibitem{zarclos}
Willem A.~de Graaf.
Computing the {Z}ariski closure of a finitely generated rational matrix group.
{\tt https://arxiv.org/pdf/2505.02577}, 2025.

\bibitem{gl24} Willem A. de Graaf, H\^ong V\^an L\^e.
Semisimple elements and the little Weyl group of real semisimple $\Z_m$-graded Lie algebras
{\it Linear Algebra Appl.} 703 423--445 (2024).
  
\bibitem{hum}
James~E. Humphreys.
\newblock {\em Introduction to {L}ie algebras and representation theory},
  volume~9 of {\em Graduate Texts in Mathematics}.
\newblock Springer-Verlag, New York-Berlin, 1978.
\newblock Second printing, revised.

\bibitem{jac}
Nathan Jacobson.
\newblock {\em Lie algebras}.
\newblock Interscience Tracts in Pure and Applied Mathematics, No. 10.
  Interscience Publishers (a division of John Wiley \& Sons), New York-London,
  1962.

\bibitem{kac}
V.~G. Kac.
\newblock {\em Infinite Dimensional Lie Algebras}.
\newblock Cambridge University Press, Cambridge, third edition, 1990.  

\bibitem{Antotrad}
L. Oeding.
A translation of ``Classification of four-vectors of an 8-dimensional space'', by Antonyan, L. V., with an appendix by the translator,
{\it Trans. Moscow Math. Soc.} 83, 227-250 (2022).

\bibitem{popov}
A. M. Popov, 
Finite isotropy subgroups in general position of simple linear Lie groups. (English. Russian original) 
{\em Trans. Mosc. Math. Soc.} 1986, 3-63 (1986); translation from Tr. Mosk. Mat. O.-va. 48, 7-59 (1985).

\bibitem{salamon}
S. Salamon. {\em Riemannian Geometry and Holonomy Groups.} Longman Scientific and Technical, Harlow, Essex, UK, 1989.

\bibitem{serre}
J.-P. Serre. {\em Galois cohomology}, Springer-Verlag, Berlin, 1997.

\bibitem{steinberg}
R.~Steinberg.
\newblock {\em Lectures on {C}hevalley groups}.
\newblock Yale University, New Haven, 1967.
\newblock Notes prepared by John Faulkner and Robert Wilson.

\bibitem{steinberg75}
Robert Steinberg.
\newblock Torsion in reductive groups.
\newblock {\em Advances in Math.}, 15:63--92, 1975.


\bibitem{vinberg}
{\`E}.~B. Vinberg. The {W}eyl group of a graded {L}ie algebra. {\em Izv. Akad. Nauk SSSR Ser. Mat.}, 40(3):488--526, 1976.
English translation: Math. USSR-Izv. 10, 463--495 (1976).

\bibitem{vinberg2}
{\`E}.~B. Vinberg.
\newblock Classification of homogeneous nilpotent elements of a semisimple
  graded {L}ie algebra.
\newblock {\em Trudy Sem. Vektor. Tenzor. Anal.}, (19):155--177, 1979.
\newblock English translation: Selecta Math. Sov. 6, 15-35 (1987).


\bibitem{elashvin}
{\`E}.~B. Vinberg and A.~G. Elashvili. A classification of the three-vectors of nine-dimensional space.
{\em Trudy Sem. Vektor. Tenzor. Anal.}, 18:197--233, 1978. English translation: Selecta Math. Sov., 7, 63-98, (1988).


\end{thebibliography}
\end{document}